\documentclass[a4paper,11pt]{amsart}
\usepackage{amsmath}
\usepackage{amsthm,amsfonts,amssymb,graphics,graphicx}
\usepackage[foot]{amsaddr}
\usepackage{hyperref}
\usepackage[normalem]{ulem}
\usepackage{enumerate}
\usepackage{comment}
\usepackage{geometry}
\usepackage{dsfont}
\usepackage{centernot}
\usepackage{bbm}
\usepackage{subcaption}
\usepackage[style=numeric,
            giveninits=true,
            maxnames=4]{biblatex}
\AtEveryBibitem{%
  \clearfield{volume}%
  \clearfield{number}%
  \clearfield{pages}%
  \clearfield{doi}%
  \clearfield{isbn}%
  \clearfield{issn}%
  \clearfield{url}%
  \clearfield{note}%
  \clearfield{edition}%
  \clearlist{address}%
  \clearlist{location}%
  \clearlist{publisher}% 
  \clearfield{series}%
}

\DeclareBibliographyDriver{article}{%
  \printnames{author}%
  \setunit{\labelnamepunct}\newblock
  \printfield[title]{title}%
  \newunit\newblock
  \printfield{journaltitle}%
  \setunit{\addcomma\space}%
  \printfield{year}%
  \finentry}

\DeclareBibliographyDriver{book}{%
  \printnames{author}%
  \setunit{\labelnamepunct}\newblock
  \printfield[title]{title}%
  \newunit\newblock
  \printlist{publisher}%
  \setunit{\addcomma\space}%
  \printdate%
  \finentry}

\DeclareSortingScheme{nameyearmonthtitle}{
  \sort{
    \name{author}
  }
  \sort{
    \field{year}
  }
  \sort{
    \field{month}
  }
  \sort{
    \field{title}
  }
}

\usepackage{tikz,tikz-3dplot}
\tdplotsetmaincoords{80}{110}
\usepackage{pgfplots}
\pgfplotsset{compat=1.17}
\usetikzlibrary{intersections}
\usepgfplotslibrary{fillbetween}

\usetikzlibrary{decorations.pathmorphing,patterns,fpu,calc}
\usetikzlibrary{shapes.misc}
\makeatletter
\tikzset{use fpu reciprocal/.code={%
\def\pgfmathreciprocal@##1{%
    \begingroup
    \pgfkeys{/pgf/fpu=true,/pgf/fpu/output format=fixed}%
    \pgfmathparse{1/##1}%
    \pgfmath@smuggleone\pgfmathresult
    \endgroup
}}}%

\newtheorem{theorem}{Theorem}[section]
\newtheorem{lemma}[theorem]{Lemma}
\newtheorem{proposition}[theorem]{Proposition}

\theoremstyle{remark}
\newtheorem{remark}[theorem]{Remark}

\newtheorem{definition}[theorem]{Definition}
\newtheorem{algorithm}[theorem]{Algorithm}

\numberwithin{equation}{section}

\newcommand {\R} {\mathbb{R}}
\renewcommand {\S} {\mathbb{S}}
\renewcommand {\mod} {\mathrm{mod}}
\renewcommand {\Im} {\mathrm{Im}\:}
\renewcommand {\Re} {\mathrm{Re}\:}
\newcommand{\T}{\mathbb{T}}
\newcommand {\D} {\mathbb{D}}
\newcommand {\E} {\mathbb{E}}
\renewcommand {\H} {\mathbb{H}}
\newcommand {\C} {\mathbb{C}}
\newcommand {\N} {\mathbb{N}}
\newcommand {\Z} {\mathbb{Z}}
\renewcommand{\P} {\mathbb{P}}
\newcommand {\Var} {\mathrm{Var}}
\newcommand {\diam} {\mathrm{diam}}

\newcommand{\ind}{\mathds{1}}

\begin{document}
\title[Conformal welding of independent GMC measures]{Conformal welding of independent Gaussian multiplicative chaos measures} 
\author{Antti Kupiainen\textsuperscript{1}}
\address{\textsuperscript{1}Department of Mathematics and Statistics, University of Helsinki}
\address{\textsuperscript{2}Current address: School of Mathematics and Statistics, Technological University Dublin}
\email{antti.kupiainen@helsinki.fi}
\author{Michael McAuley\textsuperscript{1,2}}
\email{m.mcauley@cantab.net}
\author{Eero Saksman\textsuperscript{1}}
\email{eero.saksman@helsinki.fi}
%60G57-Random Measures
%30C62-Quasiconformal mappings in the complex plane
%60J67-Stochastic (Schramm-)Loewner evolution (SLE)
\begin{abstract}
We solve the classical conformal welding problem for a composition of two random homeomorphisms generated by independent Gaussian multiplicative chaos measures with small parameter values. In other words, given two such measures on the boundary of the unit disk we show that there exist conformal maps to complementary domains on the Riemann sphere such that the pushforward of the normalised measures agree on their common boundary.
\end{abstract}
\date{\today}
\thanks{}

\maketitle
\section{Introduction and main result}\label{s:Intro}
\subsection{Introduction}
Schramm-Loewner evolution (SLE) and Liouville quantum gravity (LQG) are two major modern contributions to the field of mathematical physics. The former was first described by Schramm \cite{schramm00} as the scaling limit of loop-erased random walk and is now either known or conjectured to be the scaling limit of many other planar statistical mechanics models at criticality \cite{smirnov2001,lawler2004,schramm2005}. The latter is a model of random surfaces inspired by the work of Polyakov on conformal field theory \cite{polyakov2,polyakov1}. We give more details and some references regarding this model below.

These two objects were elegantly related to one another by Sheffield \cite{sheffield2016} through conformal welding: roughly speaking, if one takes a particular coupling of a LQG surface and chordal SLE then the SLE curve corresponds to `zipping up' the boundary of the surface in a length-preserving way. This insight is the foundation of many results for SLE, LQG and random planar maps (see \cite{ghs23} for a survey). 

The classical welding procedure can be traditionally viewed as consisting of two directions: the `direct problem' deals with constructing the welding homeomorphism $\phi:\S^1 \to \S^1$ when the Jordan curve $\Gamma\subset\C$ is given, where $\S^1$ is the unit circle in the complex plane. This direction is always solvable. The direction that is usually much harder in the deterministic world consists of solving the `welding problem', i.e., given a welding homeomorphism $\phi:\S^1\to\S^1$, does there exist a Jordan curve $\Gamma$ with this welding homeomorphism, and,  secondly, is it unique modulo M\"obius maps?

In dealing with rough random curves, the usually easier `direct problem' can also be remarkably difficult to deal with. This is demonstrated by Sheffield's fundamental work, where identifying the probabilistic structure of the welding homeomorphism $\phi$ for SLE-curves turns out to be very intricate. The uniqueness of the welding in this situation follows rather easily from the known H\"older property of the SLE-curves.

Our motivation in this work is to develop an alternative approach to directly solving the `welding problem' for given rough random self-homeomorphisms of $\S^1$. Thus the aim is to prove the existence of the unique (up to M\"obius maps) Jordan curves solving the welding problem without any a priori knowledge of the existence or any properties of the curves. The method can thus be used also as a way to construct such random loops as soon as the random welding homeomorphism is given.

Our approach is especially designed to cover SLE-loops (at least for small parameter values), but it has some inherent flexibility and can potentially be used in more general situations. Later on, we indeed demonstrate the latter by showing that it can also be used to weld two quantum wedges with different parameters. Below, we begin with an informal statement of our result in the basic case of two quantum wedges with the same parameter. The generalization to non-equal  parameter values is described in Section~\ref{s:generalisations} via Theorem~\ref{t:WeldingGeneral}.

Let $h_1,h_2$ be two independent copies of the planar Gaussian free field restricted to the unit circle $\S^1$ (which we identify with $\T:=\R\slash\Z\simeq[0,1)$). We formally define random measures on $[0,1)$ by
\begin{displaymath}
\tau^{(i)}(dx):=e^{\gamma h_i(x)}\;dx
\end{displaymath}
for $\gamma\in[0,\sqrt{2})$ and $i=1,2$. (Since $h_i$ is not defined pointwise, this measure must be defined by a normalisation procedure described below.) One can think of $\tau^{(i)}$ as measuring `quantum boundary length' for a LQG surface parameterised by the unit disk (this is a particular case of a Gaussian multiplicative chaos measure). Our main result states that for $\gamma>0$ sufficiently small, we can map these two surfaces conformally onto disjoint domains in $\hat{\C}$ with common boundary such that the normalised quantum boundary lengths between points coincide for each surface (see Figure~\ref{fig:WeldingIllustration}). 
\begin{figure}
    \centering
    \begin{tikzpicture}[scale=1]
	\begin{scope}[shift = {(0,1.8)}]
    \begin{scope}
	    \clip(-3,3) -- (-3,0) -- (-2,0) arc (180:360:2cm and 0.6cm) -- (3,0) -- (3,3) -- (-3,3);
		\shade[ball color=gray!60!white, opacity=0.70] (0,0) circle (2cm);
    \end{scope}

    \draw[dashed] (2,0) arc (0:180:2cm and 0.6cm);
    \draw (2,0) arc (0:-180:2cm and 0.6cm);
    \draw (-2,0) arc (180:0:2cm and 2cm);
    \foreach \x in { 70, 80, 95, 140, 190, 220, 260,  310, 330 } {
    \draw (\x:2cm and 0.6cm) node[cross out, draw, thick, minimum size=4pt, inner sep=0pt] {};
    }
    \draw (0:2cm and 0.6cm) node[circle, fill = black, draw, thick, minimum size=6pt, inner sep=0pt] {};
    \node[below right] at (0:2cm and 0.6cm) {$1$};
    \node at (-2,2) {$\hat{\mathbb{C}}\setminus\mathbb{D}$};
	\end{scope}    
    
	\begin{scope}
	    \clip(-3,-3) -- (-3,0) -- (-2,0) arc (180:360:2cm and 0.6cm) -- (3,0) -- (3,-3) -- (-3,-3);
		\shade[ball color=gray!60!white, opacity=0.70] (0,0) circle (2cm);
	\end{scope}
    \shade[shading=radial, inner color=gray!20!white, outer color=gray!60!white, opacity=0.70] (2,0) arc (0:360:2cm and 0.6cm); 
    \draw (2,0) arc (0:360:2cm and 0.6cm);
    \draw (-2,0) arc (180:360:2cm and 2cm);
    \foreach \x in {-25,  35, 90, 100, 135, 185, 220, 245, 280 } {
    \draw (\x:2cm and 0.6cm) node[circle, draw, thick, minimum size=6pt, inner sep=0pt] {};
    }
    \draw (0:2cm and 0.6cm) node[circle, fill = black, draw, thick, minimum size=6pt, inner sep=0pt] {};
    \node[below right] at (0:2cm and 0.6cm) {$1$};
    \node at (-2,-2) {$\overline{\mathbb{D}}$};

    \draw[->] (2.5,2.5)--(4.5,2) node [midway, above] {$f_2$};
    \draw[->] (2.5,-0.5)--(4.5,0)node [midway, above] {$f_1$};

    \begin{scope}[shift = {(7,0.9)}]
		\shade[ball color=gray!60!white, opacity=0.70] (0,0) circle (2cm);
        \draw (0,0) circle (2cm);
        \draw[line width = 1] plot [smooth] coordinates {(-2,0)(-1.8,-0.5)(-1.5,-0.4)(-1.3,-0.1)(-0.6,-0.8)(-0.1,-0.6)(0.3,-0.4)(0.7,-0.5)(0.9,0.1)(1.2,-0.7)(1.5,-0.4)(1.9,-0.2)(2,0)};
        \foreach \point in {(-1.3,-0.1), (-0.6,-0.8), (0.9,0.1),(1.2,-0.7)} {
        \draw \point node[circle, draw, thick, minimum size=6pt, inner sep=0pt] {};
        \draw \point node[cross out, draw, thick, minimum size=4pt, inner sep=0pt] {};
        }
        \draw (0.3,-0.4) node[circle, fill, thick, minimum size=6pt, inner sep=0pt] {};        
        
        %Curve on back of sphere
        \draw[line width = 1, dashed] plot [smooth] coordinates {(2,0)(1.6,0.3)(1.4,0.6)(1.2,0.5)(0.8,0.5)(0.4,0.2)(-0.2,0.5)(-0.7,0.2)(-1.4,0.6)(-1.8,0.2)(-2,0)};
        \foreach \point in {(2,0),(1.4,0.6),(0.4,0.2),(-0.7,0.2),(-1.4,0.6)} {
        \draw \point node[circle, draw, thick, minimum size=6pt, inner sep=0pt] {};
        \draw \point node[cross out, draw, thick, minimum size=4pt, inner sep=0pt] {};
        }
    \end{scope}
\end{tikzpicture}	

    \caption{The conformal maps $f_1,f_2$ can be chosen so that $f_1(1)=f_2(1)$ and for $x,y\in[0,1]$ we have $f_1(e^{2\pi ix})=f_2(e^{2\pi i y})$ if and only if $\tau^{(1)}([0,x])/\tau^{(1)}([0,1])=\tau^{(2)}([0,y])/\tau^{(2)}([0,1])$. Crosses and circles denote points to be matched up in this way. The images of $\partial\D$ under $f_1$ and $f_2$ coincide to give a closed curve in $\C$ which we think of as a subset of $\hat{\C}\simeq\S^2$.}
    \label{fig:WeldingIllustration}
\end{figure}
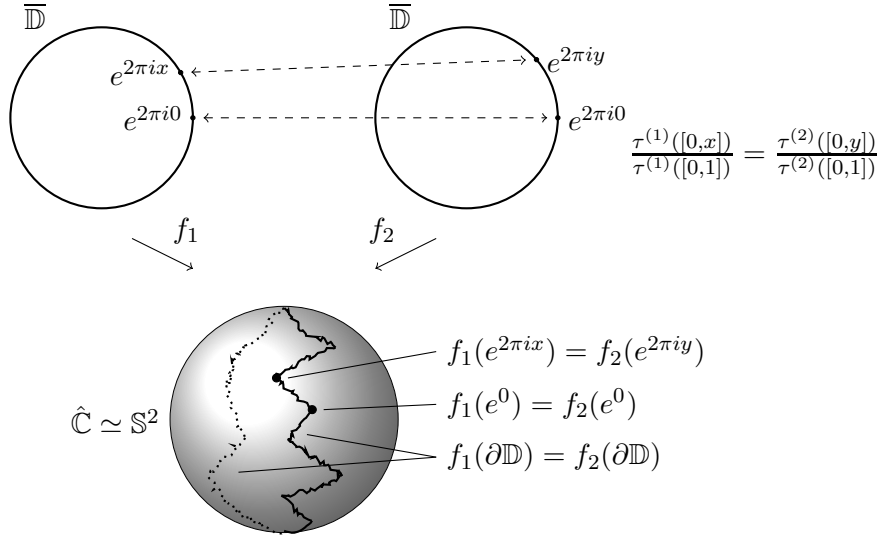

Before discussing our result and its relation to the literature in more detail, it will be helpful for us to give a precise statement and an outline of our method of proof.

\subsection{Statement of main result}\label{ss:Statement}
The planar Gaussian free field can be thought of intuitively as the Gaussian process $h$ with covariance function
\begin{displaymath}
\E[h(x)h(y)]=\log\frac{1}{\lvert x-y\rvert},\quad x,y\in\R^2.
\end{displaymath}
The trace of the two-dimensional Gaussian free field on the circle $\S^1$ is then the restriction of $h$ to $\S^1$. There is some subtlety to both of these definitions as this covariance function diverges on the diagonal (and also at infinity, although this won't be relevant for the trace), however they can be made sense of in terms of random distributions. We say that a random distribution $g$ taking values in $\mathcal{D}^\prime(\S^1)$ has covariance kernel $K:\S^1\times\S^1\to\R$ if for any smooth functions $\varphi_1,\varphi_2:\S^1\to\R$
\begin{displaymath}
\mathrm{Cov}[g(\varphi_1),g(\varphi_2)]=\int_{\S^1\times \S^1}K(x,y)\varphi_1(x)\varphi_2(y)\;dm_{\S^1}(x)dm_{\S^1}(y)
\end{displaymath}
where $m_{\S^1}$ denotes normalised Lebesgue measure on $\S^1$. In the case that $g$ is Gaussian and centred (i.e., $\E[g(\varphi)]=0$ for every $\varphi\in C^\infty(\S^1)$) its distribution is completely specified by this covariance kernel (see \cite{lifshits2012} for background on general Gaussian processes). We therefore define the trace of the Gaussian free field on $\S^1$ to be the centred Gaussian process $h$ taking values in $\mathcal{D}^\prime(\S^1)$ with covariance kernel
\begin{displaymath}
K(x,y)=\log\frac{1}{\lvert x-y\rvert},\quad x,y\in\S^1.
\end{displaymath}
Identifying $\S^1$ with $\T\simeq[0,1)$, the covariance kernel takes the form
\begin{displaymath}
K(t,u)=\log\frac{1}{2\sin(\pi\lvert t-u\rvert)},\quad t,u\in\T\simeq[0,1).
\end{displaymath}

We wish to define a measure $\tau$ on $\T$ by
\begin{displaymath}
\tau(dx)=e^{\gamma h(x)}dx
\end{displaymath}
where $\gamma>0$ is a parameter. This naive definition is inconsistent since $h$ is not defined pointwise, however it can be justified by taking a limit of regularised versions of $h$. Different regularisation procedures are possible (all leading to the same limit, see \cite{shamov2016}). We will use a regularisation based on a white noise representation of $h$ from \cite{bm02,bacry03} since this representation will be an important part of proving our main result.

We denote by $\lambda$ the hyperbolic measure on the upper half plane $\H$, that is $\lambda(dxdy)=(1/y^2)dxdy$. We let $W$ be a periodic Gaussian white noise with respect to this measure defined as follows: $W$ is a centred Gaussian process indexed by
\begin{displaymath}
\mathcal{B}^\ast(\H):=\left\{A\subset\H:A\text{ is Borel,  }\lambda(A)<\infty\text{ and }\sup_{(x_1,y_1),(x_2,y_2)\in A}\lvert x_1-x_2\rvert\leq 1\right\}
\end{displaymath}
with covariance function
\begin{displaymath}
\mathrm{Cov}(W(A_1),W(A_2))=\lambda\left(A_1\cap\left(\bigcup_{n\in\Z} (A_2+n)\right)\right).
\end{displaymath}
We define the sets
\begin{displaymath}
\mathcal{H}=\{(x,y)\in\H\;:\;\lvert x\rvert< 1/2, y\geq(2/\pi)\tan(\lvert\pi x\rvert)\},\quad \mathcal{H}_\epsilon=\{(x,y)\in \mathcal{H}\;:\;y\geq\epsilon\}\quad\text{for }\epsilon>0
\end{displaymath}
(see Figure~\ref{fig:whitenoise}). Finally we define the Gaussian field $H_\epsilon(\cdot)$ by 
\begin{displaymath}
H_\epsilon(x)=W(\mathcal{H}_\epsilon+x)\quad\text{for }x\in\T\simeq[0,1).
\end{displaymath}
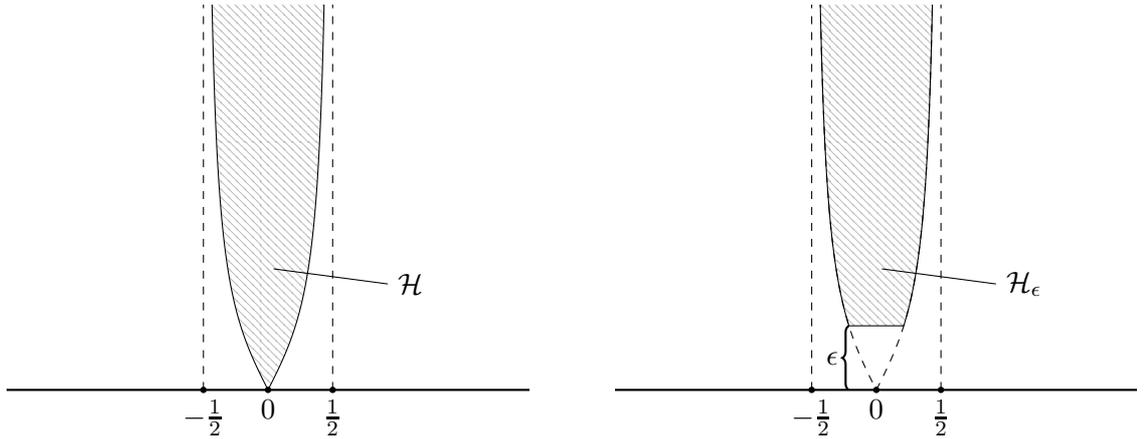
\begin{figure}[ht]
    \centering
 \begin{tikzpicture}[scale=1]

\begin{axis}[trig format plots=rad,samples=500,domain=-0.5:0.5,ymax=3,ymin=-0.35,axis equal,hide axis]
\addplot[samples=500,domain=-.48:0.48,name path=A] {(2/pi)*tan(abs(pi*x))};
\addplot[domain=-1.1:1.1,name path=B]{5};
\addplot[pattern=north west lines,opacity=0.5] fill between[of=A and B];
\filldraw[black] (0,0) circle (1pt) node[anchor=north]{$0$};
\filldraw[black] (0.5,0) circle (1pt) node[anchor=north]{$\frac{1}{2}$};
\filldraw[black] (-0.5,0) circle (1pt) node[anchor=north]{$-\frac{1}{2}$};
\draw[thick] (-5,0)--(5,0);
\draw [dashed] (0.5,0)--(0.5,5);
\draw [dashed] (-0.5,0)--(-0.5,5);
\end{axis}

\draw (3.5,2.2)--(5,2);
\node[right] at (5,2) {$\mathcal{H}$};

\begin{axis}[trig format plots=rad,samples=500,domain=-0.5:0.5,ymax=3,ymin=-0.35,axis equal,hide axis,xshift=8cm]
\addplot[domain=-.48:0.48,dashed] {(2/pi)*tan(abs(pi*x))};
\addplot[domain=-.48:0.48,name path=A] {max((2/pi)*tan(abs(pi*x)),0.5)};
\addplot[domain=-1.1:1.1,name path=B]{5};
\addplot[pattern=north west lines,opacity=0.5] fill between[of=A and B];
\filldraw[black] (0,0) circle (1pt) node[anchor=north]{$0$};
\filldraw[black] (0.5,0) circle (1pt) node[anchor=north]{$\frac{1}{2}$};
\filldraw[black] (-0.5,0) circle (1pt) node[anchor=north]{$-\frac{1}{2}$};
\draw[thick] (-5,0)--(5,0);
\draw [dashed] (0.5,0)--(0.5,5);
\draw [dashed] (-0.5,0)--(-0.5,5);
\draw [thick,decorate,
    decoration = {brace}] ({-atan(pi/4)/180},0) -- node[left] {$\epsilon$} ({-atan(pi/4)/180},0.5);
 
\end{axis}

\draw (11.5,2.2)--(13,2);
\node[right] at (13,2) {$\mathcal{H}_\epsilon$};

\end{tikzpicture}
    
    \caption{The sets $\mathcal{H}$ and $\mathcal{H}_\epsilon$ are used to define a white noise representation for the Gaussian free field.}
    \label{fig:whitenoise}
\end{figure}
This is our regularised process which yields $h$ in the limit:
\begin{lemma}[{\cite[Lemma~3.4]{ajks}}]\label{l:GFFExistence}
There exists a version of the white noise $W$ such that for all $\epsilon>0$ the map $x\mapsto H_\epsilon(x)$ is continuous and as $\epsilon\to 0^+$, $H_\epsilon(\cdot)$ converges in $\mathcal{D}^\prime(\T)$ to some $H(\cdot)$ such that
\begin{displaymath}
H(\cdot)\sim h+G
\end{displaymath}
where $G\sim\mathcal{N}(0,2\log(2))$ is a (scalar) Gaussian variable independent of $h$.
\end{lemma}
The continuity and convergence stated here follow from standard arguments for random fields (Dudley's theorem). To justify the distribution of $H(\cdot)$, an explicit computation shows that for $t\in(0,1)$
\begin{displaymath}
\lambda(H\cap(H+t))+\lambda(H\cap(H+t-1))=2\log(2)+\log\frac{1}{2\sin(\pi t)}
\end{displaymath}
matching the covariance function of $h+G$. With this representation, we can now define the desired measure $\tau$.

For any $\gamma>0$ and bounded Borel function $g$, the process
\begin{equation}\label{e:Martingale}
    \epsilon\mapsto\int_0^1g(x)e^{\gamma H_\epsilon(x)-\frac{\gamma^2}{2}\Var[H_\epsilon(x)]}\;dx
\end{equation}
is a martingale with respect to decreasing $\epsilon\in(0,1]$ (this follows easily from the definitions). The martingale is bounded in $L^1$ and hence converges almost surely. In particular, setting $g\equiv 1$ we see that the family of measures $e^{\gamma H_\epsilon(x)-\frac{\gamma^2}{2}\Var[H_\epsilon(x)]}\;dx$ indexed by $\epsilon$ are bounded in total variation norm. Hence by the generalised form of Prokhorov's theorem, every sequence of such measures has a subsequence which is weak-$\ast$ convergent. As $\epsilon\to 0$, the convergence of \eqref{e:Martingale} for, say, all polynomials with rational coefficients shows that the subsequential weak-$\ast$ limits coincide. Hence we can unambiguously define the almost sure weak-$\ast$ limit
\begin{equation}\label{e:MeasureDef}
\lim_{\epsilon\to 0^+}e^{\gamma H_\epsilon(x)-\frac{\gamma^2}{2}\Var[H_\epsilon(x)]}e^{-\gamma G}2^{\gamma^2}dx=:\tau(dx).
\end{equation}
Although this definition is valid for all $\gamma>0$, the limiting measure is known to be trivial (i.e., identically zero) whenever $\gamma^2\geq 2$ \cite{RV10}.

Let $\tau^{(1)},\tau^{(2)}$ be independent copies of the measure defined by \eqref{e:MeasureDef}. Let $\Theta_1,\Theta_2$ be two arbitrary random variables taking values in $[0,1]$ (in particular they may depend on $\tau^{(1)}$ and $\tau^{(2)}$ or they may be identically zero). We define the homeomorphisms $\phi_1,\phi_2$ by
\begin{equation}\label{e:DefineHomeo}
\phi_j(e^{2\pi ix}):=\exp\left(2\pi i\Theta_j+2\pi i\cdot\frac{\tau^{(j)}([0,x])}{\tau^{(j)}([0,1])}\right)\qquad\text{for }j=1,2\text{ and }x\in[0,1).
\end{equation}
Our main result is the following:
\begin{theorem}\label{t:Welding}
There exists $\gamma_0\in(0,\sqrt{2})$ such that for each $\gamma\in[0,\gamma_0]$ the following holds with probability one: there exist conformal maps
\begin{displaymath}
f_1:\D\to D,\quad\text{and}\quad f_2:\C\setminus\overline{\D}\to \C\backslash \overline{D}
\end{displaymath}
(where $D\subset\C$ is some simply connected domain) which may be extended to homeomorphisms of their closures such that $f_1\circ\phi_1^{-1}=f_2\circ\phi_2^{-1}$. Moreover the maps $f_1$ and $f_2$ are unique up to post-composition with a M\"obius transformation.
\end{theorem}

Described more succinctly, this theorem says that we can solve the classical conformal welding problem for the homeomorphism $\phi_1^{-1}\circ\phi_2$. We note that the welding problem was solved for $\phi_1$ and for $\phi_1\circ\phi_2^{-1}$ in \cite{ajks}. It was recently shown in \cite{fs25} that if $\Theta_1,\Theta_2$ are chosen to be uniform on $[0,1]$, independent of $\tau^{(1)}$ and $\tau^{(2)}$, then the welding curve induced by our construction is equivalent to an SLE loop measure. In Section~\ref{s:Literature} we discuss in more detail the motivation for considering the welding problem for these different homeomorphisms and compare our method of proof with that of \cite{ajks}.

In Section~\ref{s:DiffParam} we state and prove a generalisation of Theorem~\ref{t:Welding} for which the $\gamma$-parameters of $\phi_1$ and $\phi_2$ are different (and both less than $\gamma_0$). In particular, this includes the welding result for $\phi_1$ proven in \cite{ajks} (when $\gamma\leq\gamma_0$).

Our proof does not yield a quantitative lower bound for $\gamma_0$, however many of our arguments are valid for all $\gamma\in[0,\sqrt{2})$ so we hope that the general approach could be adapted to cover all such values (see the discussion in Section~\ref{s:Literature}).

\subsection{Outline of proof}\label{s:OutlineProof}
Our method of proof can be broken down into the following steps:
\medskip

\noindent\underline{Welding via the Beltrami equation:} A classical approach to solving the conformal welding problem is through the theory of quasiconformal mappings and the Beltrami equation. (For the purpose of this outline, readers unfamiliar with quasiconformal maps can think of them as simply a generalisation of conformal maps. A definition will be given in Section~\ref{s:Beltrami} and further background may be found in \cite{ahlfors2006,astala2009elliptic}.)

We extend $\phi_1$ and $\phi_2$ to $\Phi_1:\overline{\D}\to\overline{\D}$ and $\Phi_2:\C\backslash\D\to\C\backslash\D$ respectively via the Beurling-Ahlfors extension (which we describe in Section~\ref{s:Beltrami}). These extensions are orientation-preserving homeomorphisms which are differentiable almost everywhere. For a function $g$ satisfying the latter properties, we define the complex dilatation $\mu_g$ to be the measurable function satisfying $\partial_{\overline{z}}g=\mu_g\partial_z g$ (where $\partial_{\overline{z}}=(1/2)(\partial_x+i\partial_y)$ and $\partial_z=(1/2)(\partial_x-i\partial_y)$ denote the standard Wirtinger derivatives). Considering the Jacobian of $g$ shows that $\lvert\mu_g\rvert<1$ almost everywhere, so we may define the distortion of $g$ as $K_g:=(1+\lvert\mu_g\rvert)/(1-\lvert\mu_g\rvert)$. The distortion/dilatation in some sense capture how close the function $g$ is to being conformal, since $\mu_g=0$ implies that $g$ is holomorphic (up to sets of measure zero). We wish to find a quasiconformal homeomorphism $F:\C\to\C$ such that
\begin{equation}\label{e:Beltrami}
\mu_F(z)=\begin{cases}
\mu_{\Phi_1^{-1}}(z) &\text{if }z\in\D\\
\mu_{\Phi_2^{-1}}(z) &\text{if }z\in\C\backslash\D.
\end{cases}
\end{equation}
This is a specific instance of the Beltrami equation. If we could obtain such an $F$, then the Stoilow factorisation theorem (a uniqueness result for the Beltrami equation) would state that the functions $f_1,f_2$ defined by
\begin{equation}\label{e:Welding}
\begin{aligned}
f_1&:=F\circ\Phi_1:\overline{\D}\to F(\overline{\D})\\
f_2&:=F\circ\Phi_2:\C\backslash\D\to F(\C\backslash\D)
\end{aligned}
\end{equation}
are conformal on the interior of their domains. Then by definition
\begin{displaymath}
f_1\circ\phi_1^{-1}=F|_{\partial\D}=f_2\circ\phi_2^{-1}
\end{displaymath}
and so $f_1$ and $f_2$ would solve our welding problem.

The standard existence theorem for the Beltrami equation asserts that equations of the form \eqref{e:Beltrami} have a solution whenever the right hand side is bounded (in absolute value) uniformly away from one. This is known as the uniformly elliptic case. Unfortunately this condition is not satisfied in our setting for points close to $\partial\D$; our Beltrami equation is thus described as degenerate.

To make progress, we imitate a classical argument for solving degenerate Beltrami equations that originated with Lehto \cite{lehto1970} (see \cite[Chapter~20]{astala2009elliptic} for a modern treatment and further references). We define a sequence of Beltrami equations by
\begin{displaymath}
\mu_{F_n}(z)=\begin{cases}
\frac{n}{n+1}\mu_{\Phi_1^{-1}}(z) &\text{if }z\in\D\\
\frac{n}{n+1}\mu_{\Phi_2^{-1}}(z) &\text{if }z\in\C\backslash\D
\end{cases}.
\end{displaymath}
Each such equation is uniformly elliptic and so has a well-defined solution. Supposing that the sequence $F_n$ is equicontinuous, the Arzel\`a-Ascoli theorem produces a subsequential limit which, by standard analytic arguments, may be shown to solve \eqref{e:Beltrami}. It is relatively straightforward to establish equicontinuity away from $\partial\D$ in our setting: this follows from almost sure bounds on the distortion of $\Phi_1^{-1}$ and $\Phi_2^{-1}$ which are a consequence of well-known moment bounds on $\tau^{(1)}$ and $\tau^{(2)}$ and the definition of the Beurling-Ahlfors extension. It is (apparently) much more challenging to prove equicontinuity on $\partial\D$ and doing so occupies the majority of our paper, as we describe below.

In fact we eventually establish the stronger statement that the family $F_n$ satisfies a uniform H\"older continuity bound on $\partial\D$. This implies that $F(\partial\D)$ is a H\"older domain which, by a conformal removability result due to Jones and Smirnov \cite{jones2000removability}, guarantees the uniqueness of our solution to the welding problem.
\medskip

\noindent\underline{H\"older continuity via `undistorted' annuli:} The primary information we have about the functions $F_n$ is a uniform bound on their distortion and from this we wish to prove a uniform H\"older continuity bound. We make the link between these properties using conformal modulus/extremal length (which is described in Section~\ref{s:Beltrami} for those who are unfamiliar with the concept). Specifically it is enough to show that for each point of $\partial\D$, there exists a sequence of geometrically shrinking concentric annuli $\mathbb{A}_i$ surrounding the point, such that the images of these annuli under any $F_n$ has conformal modulus bounded away from zero (see Figure~\ref{fig:AnnuliDistort}). The latter property follows if the distortion of $F_n$ on each $\mathbb{A}_i$ is not too large (relative to the size of $\mathbb{A}_i$).

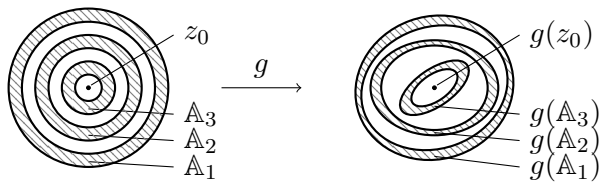
\begin{figure}[h]
    \centering
\begin{tikzpicture}[scale=0.35]
\draw [->] (5,0)--node[above]{$g$}(8,0);

\node[right] at (3.2,2) {$z_0$};
\draw (3.2,2)--(0.1,0.1);
\node[right] at (3.2,-3) {$\mathbb{A}_1$};
\draw (3.2,-3)--(0,-2.8);
\node[right] at (3.2,-2) {$\mathbb{A}_2$};
\draw (3.2,-2)--(0,-1.8);
\node[right] at (3.2,-1) {$\mathbb{A}_3$};
\draw (3.2,-1)--(0,-.8);

\draw[fill] (0,0) circle (2pt);
\fill[even odd rule,pattern=north west lines,pattern color=gray] (0,0) circle (.5) circle (1) circle (1.5) circle(2) circle(2.5) circle(3);
\draw[thick](0,0)  circle (.5) circle (1) circle (1.5) circle(2) circle(2.5) circle(3);

\draw[fill] (13,0) circle (2pt);
\begin{scope}[shift={(13,0)}]
\fill[even odd rule,pattern=north west lines,pattern color=gray,rotate=35] (0,0) ellipse (1 and 0.5) ellipse (1.5 and 0.7);
\fill[even odd rule,pattern=north west lines,pattern color=gray](0,0) ellipse (2 and 1.6) ellipse (2.4 and 1.8);
\fill[even odd rule,pattern=north west lines,pattern color=gray,rotate=20] (0,0) ellipse (2.8 and 2.3) ellipse (3 and 2.7);

\draw[thick,rotate=35] (0,0) ellipse (1 and 0.5) ellipse (1.5 and 0.7);
\draw[thick] (0,0) ellipse (2 and 1.6) ellipse (2.4 and 1.8);
\draw[thick,rotate=20] (0,0) ellipse (2.8 and 2.3) ellipse (3 and 2.7);
\node[right] at (3.2,2) {$g(z_0)$};
\draw (3.2,2)--(0.1,0.1);
\node[right] at (3.2,-3) {$g(\mathbb{A}_1)$};
\draw (3.2,-3)--(0,-2.6);
\node[right] at (3.2,-2) {$g(\mathbb{A}_2)$};
\draw (3.2,-2)--(0,-1.7);
\node[right] at (3.2,-1) {$g(\mathbb{A}_3)$};
\draw (3.2,-1)--(0,-.7);
\end{scope}
\end{tikzpicture}
    \caption{To bound the modulus of continuity of a homeomorphism $g$ near a point $z_0$, it is enough to show that the images of small annuli $\mathbb{A}_i$ surrounding $x$ are not too `distorted' (i.e., they are not too long or thin). We apply this reasoning to $g=F_n$ around points $z_0\in\partial\D$.}
    \label{fig:AnnuliDistort}
\end{figure}

The distortion of $F_n$ on any $\mathbb{A}_i$ is defined in terms of the distortion of $\Phi_1$ and $\Phi_2$ on $\Phi_1^{-1}(\mathbb{A}_i)\cup\Phi_2^{-1}(\mathbb{A}_i)$. This set would be difficult to work with if we chose the $\mathbb{A}_i$ deterministically, so instead we define them as images under $\Phi_1$ and $\Phi_2$ of deterministic sets which we call `half-annuli'. These sets, and the ensuing arguments, are more easily described by a conformal change of domain.

We identify $\R$ as the periodic extension of $\partial\D$ via $z\mapsto e(z):=\exp(2\pi iz)$ and we define our half annuli as follows: for $t\geq 0$, $x\in\R$ and a small parameter $\rho\in(0,1)$ let
\begin{displaymath}
    A_t(x):=x+([-\rho^t,\rho^t]\times[0,\rho^t])\setminus((-\rho^{t+1/4},\rho^{t+1/4})\times(0,\rho^{t+1/4}))
\end{displaymath}
(see Figure~\ref{fig:HalfAnnuli}) and let $\widetilde{A}_t(x)$ be the reflection of this set in the real axis. Observe that as $t$ increases, the half-annuli shrink towards the point $x$.

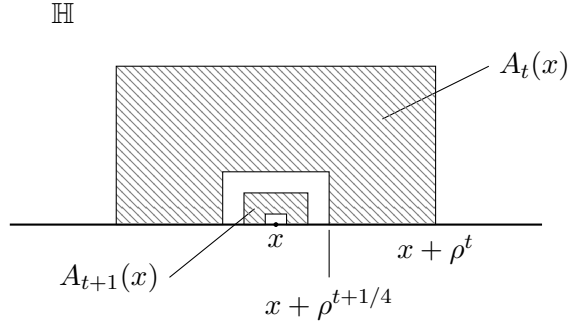
\begin{figure}[h]
    \centering
    \begin{tikzpicture}[scale=0.7]

\node at (-4,4) {$\mathbb{H}$};
\draw[thick] (-5,0)--(5,0);
\node[right] at (4,3) {$A_t(x)$};
\draw (4,3)--(2,2);
\node[left] at (-2,-1) {$A_{t+1}(x)$};
\draw (-2,-1)--(-0.4,.3);
\node[below] at (3,0) {$x+\rho^t$};
\node[below] at (1,-1) {$x+\rho^{t+1/4}$};
\draw (1,-1)--(1,-0.1);

\filldraw[black] (0,0) circle (1pt) node[anchor=north]{$x$};

\draw [pattern=north west lines,pattern color=gray](-3,0)--(-3,3)--(3,3)--(3,0)--(1,0)--(1,1)--(-1,1)--(-1,0)--(-3,0);
\draw[pattern=north west lines,pattern color=gray] (-3/5,0)--(-3/5,3/5)--(3/5,3/5)--(3/5,0)--(1/5,0)--(1/5,1/5)--(-1/5,1/5)--(-1/5,0)--(-3/5,0);

\end{tikzpicture}
    \caption{The half-annuli $A_t(x)$ and $A_{t+1}(x)$.}
    \label{fig:HalfAnnuli}
\end{figure}

For $j=1,2$ we let $\Psi_j$ be the periodic extension of $e^{-1}\circ\Phi_j\circ e$ so that $\Psi_1$ and $\Psi_2$ are (periodic) homeomorphisms of $\H$ and $\C\setminus\H$ respectively. Suppose that we have a pair of half-annuli $A_t(x)$ and $\widetilde{A}_s(y)$ such that the union of their images, under $\Psi_1$ and $\Psi_2$ respectively, contains a topological annulus $\mathbb{A}$ and the distortion of $F_n\circ e$ is not too large on $\mathbb{A}$. Then under our identification of $\R$ and $\partial\D$, this allows us to obtain an annulus around some point in $\partial\D$ on which the distortion of $F_n$ is bounded above (see Figure~\ref{fig:AnnuliIllustration}).

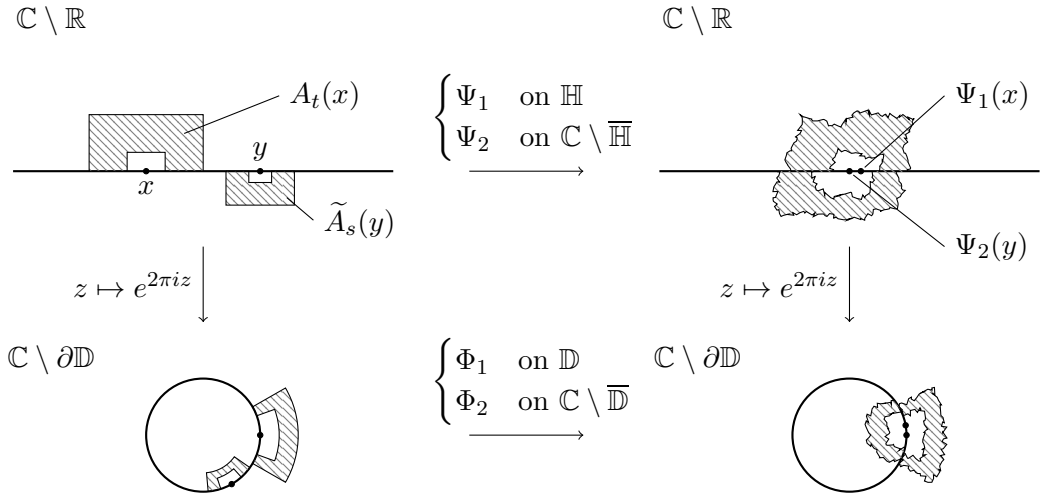
\begin{figure}[h]
    \centering
    \begin{tikzpicture}[use fpu reciprocal,scale=0.5]
\pgfmathsetseed{1}
\node at (-4,4) {$\mathbb{C}\setminus\mathbb{R}$};
\draw[thick] (-5,0)--(5,0);
\filldraw[black] (-1.5,0) circle (2pt) node[anchor=north]{$x$};
\filldraw[black] (1.5,0) circle (2pt) node[anchor=south]{$y$};
\begin{scope}[scale=0.5,shift={(-3,0)}]
\draw [pattern=north west lines,pattern color=gray](-3,0)--(-3,3)--(3,3)--(3,0)--(1,0)--(1,1)--(-1,1)--(-1,0)--(-3,0);
\end{scope}
\begin{scope}[scale=0.3,shift={(5,0)}]
\draw [pattern=north west lines,pattern color=gray](-3,0)--(-3,-3)--(3,-3)--(3,0)--(1,0)--(1,-1)--(-1,-1)--(-1,0)--(-3,0);
\end{scope}
\node[right] at (2,2) {$A_t(x)$};
\draw (2,2)--(-0.5,1);
\node[right] at (2.9,-1.3) {$\widetilde{A}_s(y)$};
\draw (2.9,-1.3)--(2.2,-0.7);

\draw[->] (7,0) -- node [text width=2.5cm,midway,above] {$\begin{cases}\Psi_1&\text{on }\mathbb{H}\\\Psi_2&\text{on } \mathbb{C}\setminus\overline{\mathbb{H}}\end{cases}$} (10,0);

\begin{scope}[shift={(17,0)}]
\node at (-4,4) {$\mathbb{C}\setminus\mathbb{R}$};
\draw[thick] (-5,0)--(5,0);
\filldraw[black] (0,0) circle (2pt);
\filldraw[black] (0.3,0) circle (2pt);

\begin{scope}
\clip (-5,0) rectangle (5,4);
\draw[pattern=north west lines, pattern color = gray, decorate,decoration={random steps,segment length=1pt,amplitude=1pt}] (-1.5,-0.2) to ++ (-0.2,0.5) to ++(0.4,0.7) to ++(0,0.5) to ++(1,-0.3) to ++(0.5,+0.4)to++(1.5,-0.2)
to ++(-0.3,-0.8)
to ++(0.3,-0.9)
to ++(-1,0) 
to ++(0.2,+0.7)
to ++(-0.5,0) to ++(-0.7,0.2)
to ++(-0.2,-0.7)
to ++(-1,-0.1);
\end{scope}
\begin{scope}[shift={(-0.1,0)}]
\clip (-5,-4) rectangle (5,0);
\draw[pattern=north west lines, pattern color = gray, decorate,decoration={random steps,segment length=1pt,amplitude=1pt}] (-1.7,0.2)--(-1.9,-0.8)--(-1.7,-1.2)--(-1.2,-1.1)--(0.5,-1.3)--(0.8,-1)--(1.3,-1)--(1.5,-0.5)--(1.3,0.2)--(0.7,0.2)--(0.7,-0.3)--(0.5,-0.7)--(0,-0.6)--(-0.4,-0.7)--(-0.8,-0.4)--(-0.9,0.2)--cycle;
\end{scope}

\node[right] at (2.5,2){$\Psi_1(x)$};
\draw (2.5,2)--(0.4,0.1);
\node[right] at (2.5,-2){$\Psi_2(y)$};
\draw (2.5,-2)--(0.1,-0.1);

\end{scope}

\draw [->] (0,-2)--node[left]{$z\mapsto e^{2\pi i z}$}(0,-4);
\draw [->] (17,-2)--node[left]{$z\mapsto e^{2\pi i z}$}(17,-4);

\begin{scope}[shift={(0,-7)}]
\node at (-4,2) {$\mathbb{C}\setminus\partial\mathbb{D}$};
\draw[thick] (0,0) circle (1.5);
\filldraw[black] (0:1.5) circle (2pt);
\filldraw[black] (-60:1.5) circle (2pt);

\draw[pattern=north west lines,pattern color=gray] (0,0)+(-20:2) arc[start angle=-20, end angle=20, radius=2]--(20:1.5)arc[start angle=20, end angle=30, radius=1.5]--(30:2.5)arc[start angle=30, end angle=-30, radius=2.5]--(-30:1.5)arc[start angle=-30, end angle=-20, radius=1.5]--(-20:2);

\draw[pattern=north west lines,pattern color=gray,rotate=-60] (0,0)+(-12:1.2) arc[start angle=-12, end angle=12, radius=1.2]--(12:1.5)arc[start angle=12, end angle=25, radius=1.5]--(25:1)arc[start angle=25, end angle=-25, radius=1]--(-25:1.5)arc[start angle=-25, end angle=-12, radius=1.5]--(-12:1.2);
\end{scope}

\draw[->] (7,-7) -- node [text width=2.5cm,midway,above] {$\begin{cases}\Phi_1&\text{on }\mathbb{D}\\\Phi_2&\text{on } \mathbb{C}\setminus\overline{\mathbb{D}}\end{cases}$} (10,-7);

\begin{scope}[shift={(17,-7)}]
\node at (-4,2) {$\mathbb{C}\setminus\partial\mathbb{D}$};
\draw[thick] (0,0) circle (1.5);
\filldraw[black] (1.5,0) circle (2pt);
\filldraw[black] (10:1.5) circle (2pt);

\begin{scope}
\clip (0,0) circle (1.5);
\draw[pattern=north west lines,pattern color=gray,decorate,decoration={random steps,segment length=1pt,amplitude=1pt},rotate=5] (0,0)+(-20:1) arc[start angle=-20, end angle=20, radius=1]--(20:1.6)arc[start angle=20, end angle=30, radius=1.6]--(45:1)--(30:.5)arc[start angle=30, end angle=-30, radius=.5]--(-40:1.2)--(-30:1.6)arc[start angle=-30, end angle=-20, radius=1.6]--cycle;

\end{scope}

\begin{scope}
\clip (-45:1.5)arc[start angle=-45, end angle=45, radius=1.5]--(45:3)arc[start angle=45, end angle=-45, radius=3]--(-45:1.5);
\draw[pattern=north west lines,pattern color=gray,decorate,decoration={random steps,segment length=1pt,amplitude=1pt}] (0,0)+(-20:2) arc[start angle=-20, end angle=20, radius=2]--(20:1.5)arc[start angle=20, end angle=30, radius=1.5]--(30:2.5)arc[start angle=30, end angle=-30, radius=2.5]--(-30:1.5)arc[start angle=-30, end angle=-20, radius=1.5]--(-20:2);
\end{scope}
\end{scope}

\draw[->] (15,-9) -- node [midway,below right] {$F_n$} (12,-11.5);

\begin{scope}[shift={(8.5,-14)},scale = 1.1]
\node at (-4,2) {$\mathbb{C}\setminus F_n(\partial\mathbb{D})$};
\draw[pattern=north west lines,pattern color=gray,decorate,decoration={random steps,segment length=0.5pt,amplitude=0.5pt}] (-5:1.3)--(-5:1) -- (-10:0.8)--(5:0.7)--(20:0.6)--(35:0.6)--(45:0.7)--(63:0.8)--(62:1.1)--(65:1.4)--(65:1.9)--(60:2.2)--(60:2.6)--(55:2.5)--(50:2.4)--(40:2.6)--(30:2.4)--(15:2.6)--(0:2.4)--(0:1.8)--cycle;
\draw[fill=white, decorate,decoration={random steps,segment length=0.5pt,amplitude=0.5pt}] (0:1.2)--(5:1.6)--(10:2.1)--(15:2.2)--(25:2.2)--(35:2.0)--(45:2.2)--(50:2)--(50:1.8)--(55:1.6)--(55:1.2)--(50:1.1)--(35:1.0)--(20:0.8)--(10:0.9)--cycle;

\draw[thick,decorate,decoration={random steps,segment length=1pt,amplitude=0.5pt}] (0:1.5)--(30:1.2)--(40:1.3)--(70:1.7)--(75:1.3)--(100:1.8)--(130:1.8)--(130:1.1)--(160:1.1)--(180:1.1)--(185:1.6)--(210:1.6)--(240:1.4)--(245:1.7)--(280:1.4)--(300:1.2)--(325:1.6)--(345:1.4)--(0:1.5);

\filldraw[black] (40:1.3) circle (2pt);
\filldraw[black] (30:1.25) circle (2pt);
\end{scope}

\end{tikzpicture}

     \caption{The images of the upper and lower half-annuli $A_t(x)$ and $\widetilde{A}_s(y)$ may `match up', in the sense that their union contains a topological annulus, if $\Psi_1(x)\approx\Psi_2(y)$. This is equivalent to finding annuli around points of $\partial\D$, although it is notationally easier to work on $\R$.}
    \label{fig:AnnuliIllustration}
\end{figure}

Our objective then is to find, for every point in $[0,1)$, sufficiently many pairs of such half-annuli whose images surround the chosen point. Consider a finely spaced grid of points in $[0,1)$. H\"older continuity of $\psi_1$ and $\psi_2$ implies that, for each $x$ in the grid we can find a $y$ in the grid such that $\psi_1(x)\approx\psi_2(y)$. Roughly speaking this means that the `centres' of the half annuli $\Psi_1(A_t(x))$ and $\Psi_2(\widetilde{A}_s(y))$ match and this gives us some hope of finding a suitable annulus in the union of these sets. In turn, this requires that the two image half-annuli are of comparable \emph{size} and each have a somewhat regular \emph{shape}. Finally we require that the distortion of $F_n\circ e$ on each half-annulus is not too large. Using the Borel-Cantelli lemma, it is then sufficient to show that these conditions hold with high probability for some sequence of values $(s_n)$, $(t_n)$ and each $x$ and $y$ in the grid. 

\medskip

\noindent\underline{Conditions for `good' half annuli:} To estimate the probability of the above conditions, we must relate them to the measures $\tau^{(1)}$ and $\tau^{(2)}$. From the explicit linear form of the Beurling-Ahlfors extension, it follows that $\Psi_1(A_t(x))$ contains a regular half-annulus of size comparable to $\psi_1(x+\rho^t)-\psi_1(x-\rho^t)$ provided that $\psi_1$ does not vary too much on scale $\rho^t$ near $x$. In fact, an elementary argument shows that this holds on the intersection of finitely many events of the form \begin{equation}\label{e:ShapeEvent}
\frac{\tau^{(1)}(x+\rho^{t}I)}{\tau^{(1)}(x+\rho^{t}J)}\leq c
\end{equation}
where $I,J\subset[0,1]$ and $c>0$ do not depend on $t$. Hence the event that $\Psi_1(A_t(x))$ and $\Psi_2(\widetilde{A}_s(y))$ are of comparable size, denoted $\mathrm{Size}(t,s)$, is of the form
\begin{equation}\label{e:SizeEvent}
    c^{-1}\leq\frac{\tau^{(1)}(x+[-\rho^t,\rho^t])}{\tau^{(1)}([0,1])}\Big/\frac{\tau^{(2)}(y+[-\rho^s,\rho^s])}{\tau^{(2)}([0,1])}\leq c
\end{equation}
for some $c>0$. By a change of variables and simple bounds on the derivative of the Beurling-Ahlfors extension, the distortion of $F_n\circ e$ on $\Psi_1(A_t(x))$ is bounded uniformly over $n$ by
\begin{displaymath}
    c\int_{(u,v)\in A_t(x)}\frac{\lvert\psi_1(u+v)-\psi_1(u-v)\rvert^2}{v^2}\;dudv.
\end{displaymath}
We require the ratio of this integral to the size of $\Psi_1(A_t(x))$ squared to be bounded above by a constant. By partitioning the integral over dyadic squares, this ratio is bounded by an $\ell^2$ sum of terms like those on the left-hand side of \eqref{e:ShapeEvent}. We therefore define $\mathrm{Shape}^{(1)}(t)$ to be a suitable intersection of events of the form \eqref{e:ShapeEvent} and define $\mathrm{Shape}^{(2)}(s)$ analogously for the lower half-annulus $\widetilde{A}_s(y)$.

The classical existence  theory for quasiconformal maps would yield a solution to our welding problem if $\psi_1$ and $\psi_2$ were quasisymmetric. For comparison, we mention that in this case events of the form \eqref{e:ShapeEvent} would hold deterministically for all $x$ and $t$ (with an appropriate choice of constant $c>0$) verifying the shape events. Moreover for each $x$ and $t$ we would be able to find $y$ and $s$ such that $\psi_1(x)=\psi_2(y)$, $s$ is comparable to $t$ and \eqref{e:SizeEvent} holds. Then choosing a sequence $t_n=n$ would yield the required concentric annuli (the condition that $s$ is comparable to $t$ would ensure that the annuli shrink approximately geometrically). In the present setting, the size and shape events may fail for any choice of constants $c$. We therefore need a probabilistic argument that identifies suitable sequences $t_n$ and $s_n$ for which these events occur.

\medskip

\noindent\underline{An approximate decomposition for $\tau^{(1)}$ and $\tau^{(2)}$:} We need to show that the shape and size events occur for many different values of $t$ and $s$. To control the dependence between events, we decompose the measures $\tau^{(1)}$ and $\tau^{(2)}$ using the hyperbolic white noise with which they are defined.

Let $\tau^{(j)}_t$ denote the measure constructed analogously to $\tau^{(j)}$ by restricting the hyperbolic white noise to the horizontal strip $\{(u,v)\in\H\;:\;v\leq\rho^t\}$. For $t\geq 0$ let $B_t=[-\rho^t,\rho^t]$. From the definition of $\tau^{(j)}$ in \eqref{e:MeasureDef} it may seem plausible to the reader that for large $t$
\begin{equation}\label{e:ApproxDecomp}
    \tau^{(j)}(x+A)\approx\tau^{(j)}_t(x+A\setminus B_{t+7/8})\exp\Big(\gamma H^{(j)}_{\rho^t}(x)-\frac{\gamma^2}{2}\Var[H^{(j)}_{\rho^t}(0)]\Big)
\end{equation}
for intervals $A$ which have length of order $\rho^t$. We give further heuristics for this in Section~\ref{s:Decompose} where we also verify rigorously that the approximation is valid, with high probability, for many values of $t$ and $j=1,2$. The justification follows from essentially elementary methods: moment bounds on $\tau^{(j)}$ and independence of the hyperbolic white noise on disjoint regions.

\medskip

\noindent\underline{Reduced size and shape events:} The approximate decomposition of $\tau^{(1)}$ and $\tau^{(2)}$ simplifies our events of interest; substituting the right hand side of \eqref{e:ApproxDecomp} into \eqref{e:ShapeEvent} yields the `reduced' shape event
\begin{displaymath}
    \frac{\tau^{(1)}_t(x+(\rho^{t}I)\setminus B_{t+7/8})}{\tau^{(1)}_t(x+(\rho^{t}J)\setminus B_{t+7/8})}\leq c^\prime
\end{displaymath}
which we denote by $\mathrm{ShapeRed}^{(1)}(t)$. These events are much nicer to work with since they depend on disjoint regions of the hyperbolic white noise for distinct values of $t$ and so are independent. (More precisely, this is true of $\mathrm{ShapeRed}^{(1)}(t)$ and $\mathrm{ShapeRed}^{(1)}(t^\prime)$ for $\lvert t^\prime-t\rvert\geq 1$.) We therefore use standard large deviation estimates for Bernoulli random variables to show that the reduced shape events occur for many values of $t$ (for $j=1,2$).

Similarly if we substitute \eqref{e:ApproxDecomp} into \eqref{e:SizeEvent} then we obtain an event of the form
\begin{displaymath}
1/c^\prime\leq\exp\left(X^{(H)}_{t,s}\right)\frac{\rho^{-t}\tau^{(1)}_{t}(x+[-\rho^{t},\rho^{t}])}{\rho^{-s}\tau^{(2)}_{s}(y+[-\rho^{s},\rho^{s}])}\leq c^\prime
\end{displaymath}
where
\begin{displaymath}
X^{(H)}_{t,s}:=\gamma H^{(1)}_{\rho^t}(x)-\gamma H^{(2)}_{\rho^s}(y)-\Big(1+\frac{\gamma^2}{2}\Big)\log(1/\rho)(t-s).
\end{displaymath}
It is relatively easy to control the fraction in the above expression as it has the same independence properties as the events $\mathrm{ShapeRed}^{(j)}(t)$ for different values of $t$ and $s$. 

Our final objective then is to find many values of $t$ and $s$ (which differ by at least one) such that $\lvert X_{t,s}^{(H)}\rvert$ is not too large. From the definition of $H^{(j)}_\epsilon$ it is apparent that $t\mapsto H^{(1)}_{\rho^t}(x)$ is continuous with independent Gaussian increments and so $X_{t,s}^{(H)}$ can be viewed as a difference of two time-changed Brownian motions with drift, where we allow the time parameters $t$ and $s$ to vary independently. This motivates us to find appropriate values of $t$ and $s$ iteratively: given a current value of $X_{t,s}^{(H)}$, we increase either $t$ or $s$ so that the drift term brings our process closer to zero (on average). In this way we define the sequence of points $(t_n)$ and $(s_n)$. The resulting iterative process is an oscillating random walk (which we define in Section~\ref{ss:algorithm}) and we are able to use basic probabilistic methods to show that $\lvert X_{t_n,s_n}^{(H)}\rvert$ is small for many $n$ with high probability.

\subsection{Discussion}\label{s:Literature}\hfill\\
\noindent\underline{Motivation:} Our result contributes to a line of work which originated with the conjecture by Peter Jones that a variant of SLE could be obtained by solving the conformal welding problem for the homeomorphism $\phi_1$ defined by \eqref{e:DefineHomeo}. It was shown in \cite{ajks} that this problem admits a solution, defining a family of random closed curves. Soon after, Sheffield \cite{sheffield2016} proved that SLE curves correspond directly to welding two independent Liouville Quantum Gravity surfaces according to their boundary measures. More precisely, Sheffield showed that an SLE curve could be coupled with a half-plane Gaussian free field in such a way that welding the two halves of the real line according to the measures induced by the field reproduces the curve. This result has had many significant consequences and generalisations related to random surfaces: see \cite{ghs23} for an overview. Sheffield's result strongly suggested that Jones' conjecture would in fact be correct if one instead solves the conformal welding problem for $\phi_1^{-1}\circ\phi_2$. While the current article was under review, this was confirmed by Fan and Sung \cite[Lemma~4.8]{fs25}, provided that $\Theta_1,\Theta_2\sim\mathrm{Unif}[0,1]$ are independent of $\tau^{(1)}$ and $\tau^{(2)}$.

Our motivation in this work is to build towards an alternative approach to conformal welding for random curves. In contrast with Sheffield and (most) subsequent authors, we start with a random homeomorphism (or equivalently, the measures to be matched) and use this to generate a curve. We therefore do not require a priori control on the welding curve. This approach offers some advantages to studying the relationship between measures and curves: first, it offers a path toward proving continuity properties of the curves (see Theorem~\ref{t:translations} or \cite[Theorem~5.3]{ajks}). Second, it may generalise more readily to other measures; for example in Theorem~\ref{t:WeldingGeneral} we weld measures with distinct parameter values $\gamma_1\neq\gamma_2$. This seems difficult to accomplish using other methods. 

\smallskip

\noindent\underline{Related welding models}: In \cite{AHS20}, Sheffield’s welding result was extended to variants of SLE with two marked points and in \cite{AHS22} to SLE loop measures constructed in \cite {Zhan}. These works were built upon in \cite{fs25} to show that, after suitable normalisation, the SLE loop measures are mutually absolutely continuous with respect to the welding curves induced by $\phi_1^{-1}\circ\phi_2$. We refer to \cite{fs25} for definitions and a precise statement. It follows that these various welding problems are all connected and in particular that our welding result (in the case $\gamma_1=\gamma_2$) follows in some sense from that of Sheffield (although substantial further insight and machinery is required to establish this). On the other hand, our result for differing parameter values (i.e., Theorem~\ref{t:WeldingGeneral}) seems very unlikely to follow from the previous line of reasoning: we would not expect the welding curve in this case to be a variant of SLE, so many of the tools used in the work of Sheffield and others would not be present a priori. We consider it an interesting open question to describe the properties of the welding curve when $\gamma_1\neq\gamma_2$.

\smallskip

\noindent\underline{Approaches to conformal welding:} The proof of Sheffield's welding result in \cite{sheffield2016} made use of stochastic calculus and a scaling argument, both of which rely crucially on properties of SLE. (We refer the reader to \cite{bp25} for an accessible presentation of this argument and also to \cite{ps24} for a more recent, simplified proof of the result.) Most subsequent work on conformal welding in the setting of LQG has used this result as an input to establish welding results for related curves and measures.

Our approach is more closely related to those of classical complex analysis which take a homeomorphism as input to construct welding curves.  General background on the conformal welding problem, along with some historical references, can be found in \cite{ham02}. The fundamental theorem of conformal welding (proven by Pfluger \cite{pfl60}, see \cite[Chapter~5.10]{astala2009elliptic}) states that the conformal welding problem has a unique (up to M\"obius transformation) solution for a homeomorphism $g:\R\to\R$ whenever $g$ is quasisymmetric; that is there exists some $C>0$ such that
\begin{displaymath}
    \frac{1}{C}\leq\frac{g(x+y)-g(x)}{g(x)-g(x-y)}\leq C\qquad\text{for all }x,y\in\R.
\end{displaymath}
This can be proven by solving the Beltrami equation for the Beurling-Ahlfors extension of $g$ (which will be quasiconformal) and using the Stoilow factorisation theorem, as in the outline of our proof.

We cannot apply this classical result because our homeomorphisms are not quasisymmetric; the Beltrami equation we wish to solve is described as degenerate. Deterministic results show that one can solve a degenerate Beltrami equation whenever the distortion function is exponentially integrable (see \cite[Chapter 20]{astala2009elliptic}), although this is again inapplicable in our setting. An alternative approach to degenerate solutions was developed by Lehto \cite{lehto1970} (see \cite[Theorem~20.9.4]{astala2009elliptic} for a modern presentation). Our proof, building on that of \cite{ajks}, adapts elements of Lehto's argument.

Taking a slightly broader perspective, let us mention that generalised conformal weldings were introduced by Hamilton in \cite{ham91} and subsequently used by Bishop \cite{bis07} to show that every homeomorphism is `almost' a welding homeomorphism. The welding problem can also be generalised from homeomorphisms to laminations (i.e., equivalence relations on the unit circle); this was studied in a stochastic setting by Lin and Rohde \cite{lr18,lin19}.
\smallskip

\noindent\underline{Innovations:} Our approach to solving the welding problem differs in a couple of essential ways from that of \cite{ajks}. First of all, the deterministic welding argument is now performed by estimating the moduli of the inverse of an extension by using the extension itself. In this way we are led to estimating the Sobolev $H^{1/2}$-norm of our welding homeomorphism at suitable dyadic intervals. In turn, this can be done conveniently for homeomorphisms via  $\ell^2$-sums of (relative) masses of dyadic subintervals, see Lemma~\ref{l:integral} below. This is actually simpler than the use of `Lehto integrals' in \cite{ajks}, and so could be used to somewhat streamline the proof of the welding result for $\phi_1$ (i.e., the main theorem of \cite{ajks}).

However, when dealing with the homeomorphism $\phi=\phi_1^{-1}\circ \phi_2$ we face a new problem: as we are making our estimates for the `non-inverted' maps (i.e., $\Phi_1$ and $\Phi_2$), we  need to match the images of the annuli we are considering (see the right hand side of Figure~\ref{fig:AnnuliIllustration}). In principle this can be taken care of by posing extra conditions on the location of the images, but as we need to do this on all scales for a large number of annuli, the `extra' log-normal scaling factor of multiplicative chaos (i.e., the exponential term in \eqref{e:ApproxDecomp}) starts to have an effect which must be controlled carefully. This leads to the second main novelty in the techniques of the present paper: we consider an oscillating random walk algorithm in order to find enough matching random annuli, see Section~\ref{ss:algorithm} below.
\smallskip

\noindent\underline{Higher values of $\gamma$:} The main shortcoming of our result is that it is valid only for small values of $\gamma$ (which corresponds to small values of $\kappa$ for the SLE($\kappa$) loop measure by \cite{fs25}). Most parts of our argument could be extended to all subcritical values of $\gamma$, however there are two bottlenecks: first, in Proposition~\ref{p:Stationary_condition} it is necessary to match up sequences of image annuli under $\Psi_1$ and $\Psi_2$ with high probability. We do this using a simple union bound over the centres of the annuli, which necessitates strong probability bounds that hold only for $\gamma<1$ (i.e., in the $L^2$-regime of the Gaussian multiplicative chaos). Second, in Sections~\ref{ss:Sufficient}, \ref{ss:LargeDeviation} and~\ref{ss:Reduced} we give large deviation estimates for the number of image annuli satisfying various conditions. For any $\gamma\in[0,\sqrt{2})$, similar estimates hold for each of the conditions on a constant density subsequence of annuli. In order to have the intersection of these conditions hold for a large number of annuli, we need to show that the individual conditions hold for a subsequence with density close to one, which requires $\gamma$ to be close to zero.

The number of annuli for which we require the various conditions to hold is dictated by our desire to prove uniform H\"older continuity of the maps $F_n$, which in turn ensures uniqueness of the welding solution. Proving existence of the welding solution requires only equicontinuity of $F_n$ which corresponds to a much smaller number of annuli. We suspect that with a little additional work, our methods would be able to provide such an estimate for all subcritical $\gamma$, and hence prove the existence of a welding solution in this case.

While the current work was under review, Binder and Kojar completed a series of preprints \cite{bk23,bk23b,bk24} addressing the welding problem that we consider for $\gamma$ less than some explicit value. They directly apply the approach of \cite{ajks} to the inverse maps $\phi_1^{-1}$ and $\phi_2^{-1}$ which requires difficult technical estimates for the moments and decoupling of the inverse measures. Subsequently \cite{fs25} used Sheffield's result to show that the welding problem has a solution for all subcritical $\gamma$, although this result does not apply when $\phi_1$ and $\phi_2$ have different parameters.

\subsection{Acknowledgements}
The first and second authors were supported by the European Research Council (ERC) Advanced Grant QFPROBA (grant number 741487). The third author was supported by the Finnish Academy project 342183 and Center of Excellence FiRST. We thank Kari Astala, Ilya Binder, Peter Jones, Tomas Kojar, and Xin Sun for many valuable discussions. We would like to thank several anonymous referees for numerous helpful comments that improved the clarity and layout of this work and for pointing out errors in the proofs of Lemma~\ref{l:Good_annuli} and Lemma~\ref{l:Decompose}.

\section{Solving the Beltrami equation}\label{s:Beltrami}
In this section we consider a Beltrami equation which is equivalent to our welding problem. We show that this equation may be solved assuming a H\"older continuity bound for approximate solutions. Sections~\ref{s:Holder}-\ref{s:LargeDeviations} will then be dedicated to proving this bound.

We begin by stating some basic properties of the measures $\tau^{(j)}$:
\begin{lemma}\label{l:MeasureBasic}
    Let $\tau$ be defined according to \eqref{e:MeasureDef} for $\gamma\in[0,\sqrt{2})$, then there exists $c_1,c_2>0$ depending on $\gamma$ such that with probability one, for some $C_1,C_2$ and any $0\leq a<b\leq 1$
    \begin{displaymath}
        C_1\lvert b-a\rvert^{c_1}<\tau([a,b])<C_2\lvert b-a\rvert^{c_2}.
    \end{displaymath}
    Moreover for any $\gamma_0\in(0,\sqrt{2})$, if $\gamma\in[0,\gamma_0]$ and $p\in[1,2/\gamma_0^2)$ then there exists $C_{\gamma_0}>0$ such that
    \begin{displaymath}
        \E[\tau([a,b])^p]\leq C_{\gamma_0}\lvert b-a\rvert^{\zeta_p(\gamma_0)}
    \end{displaymath}
    where $\zeta_p(\gamma):=p-\frac{p^2-p}{2}\gamma^2$ denotes the multifractal spectrum of $\tau$.
\end{lemma}
\begin{proof}
    These are both essentially well-known general properties of Gaussian multiplicative chaos measures; the precise statements in our setting are given in \cite[Theorem~3.7]{ajks}. (Note that the cited result is stated only for $p\in(1,2)$, however its proof of the moment bound holds for all $p\in[1,2/\gamma_0^2)$.)
\end{proof}
We can now set up the Beltrami equation. We define the functions $\psi_1,\psi_2:\R\to\R$ by
\begin{displaymath}
    \psi_j(n+x)=n+\Theta_j+\frac{\tau^{(j)}([0,x])}{\tau^{(j)}([0,1])}\qquad\text{for }x\in[0,1), n\in\Z.
\end{displaymath}
The first part of Lemma~\ref{l:MeasureBasic} shows that each $\psi_j$ is a homeomorphism. Next we make use of the Beurling-Ahlfors extension \cite{beurling56}: for $x\in\R$ and $y\in[0,1]$ we define
\begin{equation}\label{e:BAExt}
\mathrm{Ext}(\psi_j)(x+iy)=\frac{1}{2}\int_0^1\psi_j(x+ty)+\psi_j(x-ty)\;dt+i\int_0^1\psi_j(x+ty)-\psi_j(x-ty)\;dt.
\end{equation}
We use a different definition away from the real axis, to simplify later arguments. Since $\psi_j$ is $1$-periodic, for $x\in\R$
\begin{displaymath}
\mathrm{Ext}(\psi_j)(x+i)=x+i+c_0
\end{displaymath}
where $c_0=\int_0^1\psi_j(t)\;dt-1/2$. Then by defining
\begin{displaymath}
\mathrm{Ext}(\psi_j)(x+iy)=\begin{cases}
x+iy+(2-y)c_0 &\text{if }1<y<2\\
x+iy &\text{if }y\geq 2
\end{cases}
\end{displaymath}
we see that $\mathrm{Ext}(\psi_j)$ is an extension of $\psi_j$ which is differentiable almost everywhere and equal to the identity for $\mathrm{Im}(z)\geq 2$. We then define $\Psi_1(z)=\mathrm{Ext}(\psi_1)$ and $\Psi_2(z)=\overline{\mathrm{Ext}(\psi_2)}(\overline{z})$ where $\overline{\cdot}$ denotes complex conjugation so that $\Psi_1:\overline{\H}\to\overline{\H}$ and $\Psi_2:\C\setminus\H\to\C\setminus\H$. Finally we define the homeomorphisms $\Phi_1:\overline{\D}\to\overline{\D}$ and $\Phi_2:\C\setminus\D\to\C\setminus\D$ by
\begin{equation}\label{e:ModifiedBA}
\Phi_j(z)=\exp\left(2\pi i\Psi_j\left(\frac{\log z}{2\pi i}\right)\right)\quad\text{for }j=1,2.
\end{equation}

We record here a basic property of the Beurling-Ahlfors extension which we will make use of later:
\begin{lemma}\label{l:BAderivative}
Let $D\Psi_1$ denote the derivative of $\Psi_1$. For $x\in\R$ and $y\in(0,1)$
\begin{displaymath}
\lvert D\Psi_1(x+iy)\rvert\leq\frac{4\lvert \psi_1(x+y)-\psi_1(x-y)\rvert}{y}.
\end{displaymath}
\end{lemma}
\begin{proof}
By definition of the Beurling-Ahlfors extension \eqref{e:BAExt}, a change of variables and continuity of $\psi_1$, we have
\begin{align*}
    \frac{d}{dx}\Re \Psi_1(x&+iy)=\lim_{h\searrow 0}\frac{1}{2h}\int_{-1}^1\psi_1(x+h+ty)-\psi_1(x+ty)\;dt\\
    &=\lim_{h\searrow 0}\frac{1}{2h y}\Big(\int_{x+y}^{x+y+h}\psi_1(u)\;du-\int_{x-y}^{x-y+h}\psi_1(u)\;du\Big)=\frac{\psi_1(x+y)-\psi_1(x-y)}{2y}.
\end{align*}
Very similar elementary computations show that the partial derivatives of the real/imaginary parts of $\Psi_1$ are each bounded in absolute value by $y^{-1}\lvert \psi_1(x+y)-\psi_1(x-y)\rvert$, thus proving the statement of the lemma.
\end{proof}

We now introduce some definitions from the theory of quasiconformal maps, for details and further background we refer the reader to \cite{lehto1973quasiconformal,astala2009elliptic}. Recall that for two domains $\Omega,\Omega^\prime$ the Sobolev space $W^{1,2}_\mathrm{loc}(\Omega,\Omega^\prime)$ consists of all measurable functions from $\Omega$ to $\Omega^\prime$ with a weak derivative almost everywhere such that the function and its (weak) derivative are locally square integrable. For $g\in W^{1,2}_\mathrm{loc}(\Omega,\Omega^\prime)$ we define the complex dilatation $\mu_g$ and distortion $K(\cdot,g)$ by
\begin{displaymath}
\partial_{\overline{z}}g=\mu_g\partial_zg,\quad\text{and}\quad K(z,g)=\frac{1+\lvert \mu_g(z)\rvert}{1-\lvert\mu_g(z)\rvert}\qquad\text{for a.e.\ }z\in\Omega.
\end{displaymath}
We recall also that an orientation preserving homeomorphic map $g\in W^{1,2}_\mathrm{loc}(\Omega,\Omega^\prime)$ is called $C$-quasiconformal if $\| K(z,g)\|_\infty\leq C$. A map is called quasiconformal if it is $C$-quasiconformal for some $C$.

Since $\Phi_1$ and $\Phi_2$ are $C^1$ homeomorphisms away from the circles $\lvert z\rvert=1,e^{\pm 2\pi},e^{\pm 4\pi}$, by Sard's theorem (and the inverse function theorem) their inverses are continuously differentiable almost everywhere. Hence these inverses have well defined dilatations and so we may define the measurable functions
\begin{equation}\label{e:Distortion}
\mu(z):=\begin{cases}
\mu_{\Phi_1^{-1}}(z)&\text{if }z\in\D\\
\mu_{\Phi_2^{-1}}(z)&\text{if }z\in\C\backslash\overline{\D}
\end{cases},\quad\text{and}\quad K(z):=\frac{1+\lvert\mu(z)\rvert}{1-\lvert\mu(z)\rvert}.
\end{equation}
If we could find a quasiconformal map $F:\C\to\C$ such that $\mu_F=\mu$ then a classical line of reasoning (using the Stoilow factorisation theorem) would yield a unique solution to our welding problem. Unfortunately this approach is not feasible in our case since $K$ is unbounded. Beltrami equations with unbounded distortion are known as \emph{degenerate} and there is longstanding interest in solution methods for such equations (see \cite[Chapter~20]{astala2009elliptic} or \cite[Chapter~4]{grs2012} and references therein).

The link between conformal welding and the Beltrami equation persists in the degenerate case: if $K$ is locally bounded on $\C\setminus\partial\D$ and there exists a homeomorphic map $F\in W^{1,1}_\mathrm{loc}(\C,\C)$ for which $\mu_F=\mu$ almost everywhere then we can construct a solution to the welding problem for $\phi_1^{-1}$ and $\phi_2^{-1}$. To find such an $F$, we adapt an approach which originated with Lehto in \cite{lehto1970} (and was used in a context similar to ours in \cite{ajks}). We shall define solutions $F_n$ to the Beltrami equation for approximations of $\mu$ with bounded distortion and apply the Arzel\`a-Ascoli theorem to take subsequential limits. The following three lemmas will be required to justify equicontinuity (as an input to Arzel\`a-Ascoli).

\begin{lemma}\label{l:Distortion}
For $\gamma\in[0,\sqrt{2})$ with probability one $K\in L^\infty_\mathrm{loc}(\C\setminus\partial\D)$ uniformly over $\Theta_1,\Theta_2$.
\end{lemma}
\begin{proof}
Since $\Phi_1,\Phi_2$ and their inverses are continuously differentiable almost everywhere, a simple calculation using the chain rule shows that
\begin{displaymath}
\lvert \mu(z)\rvert=\begin{cases}
\lvert\mu_{\Phi_1}\circ\Phi_1^{-1}(z)\rvert &\text{for a.e.\ }z\in\D\\
\lvert\mu_{\Phi_2}\circ\Phi_2^{-1}(z)\rvert &\text{for a.e.\ }z\in\C\backslash\D,
\end{cases}\end{displaymath}
and hence
\begin{displaymath}
K(z)=\begin{cases}
K(\Phi_1^{-1}(z),\Phi_1) &\text{for a.e.\ }z\in\D\\
K(\Phi_2^{-1}(z),\Phi_2)&\text{for a.e.\ }z\in\C\backslash\D.
\end{cases}
\end{displaymath}
The result now follows from arguments in \cite{ajks} where the distortion of $\Phi_1$ is bounded above by a discrete approximation which can be controlled by moment estimates for Gaussian multiplicative chaos. Specifically \cite[Equation (25) and Theorem 2.6]{ajks} shows that $K(\cdot,\Phi_1)$ is almost surely in $L^\infty_\mathrm{loc}(\C\setminus\partial\D)$. The same is true of $K(z,\Phi_2)$ since distortion is preserved under reflection in the unit circle (and $\Phi_2$ has the same distribution as $\Phi_1$). The upper bounds depend only on the measures $\tau^{(j)}$ for $j=1,2$ and so hold uniformly over realisations of $\Theta_1$ and $\Theta_2$.
\end{proof}

\begin{lemma}\label{l:Distortion2}
For $\gamma\in[0,\sqrt{2})$ with probability one $K\in L^1_\mathrm{loc}(\C)$.
\end{lemma}
\begin{proof}
    By Lemma~\ref{l:Distortion} it is enough to show that $K$ is integrable on some neighbourhood of $\partial\D$. Let $\mathcal{A}$ be the annulus $\{e^{-2\pi}\leq\lvert z\rvert\leq 1\}$. By a change of variables (which is valid since $\Phi_1$ is $C^1$ away from the boundary of $\mathcal{A}$)
\begin{equation}\label{e:Distortion1}
\begin{aligned}
    \int_{\Phi_1(\mathcal{A})}K(z)\;dm(z)=\int_{\Phi_1(\mathcal{A})}K(\Phi_1^{-1}(z),\Phi_1)\;dm(z)&=\int_{\mathcal{A}}K(w,\Phi_1)J(w,\Phi_1)\;dm(w)\\
&=\int_{\mathcal{A}}\lvert D\Phi_1(w)\rvert^2\;dm(w)
\end{aligned}
\end{equation}
where $J(w,\Phi_1)=\lvert \partial_z\Phi_1\rvert^2-\lvert \partial_{\overline{z}}\Phi_1\rvert^2$ denotes the Jacobian of $\Phi_1$ and $m$ denotes Lebesgue measure on $\C$. Recall that $\Phi_1=e\circ\Psi_1\circ e^{-1}$. Since $\lvert D e\rvert$ and $\lvert D e^{-1}\rvert$ are bounded by an absolute constant $c>0$ on $\mathcal{A}$, by another change of variables we conclude that
\begin{equation}\label{e:Distortion2}
\begin{aligned}
    \int_{\mathcal{A}}\lvert D\Phi_1(w)\rvert^2\;dm(w)&\leq c\int_{[0,1]\times[0,1]}\lvert D\Psi_1(x+iy)\rvert^2\;dxdy\\
    &\leq 16c\int_{[0,1]\times[0,1]}\frac{\lvert\psi_1(x+y)-\psi_1(x-y)\rvert^2}{y^2}\;dxdy\\
    &\leq \frac{16c}{\tau^{(1)}([0,1])^2}\int_{[0,1]\times[0,1]}\frac{\tau^{(1)}([x-y,x+y])^2}{y^2}\;dxdy
\end{aligned}
\end{equation}
where the second inequality uses Lemma~\ref{l:BAderivative}. Using monotonicity of $\tau^{(1)}$
\begin{multline}\label{e:Distortion3}
\int_{[0,1]\times[0,1]}\frac{\tau^{(1)}([x-y,x+y])^2}{y^2}\;dxdy\\
\begin{aligned}
    &\leq\sum_{n=0}^\infty 2^{2(n+1)}\int_{[0,1]}\int_{[2^{-(n+1)},2^{-n}]}\tau^{(1)}([x-y,x+y])^2\;dydx\\
    &\leq\sum_{n=0}^\infty 2^{n+2}\int_{[0,1]}\tau^{(1)}([x-2^{-n},x+2^{-n}])^2\;dx.
\end{aligned}
\end{multline}
Splitting the latter integral into a sum over dyadic intervals and using monotonicity of $\tau^{(1)}$ once more
\begin{align*}
    \int_{[0,1]}\tau^{(1)}([x-2^{-n},x+2^{-n}])^2\;dx&=\sum_{k=1}^{2^n}\int_{[(k-1)2^{-n},k2^{-n}]}\tau^{(1)}([x-2^{-n},x+2^{-n}])^2\;dx\\
    &\leq\sum_{k=1}^{2^n} 2^{-n}\tau^{(1)}([(k-2)2^{-n},(k+1)2^{-n}])^2\\
    &\leq 2^{-n}\cdot 3\tau^{(1)}([0,1])\sup_{k=1,\dots,2^{n}}\tau^{(1)}([(k-2)2^{-n},(k+1)2^{-n}])
\end{align*}
where the final inequality uses additivity of $\tau^{(1)}$. By Lemma~\ref{l:MeasureBasic} we see that with probability one this expression is bounded by $C 2^{-n}\tau^{(1)}([0,1])\cdot 2^{-cn}$ for some $C,c>0$ independent of $n$. Hence by \eqref{e:Distortion3}
\begin{equation}\label{e:Distortion4}
    \int_{[0,1]\times[0,1]}\frac{\tau^{(1)}([x-y,x+y])^2}{y^2}\;dxdy\leq 4C\tau^{(1)}([0,1])\sum_{n=0}^\infty 2^{-cn}<\infty.
\end{equation}
Combining \eqref{e:Distortion1}, \eqref{e:Distortion2} and \eqref{e:Distortion4} shows that $K$ is almost surely integrable on $\Phi_1(\mathcal{A})=\mathcal{A}$. Replacing $\mathcal{A}$ with the annulus $\{1\leq\lvert z\rvert\leq e^{2\pi}\}$ we can give an identical argument for $\Phi_2$ which proves that $K$ is almost surely integrable on a neighbourhood of $\partial\D$ as required.
\end{proof}

\begin{lemma}\label{l:DistortionDiam}
There exists a homeomorphism $\eta:[0,\infty)\to[0,\infty)$ such that if $f$ is any quasiconformal mapping defined on the annulus $A:=\{w:\lvert w-w_0\rvert\in(r,R)\}$, then
\begin{displaymath}
\frac{\mathrm{diam}_O f(A)}{\mathrm{diam}_I f(A)}\geq\eta\left(\frac{\log(R/r)}{\|K(\cdot,f)\|_{\infty}}\right)
\end{displaymath}
where $\mathrm{diam}_O$ and $\mathrm{diam}_I$ denote the outer and inner diameters of an annulus respectively.
\end{lemma}
Before proving this lemma, we briefly recall the notion of conformal modulus for a family of curves which will be used here and elsewhere in our arguments. For background and further details, see \cite{ahlfors73,lehto1973quasiconformal}. Let $\Gamma$ be a family of locally rectifiable curves with images in some domain $D\subset\C$. The conformal modulus of $\Gamma$, denoted $\mod\;\Gamma$, is given by
\begin{displaymath}
\inf_\rho\int_D\rho^2\;dxdy
\end{displaymath}
where the infimum is taken over all Borel-measurable functions $\rho\geq 0$ such that $\int_\gamma\rho\;\lvert dz\rvert\geq 1$ for all $\gamma\in\Gamma$. The conformal modulus of a curve family is invariant under conformal maps (indeed this was the motivation for its introduction) and more generally if $g$ is $C$-quasiconformal on $D$ then
\begin{displaymath}
\frac{1}{C}\mod\;\Gamma\leq\mod\;g(\Gamma)\leq C\mod\;\Gamma
\end{displaymath}
where $g(\Gamma)=\{g\circ\gamma\;|\;\gamma\in\Gamma\}$. For a topological annulus $A^\prime\subset\C$, we define $\mod\;A^\prime$ to be the modulus of all rectifiable curves in $A^\prime$ which have a winding number of one around any point in the bounded component of the complement of $A^\prime$. For $A=\{w\;|\;r<\lvert w\rvert<R\}$ a direct computation shows that $\mod\;A=2\pi\log(R/r)$.

\begin{proof}[Proof of Lemma~\ref{l:DistortionDiam}]
Combining \cite[Corollary 7.39]{vuorinen88} and \cite[Exercise 5.68 (16)]{anderson97} shows that for any bounded topological annulus $A^\prime\subset\C$
 \begin{displaymath}
\frac{\mathrm{diam}_O A^\prime}{\mathrm{diam}_I A^\prime}\geq\frac{1}{16}e^{\pi\mod\; A^\prime}.
\end{displaymath}
Since the modulus of $A$ is $2\pi\log(R/r)$ and $f$ is quasiconformal
\begin{displaymath}
\mod\;f(A)\geq\frac{2\pi}{\|K(\cdot,f)\|_{\infty}}\log(R/r)
\end{displaymath}
which implies the conclusion of the lemma on setting $A^\prime=f(A)$.
\end{proof}

In order to prove uniqueness of our welding solution, we will require the following criterion for conformal removability due to Jones and Smirnov. We recall that a set $A\subset\C$ is said to be \emph{conformally removable} if any homeomorphism $g:\C\to\C$ which is conformal on $\C\setminus A$ is conformal on $\C$. Let $A\subset\C$ be a proper simply connected domain, we say that $A$ is a \emph{H\"older domain} if the Riemann map from $\D$ to $A$ can be extended to a H\"older continuous map on $\overline{\D}$.

\begin{theorem}[{\cite[Corollary~2]{jones2000removability}}]\label{t:Removable}
The boundary of a H\"older domain is conformally removable.
\end{theorem}

Applying this to our setting, uniqueness of the welding solution will follow if we can prove that our approximate solutions to the Beltrami equation are uniformly H\"older continuous. The following result verifies this.

We define the sequence of quasiconformal maps $F_n$ by
\begin{equation}\label{e:QCApprox}
\partial_{\overline{z}} F_n=\frac{n}{n+1}\mu\;\partial_z F_n\qquad\text{for a.e.\ }z\in\C.
\end{equation}
Since the complex dilatation $\frac{n}{n+1}\mu$ is uniformly bounded away from $1$, classical results state that this equation can be solved for each $n\in\N$. Specifically \cite[Theorem~5.3.2]{astala2009elliptic} states that for each $n$ there exists a unique quasiconformal homeomorphism $F_n:\C\to\C$ solving the above equation such that
\begin{equation}\label{e:QCApprox2}
F_n(z)=z+O(1/z)\quad\text{as }z\to\infty.
\end{equation}

\begin{theorem}\label{t:Holder}
There exists $\gamma_0\in(0,\sqrt{2})$ such that for $\gamma\in[0,\gamma_0]$ the collection of maps $F_n$ is uniformly H\"older continuous on $\partial\D$ with probability one.
\end{theorem}

The proof of this result, given in Sections~\ref{s:Holder}-\ref{s:LargeDeviations}, is the most technically challenging part of our work.

To prove our main result, we require one more input: a convergence result for quasiconformal maps with convergent complex dilatations. These arguments are extracted from the proof of \cite[Theorem~20.9.4]{astala2009elliptic}.

\begin{proposition}\label{p:dilatation_convergence}
    Let $(G_n)_{n\in\N}$ be a sequence of quasiconformal homeomorphisms of $\C$ normalised so that $G_n(z)=z+O(1/z)$ as $z\to\infty$ for each $n$. Suppose that
    \begin{enumerate}
        \item $\{K(\cdot,G_n)\;|\;n\in\N\}$ is uniformly integrable on a neighbourhood of $\partial\D$,
        \item $\sup_n K(\cdot,G_n)\in L^\infty_\mathrm{loc}(\C\setminus\partial\D)$,
        \item the $G_n$ are equicontinuous on $\partial\D$,
        \item $\mu_{G_n}\to \mu_\infty$ pointwise for some measurable function $\mu_\infty$,
        \item the dilatations $\mu_{G_n}$ are compactly supported uniformly in $n$.
    \end{enumerate}
    Then as $n\to\infty$, $G_n$ converges locally uniformly to a homeomorphism $G\in W^{1,1}_\mathrm{loc}(\C,\C)$ such that $\mu_G=\mu_\infty$.
\end{proposition}

\begin{proof}
We wish to prove equicontinuity of the sequences $G_n$ and $G_n^{-1}$. We choose $R_0>0$ such that $R_0\D$ contains the support of $\mu_{G_n}$ for all $n$. By applying \cite[Lemma~20.2.3]{astala2009elliptic} to $z\mapsto R_0^{-1}G_n(R_0z)$ (which has dilatation supported on $\D$) we obtain that for any $z,w\in\C$
\begin{displaymath}
\lvert G_n^{-1}(z)-G_n^{-1}(w)\rvert^2\leq\frac{16\pi^2}{\log(e+R_0/\lvert z-w\rvert)}\left(\lvert z\rvert^2+\lvert w\rvert^2+\int_{R_0\D} K(\cdot,G_n)\;dm\right).
\end{displaymath}
By assumptions (1) and (2), the integral here is bounded uniformly over $n$ and so we deduce the requisite equicontinuity for $G_n^{-1}$.

We now turn to proving equicontinuity of the $G_n$. By assumption, this holds on $\partial\D$, so we consider a point $z\in\C\setminus\partial\D$ and choose $R$ such that $B(z,2R)\subset\C\setminus\partial\D$. Since $\sup_nK(\cdot,G_n)$ is locally bounded, by Lemma~\ref{l:DistortionDiam} we see that
\begin{equation}\label{e:equicontinuity}
\frac{\mathrm{diam} G_n(B(z,R))}{\mathrm{diam} G_n(B(z,r))}\to\infty
\end{equation}
uniformly in $n$ as $r\to 0$. Since the $G_n$ are conformal for $\lvert z\rvert>R_0$ and normalised by $G_n(z)=z+O(1/z)$ at infinity, Koebe's 1/4 theorem shows that $\mathrm{diam} G_n(B(z,R))$ is bounded above uniformly in $n$. Hence by \eqref{e:equicontinuity}, $\sup_n\mathrm{diam}G_n(B(z,r))\to 0$ as $r\to0$ which proves equicontinuity at $z$.

Applying the Arzel\`a-Ascoli theorem to $G_n$ and $G_n^{-1}$ we may pass to a subsequence and find a continuous map $G:\C\to\C$ such that $G_n\to G$ and $G_n^{-1}\to G^{-1}$ locally uniformly. Consider now the derivatives of $G_n$; we have the elementary distortion identity 
\begin{equation}\label{e:DerivConv}
\lvert D G_n(z)\rvert=\sqrt{K(z,G_n)J(z,G_n)}
\end{equation}
where we recall that $D G_n$ denotes the derivative of $G_n$ and $J(\cdot,G_n)$ denotes its Jacobian. Since $G_n$ is analytic outside the disc of radius $R_0$ and has the expansion $G_n(z)=z+O(1/z)$ at infinity, the area formula \cite[Theorem~2.10.1]{astala2009elliptic} implies that $\int_{B(0,R)}J(z,G_n)\;dm(z)$ is bounded above uniformly over $n$ for any $R\geq R_0$. Applying H\"older's inequality to \eqref{e:DerivConv}, and using assumptions (1) and (2), shows that $DG_n$ is uniformly integrable on any compact domain. Hence passing to a further subsequence, $DG_n$ converges weakly in $L^1_\mathrm{loc}(\C)$ to a limit (this result is sometimes known as the Dunford-Pettis theorem) which must be the derivative of $G$ and so $G\in W^{1,1}_\mathrm{loc}(\C)$.

Next we must show that $\mu_G=\mu_\infty$ almost everywhere. Let $\phi\in C_0^\infty(\C)$, then by definition of $G_n$ and weak convergence of the derivatives of $G_n$
\begin{equation}\label{e:DilatationConv}
\int_\C\phi\;(\mu_{G_n}-\mu_\infty)\partial_z G_n\;dm=\int_\C\phi\; (\partial_{\overline{z}}G_n-\mu_\infty\partial_z G_n)\;dm\to \int_\C \phi\;(\partial_{\overline{z}}G-\mu_\infty\partial_z G)\;dm
\end{equation}
as $n\to\infty$. Again using \eqref{e:DerivConv} and H\"older's inequality
\begin{align*}
    \Big\lvert\int_\C\phi\;(\mu_{G_n}-\mu_\infty)\partial_z &G_n\;dm\Big\rvert\leq\int_{R_0\D}\lvert\phi\rvert\;\lvert\mu_{G_n}-\mu_\infty\rvert\;\lvert DG_n\rvert\;dm\\
    &=\int_{R_0\D}\lvert\phi\rvert\;\lvert\mu_{G_n}-\mu_\infty\rvert\;\sqrt{K(\cdot,G_n) J(\cdot,G_n)}\;dm\\
    &\leq\|\phi\|_\infty\Big(\int_{R_0\D}\lvert\mu_{G_n}-\mu_\infty\rvert^2\; K(\cdot,G_n)\;dm\Big)^{1/2}\Big(\int_{R_0\D}J(\cdot,G_n)\;dm\Big)^{1/2}.
\end{align*}
Applying Vitali's convergence theorem to the first integral in this upper bound (and recalling that the second is uniformly bounded over $n$ by the area theorem) we see that the left hand side of \eqref{e:DilatationConv} converges to zero. Hence $G$ satisfies the desired Beltrami equation.
\end{proof}

We are now able to prove our main result:

\begin{proof}[Proof of Theorem~\ref{t:Welding} (assuming Theorem~\ref{t:Holder})]
By Theorem~\ref{t:Holder} the sequence $F_n$ almost surely satisfies a uniform H\"older bound on $\partial\D$. We fix a realisation of $\phi_1,\phi_2$ for which this bound holds and also Lemmas~\ref{l:Distortion} and~\ref{l:Distortion2} apply. Hereafter we can argue in a purely deterministic way. Setting $G_n:=F_n$ we see that all of the conditions for Proposition~\ref{p:dilatation_convergence} are satisfied with $\mu_\infty=\mu$. Hence we deduce the existence of a homeomorphic $F:\C\to\C$ in $W^{1,1}_\mathrm{loc}(\C)$ which has complex dilatation $\mu$.

We now show that existence of such an $F$ implies the existence of a unique solution to our specified welding problem. We define the maps
\begin{align*}
f_1&:=F\circ\Phi_1:\overline{\D}\to F(\overline{\D})\\
f_2&:=F\circ\Phi_2:\C\backslash\D\to F(\C\backslash\D)
\end{align*}
and observe that
\begin{displaymath}
f_1\circ\phi_1^{-1}=F|_{\partial\D}=f_2\circ\phi_2^{-1}.
\end{displaymath}
Since $K$ is locally bounded on $\C\setminus\partial\D$ we can apply the Stoilow factorisation theorem \cite[Theorem~5.5.1]{astala2009elliptic} locally to conclude that $f_1$ and $f_2$ are holomorphic and hence solve the welding problem for $\phi_1^{-1}\circ\phi_2$.

To see that the solution is unique (up to M\"obius transformations), observe that by H\"older continuity of $F|_{\partial\D}$ and the Jones-Smirnov theorem (Theorem~\ref{t:Removable}) the image $F(\partial\D)$ is conformally removable. Suppose that $g_1:\overline{\D}\to\overline{\Omega}$ and $g_2:\C\setminus\D\to\C\setminus\Omega$ is another solution to the welding problem, then the function defined by
\begin{displaymath}
\begin{cases}
g_1\circ f_1^{-1}(z) &\text{if }z\in F(\overline{\D})\\
g_2\circ f_2^{-1}(z) &\text{if }z\in F(\C\backslash\D)
\end{cases}
\end{displaymath}
is a homeomorphism of $\C$ and conformal off $F(\partial \mathbb{D})$. By conformal removability of the latter set, this map is a M\"obius transformation, as required.
\end{proof}

\section{Conditions for H\"older continuity}\label{s:Holder}
In this section we derive a probabilistic statement for events defined by $\tau^{(1)}$ and $\tau^{(2)}$ which implies the conclusion of Theorem~\ref{t:Holder}. (This statement is then proven in Sections~\ref{s:Decompose} and~\ref{s:LargeDeviations}). We proceed in three steps:
\begin{enumerate}
    \item First we show that H\"older continuity of $F_n$ follows if for any point $x\in\partial\D$, there exists a suitable sequence of concentric annuli $\mathbb{A}_i$ around $x$ such that $F_n(\mathbb{A}_i)$ has conformal modulus bounded away from zero.
    \item Next we use properties of the conformal modulus and the Beurling-Ahlfors extension to derive conditions on $\tau^{(1)}$ and $\tau^{(2)}$ which ensure that such annuli exist.
    \item Finally we use a Borel-Cantelli argument to quantify the probability bounds required to deduce almost sure uniform H\"older continuity.
\end{enumerate}

\subsection{H\"older continuity via `good' annuli}
The homeomorphisms $F_n$, for which we wish to prove uniform H\"older continuity, are characterised by their distortions, which are bounded above by $K$. Therefore the proof of Theorem~\ref{t:Holder} requires us to translate bounds on the distortion of a function into bounds on its modulus of continuity. We make this link by considering how the function behaves on small annuli centred around points of $\partial\D$. More precisely if we can show that the image of many such annuli have conformal modulus bounded away from zero, then standard estimates allow us to conclude that the function is H\"older continuous on $\partial\D$. This method is implicit in the approach of Lehto \cite{lehto1970} to solving degenerate Beltrami equations. The following deterministic lemma makes these ideas rigorous:

We say that a topological annulus $\mathbb{A}\subset\C$ \emph{surrounds} a set $U\subset\C$ if $U$ is contained in the bounded component of $\mathbb{A}^c$. We denote by $B(z,r)$ the open ball in $\C$ of radius $r$ centred at $z$.

\begin{lemma}\label{l:Annuli-Holder}
Let $c_1,c_2,c_3,\delta>0$ and let $f:\C\to\C$ be a homeomorphism. Suppose that for some $N_0\in\N$ and all $N>N_0$ the following holds: for each $z\in\partial\D$ there exists a disjoint sequence of topological annuli $\mathbb{A}_1,\dots,\mathbb{A}_{\delta N}$, such that
\begin{enumerate}
    \item $\mathbb{A}_k$ surrounds $\mathbb{A}_{k+1}$ for all $k$
    \item $\mathbb{A}_{\delta N}$ surrounds $B(z,e^{-c_1N})$
    \item The outer diameter of $f(\mathbb{A}_1)$ is at most $c_2$
    \item For each $k$, $\mod f(\mathbb{A}_k)>c_3$
\end{enumerate}
then there exists an open neighbourhood $E$ of $\partial\D$ and $C_1,C_2>0$ depending only on $c_1$, $c_2$, $c_3$, $\delta$ and $N_0$ such that
\begin{displaymath}
\lvert f(z_1)-f(z_2)\rvert\leq C_1\lvert z_1-z_2\rvert^{C_2}
\end{displaymath}
for all $z_1\in\partial\D$ and $z_2\in E$.
\end{lemma}
\begin{remark}
    In the statement above, $\delta N$ need not be an integer and so $\mathbb{A}_{\delta N}$ should rigorously be interpreted as $\mathbb{A}_{\lfloor \delta N\rfloor}$ where $\lfloor\delta N\rfloor$ is the integer part of $\delta N$. To ease notation we use this convention throughout the paper for integer values indices.
\end{remark}
\begin{proof}
For $z\in\partial\D$ and $N>N_0$ let $\hat{\mathbb{A}}$ be the smallest topological annulus containing $\mathbb{A}_1,\dots,\mathbb{A}_{\delta N}$ (as defined in the statement of the proposition). In other words, the complement of $\hat{\mathbb{A}}$ is made up of the unbounded component of $\mathbb{A}_1^c$ and the bounded component of $\mathbb{A}_{\delta N}^c$. Since $f(\hat{\mathbb{A}})$ contains $f(\mathbb{A}_k)$ for each $k=1,\dots,{\delta N}$ the standard composition law for conformal modulus \cite[Theorem~4.2]{ahlfors73} implies that $\mod f(\hat{\mathbb{A}})\geq\sum_k\mod f(\mathbb{A}_k)$. By assumption, the latter is at least $c_3(\delta N-1)$. We recall from the proof of Lemma~\ref{l:DistortionDiam} that for any topological annulus $A^\prime\subset\C$
\begin{displaymath}
\frac{\mathrm{diam}_O A^\prime}{\mathrm{diam}_I A^\prime}\geq\frac{1}{16}e^{\pi\mod A^\prime}.
\end{displaymath}
Applying the previous estimate to $A^\prime=f(\hat{\mathbb{A}})$ we have
\begin{displaymath}
\diam\; f\big(B(z,e^{-c_1N})\big)\leq\diam_I f(\hat{\mathbb{A}}(j))\leq 16 e^{-\pi c_3(\delta N-1)}\diam_O(f(\mathbb{A}_1))\leq 16 c_2e^{-\pi (c_3\delta N-1)}.
\end{displaymath}
This implies the statement of the lemma whenever $e^{-c_1 (N+1)}\leq\lvert z_1-z_2\rvert\leq e^{-c_1 N}$ for some $N\geq N_0$. Applying the above estimate repeatedly yields that for any $z_1,z_2$ with distance at most $e^{-c_1N_0}$ from $\partial\D$, $\lvert f(z_1)-f(z_2)\rvert$ is bounded by a constant depending only on $c_2$, $c_3$, $\delta$ and $N_0$. We thus obtain the statement of the lemma in the remaining case that $\lvert z_1-z_2\rvert>e^{-c_1N_0}$.
\end{proof}

Our aim is to apply this lemma to $f=F_n$ and so we might wonder how best to choose sequences of annuli satisfying the conditions of the lemma?

One might initially choose the $\mathbb{A}_i$ of the previous lemma to be regular annuli centred at equally spaced points on $\partial\D$. However since the distortion $K$ is defined in terms of $\Phi_1^{-1},\Phi_2^{-1}$ this would require us to find the inverse images of these regular annuli under random homeomorphisms, and then control the distortion of $\Phi_1,\Phi_2$ on these sets. This seems difficult if not unfeasible.

Instead we consider deterministic families of regular annuli centred at points of $\partial\D$ and then we map them under $\Phi_1,\Phi_2$ to obtain random annuli on which it is easier to control the distortion $K$ (see Figure~\ref{fig:AnnuliIllustration}). This introduces a new difficulty; if $A$ is some topological annulus which intersects both components of $\C\setminus\partial\D$ and we map it using $\Phi_1$ on $\D$ and $\Phi_2$ on $\C\setminus\overline{\D}$ then its image will no longer be a topological annulus in general. Our method will be to consider families of `half-annuli' inside and outside $\D$ and allow for the possibility that the images under $\Phi_1$ and $\Phi_2$ respectively of different families will match up to form `composite annuli'. If we could find enough such composite annuli around each point of $\partial\D$ (such that their images under $F_n$ have conformal modulus bounded away from zero) then we could apply Lemma~\ref{l:Annuli-Holder} to deduce H\"older continuity.

We now describe these constructions more precisely. For notational convenience we will find it simpler to define our half-annuli on $\C\setminus\R$ and work with $\Psi_1,\Psi_2$ before mapping back to $\C\setminus\partial\D$ (see Figure~\ref{fig:AnnuliIllustration} again). Specifically we denote
\begin{align*}
D[x,r]=[x-r,x+r]\times[0,r],\qquad A_t(x)=\mathrm{Clos}\big(D[x,\rho^{t}]\backslash D[x,\rho^{t+1/4}])
\end{align*}
where $\mathrm{Clos}(\cdot)$ denotes the topological closure and $\rho\in(0,1)$ will be specified later. Furthermore let $\widetilde{D}[x,r]$ and $\widetilde{A}_t(x)$ be the reflections of $D[x,r]$ and $A_t(x)$ respectively in $\R$. Recall that we denote the map $z\mapsto e^{2\pi iz}$ by $e(\cdot)$.

\begin{definition}\label{d:Ann}
    For $x,y,\in[0,1]$ and $s,t\geq 0$ let $\mathrm{Ann}(x,t,y,s,N)$ be the event that there exists a topological annulus $\mathbb{A}$ such that
\begin{enumerate}
    \item $\mathbb{A}\subseteq\Psi_1(A_{t}(x))\cup\Psi_2(\widetilde{A}_{s}(y))$,
    \item $\mathbb{A}$ surrounds $\Psi_1(D[x,\rho^{5N+1}])\cup \Psi_2(D[y,\rho^{5N+1}])$ and
    \item $\mathrm{mod}\;F_n\circ e(\mathbb{A})>2^{-70}$ for all $n$.
\end{enumerate}
See Figure~\ref{fig:MatchingAnnuli}.
\end{definition}

\begin{figure}[h]
    \centering
    \begin{tikzpicture}[use fpu reciprocal,scale=0.6]
\pgfmathsetseed{1}
\node at (-4,4) {$\mathbb{C}\setminus\mathbb{R}$};
\draw[] (-4,0)--(4,0);
\filldraw[black] (-1.5,0) circle (2pt) node[anchor=north]{$x$};
\filldraw[black] (-1.5,0) circle (2pt);
\filldraw[black] (1.5,0) circle (2pt) node[anchor=south]{$y$};
\begin{scope}[scale=0.5,shift={(-3,0)}]
\draw [pattern=north west lines,pattern color=gray](-3,0)--(-3,3)--(3,3)--(3,0)--(1,0)--(1,1)--(-1,1)--(-1,0)--(-3,0);
\end{scope}
\begin{scope}[scale=0.3,shift={(5,0)}]
\draw [pattern=north west lines,pattern color=gray](-3,0)--(-3,-3)--(3,-3)--(3,0)--(1,0)--(1,-1)--(-1,-1)--(-1,0)--(-3,0);
\end{scope}
\node[right] at (2,2) {$A_t(x)$};
\draw (2,2)--(-0.5,1);
\node[right] at (2.9,-1.3) {$\widetilde{A}_s(y)$};
\draw (2.9,-1.3)--(2.2,-0.7);

\draw[->] (6,0) -- node [text width=2.5cm,midway,above] {$\begin{cases}\Psi_1&\text{on }\mathbb{H}\\\Psi_2&\text{on } \mathbb{C}\setminus\overline{\mathbb{H}}\end{cases}$} (8,0);

\begin{scope}[shift={(14,0)}]
\node at (-4,4) {$\mathbb{C}\setminus\mathbb{R}$};
\draw (-4,0)--(4,0);
\filldraw[black] (-0.1,0) circle (2pt);
\filldraw[black] (0.1,0) circle (2pt);

\begin{scope}
\clip (-5,0) rectangle (5,4);
\draw[pattern=north west lines, pattern color = gray, decorate,decoration={random steps,segment length=1pt,amplitude=1pt}] (-1.5,-0.2) to ++ (-0.2,0.5) to ++(0.4,0.7) to ++(0,0.5) to ++(1,-0.3) to ++(0.5,+0.4)to++(1.5,-0.2)
to ++(-0.3,-0.8)
to ++(0.3,-0.9)
to ++(-1,0) 
to ++(0.,+0.7)
to ++(-0.5,0) to ++(-0.7,0.2)
to ++(-0.2,-0.7)
to ++(-1,-0.1);
\end{scope}
\begin{scope}[shift={(-0.1,0)}]
\clip (-5,-4) rectangle (5,0);
\draw[pattern=north west lines, pattern color = gray, decorate,decoration={random steps,segment length=1pt,amplitude=1pt}] (-1.7,0.2)--(-1.9,-0.8)--(-1.7,-1.2)--(-1.2,-1.4)--(0.5,-1.7)--(0.8,-1.5)--(1.3,-1.4)--(1.7,-1.4)--(2,0.4)--(0.7,0.2)--(0.7,-0.3)--(0.5,-0.4)--(0,-0.3)--(-0.4,-0.2)--(-0.4,0.2)--cycle;
\end{scope}

\draw[fill, color=gray,opacity=0.5] (0.9,0)--(0.9,0.7)--(-0.8,0.7)--(-0.8,0)--(-1.2,0)--(-1.2,1.1)--(1.3,1.1)--(1.3,0)--cycle;
\draw[fill, color=gray,opacity=0.5] (0.9,0)--(0.9,-0.7)--(-0.8,-0.7)--(-0.8,0)--(-1.2,0)--(-1.2,-1.1)--(1.3,-1.1)--(1.3,0)--cycle;

\node[above] at (0.7,2.2){$\Psi_2(y)$};
\draw (0.7,2.2)--(0,0.1);
\node[right] at (2.5,-2){$\Psi_1(x)$};
\draw (2.5,-2)--(0.15,-0.1);
\node[left] at (-2.5,-2){$\mathbb{A}$};
\draw (-2.5,-2)--(-1,-1);
\draw (-2.5,-2)--(-0.9,0.8);
\end{scope}

\end{tikzpicture}

    \caption{An illustration of the event $\mathrm{Ann}(x,t,y,s,N)$. The images of the two half-annuli must `match up' to contain a topological annulus $\mathbb{A}$ which surrounds $\Psi_1(D[x,\rho^{5N+1}])\cup \Psi_2(D[y,\rho^{5N+1}])$. We will later construct $\mathbb{A}$ with straight edges as illustrated, although this is not important for our arguments.}
    \label{fig:MatchingAnnuli}
\end{figure}
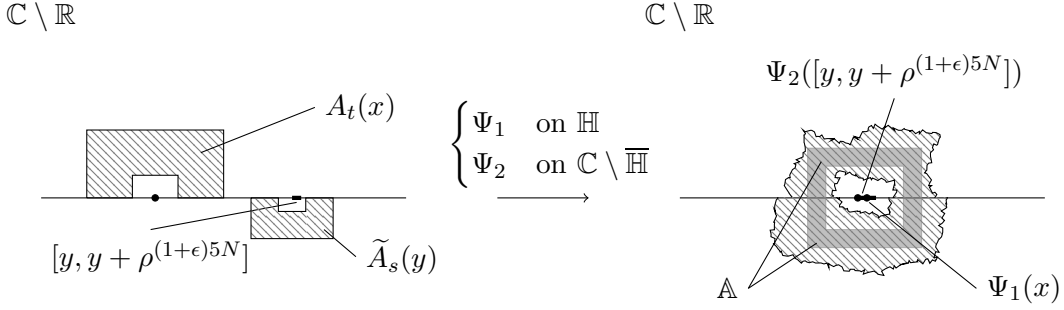

\subsection{Constructing annuli}
In order to apply Lemma~\ref{l:Annuli-Holder} we need to find conditions on $\psi_1$ and $\psi_2$ which ensure that $\mathrm{Ann}(x,t,y,s,N)$ occurs.

The first order of business is constructing topological annuli. The Beurling-Ahlfors extension allows us to give some basic geometric estimates on the images of (half)-annuli without much difficulty:
\begin{lemma}\label{l:thickness}
Let $g:\R\to\R$ be an increasing homeomorphism and $G$ its Beurling-Ahlfors extension (either the original version defined in \cite{beurling56} or the adjusted version defined in Section~\ref{s:Beltrami}; recall that the latter definition requires $x\mapsto g(x)-x$ to be $1$-periodic). Recalling that $D[x,r]=[x-r,x+r]\times[0,r]$ we have for $0<r<R\leq 1$
\begin{displaymath}
G\big(D[x,R]\setminus D[x,r]\big)\supseteq D[g(x),R^\prime]\setminus D[g(x),r^\prime]
\end{displaymath}
where
\begin{align*}
R^\prime&=\frac{1}{2}\min_{k=1,\dots,8}g\left(x-R+\frac{kR}{4}\right)-g\left(x-R+\frac{(k-1)R}{4}\right)\\
r^\prime&=g(x+2r)-g(x-2r).
\end{align*}
\end{lemma}

\begin{proof}
Without loss of generality, we may assume $x=0$. From the definition of the Beurling-Ahlfors extension in \eqref{e:BAExt}, we see that for $u+iv\in D[0,r]$ we have
\begin{align*}
    \lvert\Re G(u+iv)-g(0)\rvert&\leq\max\{g(2r)-g(0),g(0)-g(-2r)\}\\
    \Im G(u+iv)&\leq g(2r)-g(-2r).
\end{align*}
Combining these estimates proves that $G(D[x,r])\subseteq D[g(x),r^\prime]$.

Next we observe, again by the definition in \eqref{e:BAExt}, that for $u\in\R$ and $v\in[0,1]$
\begin{displaymath}
\Im G(u+iv)\geq\frac{1}{2}(g(u+v/2)-g(u-v/2)).
\end{displaymath}
If $u+iv\in[(k-5)R/4+iR/2,(k-4)R/4+iR/2]$ for some $k=1,\dots,8$ then since $g$ is increasing, we see $\Im G(u+iv)\geq\frac{1}{2}(g((k-4)R/4)-g((k-5)R/4))$. Therefore on the line segment $[-R+iR/2,R+iR/2]$, the imaginary part of $G$ is bounded below by $R^\prime$.

Finally we note that on $[R,R+iR/2]$ the real part of $G$ is bounded below by $g(R/2)$ whilst on $[-R,-R+iR/2]$ there is a corresponding upper bound of $g(-R/2)$. Since $G$ is a homeomorphism of $\H$, the union of the previous three line segments will form a path in $\H$ with end-points in $\R$. The bounds on the real and imaginary parts of this image ensure that $G(D[x,R])\supseteq D[g(x),R^\prime]$. Combined with the earlier statement for $D[x,r]$, this proves the lemma.
\end{proof}

Next we need to give lower bounds on the conformal modulus of the images of annuli under $F_n$. We obtain bounds in terms of the distortion of these maps directly from the definition of conformal modulus:
\begin{lemma}\label{l:mod}
Let $A\subset\C$ be a compact topological annulus and $f$ be a quasiconformal map defined on some neighbourhood of $A$, then
\begin{displaymath}
\mod f(A)\geq\frac{(\mathrm{Th}A)^2}{\int_A K(z,f)\;dm(z)}
\end{displaymath}
where $\mathrm{Th}(A)$ is the thickness of $A$ (i.e., the distance between the components of its complement) and $m$ denotes Lebesgue measure.
\end{lemma}
\begin{proof}
Let $\mathcal{C}_{f(A)}$ denote the set of rectifiable, injective curves in $f(A)$ joining the inner and outer boundary components. An elementary calculation for regular annuli \cite[Examples~2-3]{ahlfors2006} (along with conformal invariance of the modulus - recall the definition given after the statement of Lemma~\ref{l:DistortionDiam}) shows that $\mod\;\mathcal{C}_{f(A)}=1/\mod f(A)$. Let $\mathcal{C}_A$ be defined analogously to $\mathcal{C}_{f(A)}$. Since $f$ is quasiconformal, we can map between $\mathcal{C}_A$ and $\mathcal{C}_{f(A)}$ up to sets of conformal modulus zero. More precisely, by \cite[III Theorem~6.2]{lehto1973quasiconformal} there exist $\mathcal{C}_{A}^\prime\subseteq\mathcal{C}_A$ and $\mathcal{C}_{f(A)}^\prime\subseteq\mathcal{C}_{f(A)}$ such that
\begin{displaymath}
    \gamma\in\mathcal{C}_{A}^\prime\;\Leftrightarrow\; f\circ\gamma\in\mathcal{C}_{f(A)}^\prime\qquad\text{and}\qquad\mod\;\mathcal{C}_A\setminus\mathcal{C}_{A}^\prime=\mod\;\mathcal{C}_{f(A)}\setminus\mathcal{C}_{f(A)}^\prime=0.
\end{displaymath}
Therefore $\mod\;\mathcal{C}_{f(A)}^\prime=\mod\;\mathcal{C}_{f(A)}=1/\mod f(A)$. Moreover the proof of the cited result shows that for $\gamma\in\mathcal{C}^\prime(A)$ the derivative of $f\circ\gamma$ is equal to $Df(\gamma_t)\gamma^\prime_t$ for almost every $t$ where $Df$ denotes the (total) derivative of $f$.

We define a non-negative measurable function $\rho$ by
\begin{displaymath}
\rho(w):=(\mathrm{Th}A)^{-1}(\lvert\partial_zf\rvert-\lvert\partial_{\overline{z}}f\rvert)^{-1}\circ f^{-1}(w)
\end{displaymath}
for $w\in f(A)$. Given $\tilde{\gamma}\in\mathcal{C}^\prime_{f(A)}$, the curve $\gamma:=f^{-1}\circ\tilde{\gamma}$ is rectifiable. By the above expression for the derivative of $\tilde{\gamma}=f\circ\gamma$
\begin{displaymath}
\int_{\tilde{\gamma}}\rho(w)\lvert dw\rvert=(\mathrm{Th}A)^{-1}\int_{\gamma}\frac{\lvert Df(z)dz\rvert}{\lvert\partial_zf\rvert-\lvert\partial_{\overline{z}}f\rvert}\geq(\mathrm{Th}A)^{-1}\int_{\gamma}\lvert dz\rvert\geq 1
\end{displaymath}
where the inequality follows because for $\zeta\in\C$ we have $Df(z)\zeta=\partial_zf(z)\zeta+\partial_{\overline{z}}f(z)\overline{\zeta}$. Hence by the definition of conformal modulus
\begin{displaymath}
\frac{1}{\mod f(A)}=\mod\;\mathcal{C}^\prime_{f(A)}\leq \int_{f(A)}\rho^2(w)\;dm(w)=(\mathrm{Th}A)^{-2}\int_{f(A)}\frac{dm(w)}{(\lvert\partial_zf\rvert-\lvert\partial_{\overline{z}}f\rvert)^2}.
\end{displaymath}
Recalling that the Jacobian of $f$ can be written as $\lvert\partial_zf\rvert^2-\lvert\partial_{\overline{z}}f\rvert^2$, by the change of variables formula for quasiconformal maps \cite[Theorem~3.8.1]{astala2009elliptic}, the rightmost term above is equal to
\begin{displaymath}
(\mathrm{Th}A)^{-2}\int_{A}\frac{\lvert\partial_zf\rvert+\lvert\partial_{\overline{z}}f\rvert}{\lvert\partial_zf\rvert-\lvert\partial_{\overline{z}}f\rvert}\;dm(z)=(\mathrm{Th}A)^{-2}\int_{A}K(z,f)\;dm(z).
\end{displaymath}
Combining the previous two equations completes the proof.
\end{proof}

We will apply the previous lemma with $f=F_n$ and $A=\mathbb{A}$ as defined in $\mathrm{Ann}(x,t,y,s,N)$. The thickness of this annulus will be controlled by Lemma~\ref{l:thickness} and the average distortion on the annulus is bounded by the following estimate:

\begin{lemma}\label{l:integral}
Let $\mathcal{A}=D[0,R]\backslash D[0,r]$ for some $0<r<R<1$ then
\begin{multline*}
    \int_{\Psi_1(\mathcal{A})}K(z,\Psi_1^{-1})\;dm(z)\leq 2^5\;\lvert \psi_1(2R)-\psi_1(-2R)\rvert^2\\
    +2^7\sum_{m=0}^\infty\sum_{\ell\in S(m,r,R)}\lvert \psi_1((\ell+2)R2^{-m})-\psi_1((\ell-2)R2^{-m})\rvert^2
\end{multline*}
where $S(m,r,R):=\{\ell\in\Z\;|\;r\leq R2^{-m}\lvert\ell\rvert\leq R\}$. Moreover the same inequality holds if we replace $\Psi_1,\psi_1$ by $\Psi_2,\psi_2$ and reflect $\mathcal{A}$ in the real axis.
\end{lemma}
\begin{proof}
This follows from essentially the same argument as in the proof of Lemma~\ref{l:Distortion2}. By a change of variables and Lemma~\ref{l:BAderivative}
\begin{displaymath}
\int_{\Psi_1(\mathcal{A})}K(z,\Psi_1^{-1})\;dm(z)=\int_{\mathcal{A}}\lvert D\Psi_1(w)\rvert^2\;dm(w)\leq 16\int_\mathcal{A}\frac{\lvert \psi_1(x+y)- \psi_1(x-y)\rvert^2}{y^2}\;dxdy.
\end{displaymath}
Let $\mathcal{A}_1=[-r,r]\times[r,R]$ and $\mathcal{A}_2=\mathcal{A}\setminus\mathcal{A}_1$. Since $\psi_1$ is increasing
\begin{align*}
    \int_{\mathcal{A}_1}\frac{\lvert \psi_1(x+y)- \psi_1(x-y)\rvert^2}{y^2}\;dxdy&\leq \lvert \psi_1(2R)-\psi_1(-2R)\rvert^2\int_{-r}^r\int_r^R\frac{1}{y^{2}}\;dydx\\
    &\leq 2 \lvert \psi_1(2R)-\psi_1(-2R)\rvert^2.
\end{align*}
Turning to $\mathcal{A}_2$ and again using the fact that $\psi_1$ is increasing;
\begin{multline*}
\int_{\mathcal{A}_2}\frac{\lvert \psi_1(x+y)- \psi_1(x-y)\rvert^2}{y^2}\;dxdy\\
\begin{aligned}
&\leq \sum_{m=0}^\infty R^{-2}2^{2(m+1)}\int_{\mathcal{A}_2\cap\{R2^{-(m+1)}<y\leq R2^{-m}\}}\lvert \psi_1(x+R2^{-m})- \psi_1(x-R2^{-m})\rvert^2\;dxdy\\
&\leq\sum_{m=0}^\infty R^{-1}2^{m+2}\int_{[-R,-r]\cup[r,R]}\lvert \psi_1(x+R2^{-m})- \psi_1(x-R2^{-m})\rvert^2\;dx.
\end{aligned}
\end{multline*}
By similar reasoning, for each $m$
\begin{multline*}
    \int_{[-R,-r]\cup[r,R]}\lvert \psi_1(x+R2^{-m})- \psi_1(x-R2^{-m})\rvert^2\;dx\\
    \begin{aligned}
    &\leq\sum_{\ell\in S(m,r,R)}\int_{[(\ell-1)R2^{-m},(\ell+1)R2^{-m}]}\lvert \psi_1((\ell+2)R2^{-m})-\psi_1((\ell-2)R2^{-m})\rvert^2\;dx\\
    &\leq \sum_{\ell\in S(m,r,R)}R2^{-m+1}\lvert \psi_1((\ell+2)R2^{-m})-\psi_1((\ell-2)R2^{-m})\rvert^2
    \end{aligned}
\end{multline*}
and combining the previous four displayed equations completes the proof of the lemma.
\end{proof}

Combining the previous three lemmas, we may now give sufficient conditions (in terms of $\psi_1$ and $\psi_2$) to construct topological annuli as specified in $\mathrm{Ann}(x,t,y,s,N)$. Letting $\rho,\epsilon\in(0,1)$ and $x,y\in[0,1]$ be given, for $t\geq 0$, $\ell\in\Z$ and $m\in\N$ we define
\begin{align*}
    B_t=[-\rho^{t},\rho^{t}],\quad    J_{t,\ell}=\left[\rho^{t}\frac{\ell-5}{4},\rho^{t}\frac{\ell-4}{4}\right],\quad I_{m,\ell,t}=[(\ell-2)\rho^t2^{-m},(\ell+2)\rho^t2^{-m}]
\end{align*}
and
\begin{multline*}
    \mathrm{Shape}^{(1)}(t):=\left\{\min_{\ell=1,\dots,8}\frac{\tau^{(1)}(x+J_{t,\ell})}{\tau^{(1)}(x+B_t)}\geq 2^{-9}\right\}\cap\left\{\frac{\tau^{(1)}(x+2B_{t+1/4})}{\tau^{(1)}(x+B_t)}\leq 2^{-25}\right\}\cap\\\left\{\frac{\tau^{(1)}(x+2B_t)}{\tau^{(1)}(x+B_t)}\leq 2^7\right\}\cap
    \left\{\sum_{m=0}^\infty\sum_{\ell\in S(m,\rho^{t+1/4},\rho^{t})}\frac{\tau^{(1)}(x+I_{m,\ell,t})^2}{\tau^{(1)}(x+B_t)^2}\leq 2^{13}\right\}.
\end{multline*}
Intuitively this event controls the `shape' of the image $\Psi^{(1)}(A_t(x))$. We define $\mathrm{Shape}^{(2)}(t)$ analogously after replacing $x$ by $y$ and $\tau^{(1)}$ by $\tau^{(2)}$. For $N\in\N$, we then define
\begin{displaymath}
\mathrm{Match}(x,y,N):=\big\{\psi_1(x)\in\psi_2([y,y+\rho^{(1+\epsilon)5N}))\big\}\cup\Big\{\psi_2(y)\in\psi_1([x,x+\rho^{(1+\epsilon)5N}))\Big\}
\end{displaymath}
and
\begin{displaymath}
\mathrm{Centre}(N):=\left\{\frac{\tau^{(1)}([x,x+\rho^{(1+\epsilon)5N}])}{\tau^{(1)}(x+B_{5N+1})}\leq 2^{-25}\right\}\cap\left\{\frac{\tau^{(2)}([y,y+\rho^{(1+\epsilon)5N}])}{\tau^{(2)}(y+B_{5N+1})}\leq 2^{-25}\right\}.
\end{displaymath}
Together, these events ensure that $\psi_1(x)$ and $\psi_2(y)$ are close. Finally, for $t,s\geq 0$ we set
\begin{displaymath}
\mathrm{Size}(t,s):=\left\{2^{-12}\leq\frac{\tau^{(1)}(x+B_t)}{\tau^{(1)}([0,1])}\Big\slash\frac{\tau^{(2)}(y+B_s)}{\tau^{(2)}([0,1])}\leq 2^{12}\right\}.
\end{displaymath}
This event controls the relative sizes of $\Psi_1(A_t(x))$ and $\Psi_2(\widetilde{A}_s(y))$. Note that these events depend on $x$ and $y$ but this is suppressed in the notation. For all of the events above, we interpret the sets as subsets of $\T:=\R\slash\Z$ (so for example $B_t=[-\rho^t,\rho^t]=[0,\rho^t]\cup[1-\rho^t,1]$).

These events together imply the existence of a suitable annulus:
\begin{lemma}\label{l:Good_annuli}
Let $0\leq t,s\leq 5N$ and $\epsilon,\rho\in(0,1)$ and define
\begin{displaymath}
\mathrm{Ann}^\prime(x,t,y,s,N):=\mathrm{Size}(t,s)\cap \mathrm{Shape}^{(1)}(t)\cap \mathrm{Shape}^{(2)}(s)\cap \mathrm{Centre}(N).
\end{displaymath}
Then for all $x,y\in[0,1]$
\begin{displaymath}
\mathrm{Ann}^\prime(x,t,y,s,N)\cap\mathrm{Match}(x,y,N)\subseteq\mathrm{Ann}(x,t,y,s,N)
\end{displaymath}
where $\mathrm{Ann}(x,t,y,s)$ was specified in Definition~\ref{d:Ann}.
\end{lemma}
\begin{proof}
Throughout the lemma, we assume that
\begin{equation}\label{e:relative_size}
    \frac{\tau^{(1)}(x+B_t)}{\tau^{(1)}([0,1])}\geq\frac{\tau^{(2)}(y+B_s)}{\tau^{(2)}([0,1])}.
\end{equation}
If the reverse inequality holds, the proof remains valid by exchanging the roles of $\tau^{(1)}(x+\cdot)$ and $\tau^{(2)}(y+\cdot)$. By Lemma~\ref{l:thickness} (applied to $g=\psi_1$) on $\mathrm{Shape}^{(1)}(t)\cap \mathrm{Size}(t,s)$ we have
\begin{align*}
\Psi_1(A_t(x))&\supseteq D\left[\psi_1(x),\frac{1}{2}\min_{\ell=1,\dots,8}\frac{\tau^{(1)}(x+J_{t,\ell})}{\tau^{(1)}([0,1])}\right]\backslash D\left[\psi_1(x),\frac{\tau^{(1)}(x+2B_{t+1/4})}{\tau^{(1)}([0,1])}\right]
\\
&\supseteq\psi_1(x)+ \frac{\tau^{(1)}(x+B_t)}{\tau^{(1)}([0,1])}D[0,2^{-10}]\backslash D[0,2^{-25}]\\
&\supseteq\psi_1(x)+ \frac{\tau^{(2)}(y+B_s)}{\tau^{(2)}([0,1])}D[0,2^{-10}]\backslash D[0,2^{-13}].
\end{align*}
Similarly on $\mathrm{Shape}^{(2)}(s)$ we have
\begin{align*}
\Psi_2(\widetilde{A}_s(y))&\supseteq \psi_2(y)+ \frac{\tau^{(2)}(y+B_s)}{\tau^{(2)}([0,1])}\widetilde{D}[0,2^{-10}]\backslash \widetilde{D}[0,2^{-25}].
\end{align*}
On the event $\mathrm{Centre}(N)\cap\mathrm{Match}(x,y,N) \cap\mathrm{Size}(t,s)$, we see that
\begin{align*}
    \lvert\psi_2(y)-\psi_1(x)\rvert&\leq\max\Big\{\frac{\tau^{(1)}([x,x+\rho^{(1+\epsilon)5N}])}{\tau^{(1)}([0,1])},\frac{\tau^{(2)}([y,y+\rho^{(1+\epsilon)5N}])}{\tau^{(2)}([0,1])}\Big\}\\
    &\leq 2^{-25}\max\Big\{\frac{\tau^{(1)}(x+B_t)}{\tau^{(1)}([0,1])},\frac{\tau^{(2)}(y+B_s)}{\tau^{(2)}([0,1])}\Big\}\leq 2^{-13}\frac{\tau^{(2)}(y+B_s)}{\tau^{(2)}([0,1])}
\end{align*}
where the final inequality uses \eqref{e:relative_size}.
Combined with the previous inclusion, this implies that
\begin{equation}
    \Psi_2(\widetilde{A}_s(y))\supseteq\psi_1(x)+\frac{\tau^{(2)}(y+B_s)}{\tau^{(2)}([0,1])}\widetilde{D}[0,2^{-11}]\backslash \widetilde{D}[0,2^{-12}].
\end{equation}
We therefore define
\begin{displaymath}
\mathbb{A}:=\psi_1(x)+2^{-12}\frac{\tau^{(2)}(y+B_s)}{\tau^{(2)}([0,1])}[-2,2]^2\setminus[-1,1]^2\subseteq \Psi_1(A_t(x))\cup\Psi_2(\widetilde{A}_s(y))
\end{displaymath}
which is a topological annulus that surrounds $\Psi_1(D[x,\rho^{5N+1}])\cup\Psi_2(\widetilde{D}[y,\rho^{5N+1}])$.

By Lemma~\ref{l:integral} (applied to $g=\psi_1,\psi_2$) on $\mathrm{Shape}^{(1)}(t)\cap \mathrm{Size}(t,s)$ we have
\begin{align*}
    \int_{\Psi_1(A_{t}(x))}K(z,\Psi_1^{-1})\;dm(z)&\leq (2^{19}+2^{20})\left(\frac{\tau^{(1)}(x+B_t)}{\tau^{(1)}([0,1])}\right)^2\leq 2^{45}\left(\frac{\tau^{(2)}(y+B_s)}{\tau^{(2)}([0,1])}\right)^2
\end{align*}
and on $\mathrm{Shape}^{(2)}(s)$
\begin{align*}
    \int_{\Psi_2(\widetilde{A}_{s}(y))}K(z,\Psi_2^{-1})\;dm(z)\leq 2^{21}\left(\frac{\tau^{(2)}(y+B_s)}{\tau^{(2)}([0,1])}\right)^2.
\end{align*}
It follows easily from the chain rule for Wirtinger derivatives that the distortion is preserved under pre- or post-composition by a conformal map. More precisely, for differentiable maps $f$ and $g$, if $g$ is conformal then $K(z,f\circ g)=K(g(z),f)$ and if $f$ is conformal then $K(z,f\circ g)=K(z,g)$. Using this observation along with the definition of $K$ in \eqref{e:Distortion}, we see that for all $n$
\begin{align*}
    \int_{\mathbb{A}} K(z,F_n\circ e)\;dm(z)&\leq\int_\mathbb{A} K(e(z))\;dm(z)\\
    &\leq \int_{\Psi_1(A_{t}(x))}K(z,\Phi_1^{-1}\circ e)\;dm(z)+\int_{\Psi_2(\widetilde{A}_{s})}K(z,\Phi_2^{-1}\circ e)\;dm(z)\\
    &=\int_{\Psi_1(A_{t}(x))}K(z,\Psi_1^{-1})\;dm(z)+\int_{\Psi_2(\widetilde{A}_{s})}K(z,\Psi_2^{-1})\;dm(z)
\end{align*}
where the final equality uses the fact that $e\circ\Psi_j^{-1}=\Phi_j^{-1}\circ e$. Combining the previous bounds with Lemma~\ref{l:mod} we have $\mod\; F_n\circ e(\mathbb{A}_{t,s})\geq2^{-70}$ which completes the proof of the lemma.
\end{proof}

The previous lemma can be viewed as separating out the local and global conditions ($\mathrm{Ann}^\prime$ and $\mathrm{Match}$ respectively) required to construct a `good' annulus. We wish to construct sufficiently many `good' annuli to apply Lemma~\ref{l:Annuli-Holder} but the difficulty is that we only have useful probabilistic control over the local behaviour. Our strategy is therefore to control the local behaviour with sufficiently high probability that $\mathrm{Ann}^\prime$ occurs often for (essentially) all $x$ and $y$. We then use the deterministic fact that for each $x$, there must be some $y$ such that $\mathrm{Match}(x,y,N)$ occurs. This argument is contained in the following proposition:

\begin{definition}\label{d:AnnSeq}
    Given $x,y\in[0,1]$, $\delta>0$ and $N\in\N$, we define $\mathrm{AnnSeq}^\prime(x,y,N)$ to be the event that there exist two sequences $t_1,t_2,\dots$ and $s_1,s_2,\dots$ satisfying
\begin{enumerate}
    \item $t_m-t_{m-1},s_m-s_{m-1}\geq 1/4$ for all $m\geq 2$
    \item $\sum_{t_m,s_m\leq 5N}\ind_{\mathrm{Ann}^\prime(x,t_m,y,s_m,N)}>\delta N$.
\end{enumerate}
    We define $\mathrm{AnnSeq}(x,y,N)$ analogously if $\mathrm{Ann}^\prime$ is replaced by $\mathrm{Ann}$.
\end{definition}

\begin{proposition}\label{p:Stationary_condition}
Let $\delta,C>0$ and $\epsilon,\rho\in(0,1)$ be given and let $F_n$ be the collection of maps defined by \eqref{e:QCApprox} and \eqref{e:QCApprox2}. Suppose that
\begin{displaymath}
\P(\mathrm{AnnSeq}^\prime(x,y,N))>1-C\rho^{(2+3\epsilon)5N}
\end{displaymath}
for all $x,y\in[0,1]$ and all $N\in\N$, then with probability one the functions $(F_n)_{n\in\N}$ are uniformly H\"older continuous on $\partial\D$.
\end{proposition}
\begin{proof}
Let $P_N\subset[0,1]$ be a $(1/2)\rho^{(1+\epsilon)5N}$-net of $[0,1]$ containing at most $2\rho^{(1+\epsilon)5N}$ points. Then by assumption
\begin{equation}\label{e:Stationary1}
    \sum_{N\in\N}\sum_{x,y\in P_N}\P\left(\mathrm{AnnSeq}^\prime(x,y,N)^c\right)\leq\sum_{N\in\N}4C\rho^{5\epsilon N}<\infty.
\end{equation}
Hence by the first Borel-Cantelli lemma, there exists a random $N_0$ such that for all $N\geq N_0$ and all $x,y\in P_N$, $\mathrm{AnnSeq}^\prime(x,y,N)$ occurs. For any $N\in\N$ and $u\in\T\simeq[0,1)$, there exist $x_1,x_2\in P_n$ such that
\begin{equation}\label{e:grid_points}
    u\in\psi_j([x_j,x_j+\rho^{(1+\epsilon)5N}))\qquad\text{for }j=1,2.
\end{equation}
This means that $\mathrm{Match}(x_1,x_2,N)$ occurs and hence by Lemma~\ref{l:Good_annuli}, so does $\mathrm{AnnSeq}(x_1,x_2,N)$.

We now apply Lemma~\ref{l:Annuli-Holder} to $F_n$ for each $n\in\N$. Specifically we choose the sequence of concentric annuli to be $e(\mathbb{A}_m)$ where $\mathbb{A}_m$ is defined by the event $\mathrm{Ann}(x_1,t_m,x_2,s_m,N)$ whenever it occurs. By assumption we have at least $\delta N$ such annuli and the conditions on $t_m,s_m$ ensure that they satisfy conditions $(1)$, $(3)$, and $(4)$ of Lemma~\ref{l:Annuli-Holder} for $f=F_n$ (the third condition follows from applying Koebe's 1/4 theorem at infinity). For $N$ sufficiently large, by Definition~\ref{d:Ann} the concentric annuli all surround
\begin{displaymath}
    e\circ\Psi_1(D[x_1,2\rho^{(1+\epsilon)5N}])\cup e\circ\Psi_2(\widetilde{D}[x_2,2\rho^{(1+\epsilon)5N}]).
\end{displaymath}
By Lemma~\ref{l:MeasureBasic} and definition of the Beurling-Ahlfors extension, $\Psi_1$ and $\Psi_2$ are almost surely locally H\"older continuous. Therefore by \eqref{e:grid_points} there exists a random $N_1\in\N$ and a deterministic $\alpha>0$ (depending on $\gamma$) such that for all $N>N_1$, the above set contains $e (B(u,\rho^{(1+\epsilon)5N\alpha}))$ which verifies the second condition of Lemma~\ref{l:Annuli-Holder}. We conclude that the $F_n$ are uniformly H\"older continuous on $\partial\D$ as required.
\end{proof}

\begin{remark}\label{re:eps}
    In the remainder of our work, the reader will see that the parameter $\epsilon>0$ plays no essential role. More precisely, for any $\epsilon>0$ all of our later proofs would be valid by reducing $\gamma$ (the parameter of our measures) sufficiently. Therefore we now fix a value of $\epsilon>0$ once and for all. Our motivation for working with an arbitrary value rather than simplifying notation by setting $\epsilon=1$, is that this generality may be useful in the future for extending our results to higher values of $\gamma$.
\end{remark}

\section{An approximate decomposition for the measures}\label{s:Decompose}
To summarise the analysis of the previous section: if we can show that for each $x,y\in[0,1]$, $\mathrm{Ann}^\prime(x,t_m,y,s_m,N)$ occurs for sufficiently many pairs $(t_m,s_m)$ with high probability then we may deduce the desired H\"older continuity bound on the family of maps $F_n$. The events $\mathrm{Ann}^\prime(x,t,y,s,N)$ (defined in terms of the measures $\tau^{(1)}$ and $\tau^{(2)}$) have strong dependence for different values of $t$ and $s$. In this section we show that for many values of $t$ and $s$, the measures $\tau^{(1)}$ and $\tau^{(2)}$ can be approximately decomposed over different scales. This decomposition will allow us to control the dependence between the events $\mathrm{Ann}^\prime(\cdot)$.

\subsection{Outline of the deomposition}
We begin by roughly describing the decomposition and giving some intuition for why it is useful.

Recall that for $a>0$, the continuous Gaussian process $H_a(\cdot)$ is defined in Lemma~\ref{l:GFFExistence} with respect to the white noise above height $a$. For $u\in[0,1]$, $0<a\leq b$ and $j=1,2$ we define
\begin{align*}
E^{(j)}(u,a)&:=\exp\Big(\gamma H^{(j)}_a(u)-\frac{\gamma^2}{2}\Var[H_a^{(j)}(0)]\Big),\\
E^{(j)}(u,a,b)&:=\frac{E^{(j)}(u,a)}{E^{(j)}(u,b)}.
\end{align*}
For $t\geq 0$, let $\tau_t^{(j)}$ denote the measure $\tau^{(j)}$ defined by the restriction of the white noise to $\{(x,y)\in\H\;|\;y\leq\rho^t\}$. More precisely, we define $\tau_t^{(j)}$ by repeating the construction of $\tau^{(j)}$, given in Lemma~\ref{l:GFFExistence} and \eqref{e:MeasureDef}, using the same white noise process but replacing $\mathcal{H}$ by the set $\{(x,y)\in\mathcal{H}\;|\;y\leq\rho^t\}$. Then by independence of the white noise process on disjoint domains
\begin{equation}\label{e:WhiteNoiseDecomp1}
\tau^{(j)}_t(du)=\lim_{\xi\to 0^+} E^{(j)}(u,\xi,\rho^t)e^{-\gamma G}2^{\gamma^2}\;du,\quad\text{and}\quad\tau^{(j)}(du)=\lim_{\xi\to 0^+} E^{(j)}(u,\xi)e^{-\gamma G}2^{\gamma^2}\;du
\end{equation}
where convergence is in the weak-$*$ topology.

Now consider the events $\mathrm{Shape}^{(1)}(t)$, $\mathrm{Shape}^{(2)}(s)$ and $\mathrm{Size}(t,s)$. For the moment, to simplify exposition, let us treat $\tau^{(1)}([0,1])$ and $\tau^{(2)}([0,1])$ as constants, set $x,y=0$ and assume that $t,s$ are integers. Our events of interest are therefore determined by $\tau^{(1)}$ restricted to $B_t$ and $\tau^{(2)}$ restricted to $B_s$. The essence of our argument is the approximation
\begin{equation}\label{e:WhiteNoiseApprox}
\tau^{(j)}(I)\approx \tau^{(j)}_t(I\setminus B_{t+7/8})\cdot E^{(j)}(0,\rho^t)
\end{equation}
for intervals $I\subset B_t$ which are not too small. Intuitively this follows from \eqref{e:WhiteNoiseDecomp1} by replacing $E^{(j)}(u,\xi,\rho^t)$ with $E^{(j)}(u,\xi)/E^{(j)}(0,\rho^t)$ and treating $B_{t+7/8}$ as negligible. Note that $\tau^{(j)}(I)$ is determined by the behaviour of the hyperbolic white noise on the region shown in Figure~\ref{Fig:WhiteNoiseDecomp1} whilst our approximation depends only on the smaller region shown in Figure~\ref{Fig:WhiteNoiseDecomp2}.

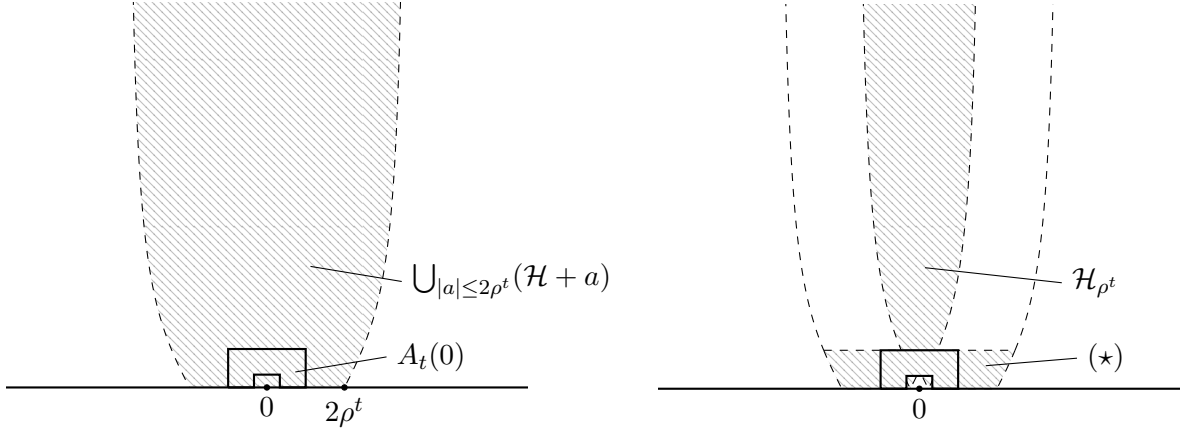
\begin{figure}[ht]
    \centering
\begin{subfigure}[b]{0.45\textwidth}
     \begin{tikzpicture}[declare function={
    func(\x)= (\x<=-1) * ((2/pi)*tan(abs(pi*(\x+1))))   +
     and(\x>-1, \x<=1) * (0)     +
     (\x>1) * ((2/pi)*tan(abs(pi*(\x-1))));
  }, ]
\begin{axis}[trig format plots=rad,samples=500,domain=-0.9:0.9,ymax=3,ymin=-0.35,axis equal,hide axis]

\addplot[samples=500,dashed,domain=-1.48:1.48,name path=A] {func(x)};
\addplot[domain=-1.1:1.1,name path=B]{5};
\addplot[pattern=north west lines,opacity=0.5] fill between[of=A and B];
\filldraw[black] (0,0) circle (1pt) node[anchor=north]{$0$};
\filldraw[black] (1,0) circle (1pt) node[anchor=north]{$2\rho^t$};
\draw[thick] (-5,0)--(5,0);
\draw [thick](-0.5,0)--(-0.5,0.5)--(0.5,0.5)--(0.5,0)--(0.3,0)--(0.3,0.3)--(-0.3,0.3)--(-0.3,0)--(-0.5,0);
\end{axis}

\draw (4.7,2.2)--(5.9,2);
\node[right] at (5.9,2) {$\bigcup_{\lvert a\rvert\leq2\rho^t}(\mathcal{H}+a)$};
\draw (4.1,0.8)--(5.6,1);
\node[right] at (5.6,1) {$A_t(0)$};
\end{tikzpicture}

\caption{The restriction of $\tau^{(1)}$ to $B_t$ is determined by the white noise restricted to ${\bigcup_{\lvert a\rvert\leq2\rho^t}(\mathcal{H}+a)}$.}\label{Fig:WhiteNoiseDecomp1}
\end{subfigure}
        \hfill
    \begin{subfigure}[b]{0.45\textwidth}
\begin{tikzpicture}[declare function={
     func2(\x)= max(0.5,(2/pi)*tan(abs(pi*\x)));
     func3(\x)= (\x<=-1) * ((2/pi)*tan(abs(pi*(\x+1))))   +
     and(\x>-0.12, \x<=0.12) * min(0.3,(2/pi)*tan(abs(pi*(\x+0.12))),(2/pi)*tan(abs(pi*(\x-0.12))))    +
     (\x>1) * ((2/pi)*tan(abs(pi*(\x-1))));
     func4(\x)=(func3(\x)<0.5)*func3(\x);
  }, ]

\begin{axis}[trig format plots=rad,samples=500,domain=-0.9:0.9,ymax=3,ymin=-0.35,axis equal,hide axis]
%restrict y to domain =0:3,
\addplot[samples=500,dashed,domain=-0.48:0.48,name path=A] {func2(x)};
\addplot[domain=0.3:0.8,name path=B]{5};
\addplot[pattern=north west lines,opacity=0.5] fill between[of=A and B];
\addplot[domain=-1.2:1.2,dashed,name path=C] {func4(x)};
\addplot[domain=-1.2:1.2,dashed,name path=D]{0.5};
\addplot[pattern=north west lines,opacity=0.5] fill between[of=C and D];
\addplot[domain=1.2:1.48,dashed]{(2/pi)*tan(abs(pi*(\x-1)))};
\addplot[domain=-1.48:-1.2,dashed]{(2/pi)*tan(abs(pi*(\x+1)))};
\filldraw[black] (0,0) circle (1pt) node[anchor=north]{$0$};
\draw[thick] (-5,0)--(5,0);
\draw [thick](-0.5,0)--(-0.5,0.5)--(0.5,0.5)--(0.5,0)--(0.3,0)--(0.3,0.3)--(-0.3,0.3)--(-0.3,0)--(-0.5,0);
\end{axis}

\draw (3.5,2.2)--(4.8,2);
\node[right] at (4.8,2) {$\mathcal{H}_{\rho^t}$};
\draw (4.6,0.9)--(5.7,1);
\node[right] at (5.7,1) {$(\star)$};
\end{tikzpicture}

\caption{The approximation in \eqref{e:WhiteNoiseApprox} is determined by the white noise restricted to $\mathcal{H}_{\rho^t}\cup(\star)$ where $(\star)=\bigcup_{\rho^{t+7/8}\leq\lvert a\rvert\leq2\rho^t}(\mathcal{H}+a)\cap\{y\leq\rho^t\}$.}\label{Fig:WhiteNoiseDecomp2}
    \end{subfigure}
    \caption{}
    \label{fig:WhiteNoiseDecomp1+2}
\end{figure}

If we substitute the right hand side of \eqref{e:WhiteNoiseApprox} into $\mathrm{Shape}^{(1)}(t)$ then the exponential terms cancel and we obtain an intersection of events of the form
\begin{equation}\label{e:ApproxEvents}
\frac{\tau_t^{(1)}((\rho^t\cdot I)\setminus B_{t+7/8})}{\tau_t^{(1)}((\rho^t\cdot J)\setminus B_{t+7/8})}\leq c
\end{equation}
for intervals $I,J$ which each have length of order $1$. These events depend (through the hyperbolic white noise) on disjoint regions for distinct integer values of $t$ (see Figure~\ref{Fig:WhiteNoiseDecomp3}) and so are independent. Moreover as $t$ increases, a basic scaling argument for the white noise shows that $\rho^{-t}\tau_t^{(1)}(\rho^t\cdot)$ converges in distribution to a non-degenerate random measure. Hence the intersection of events of the form \eqref{e:ApproxEvents} will approach an i.i.d.\ sequence for large values of $t$, so we can easily prove large deviation bounds for the number of such events which occur.

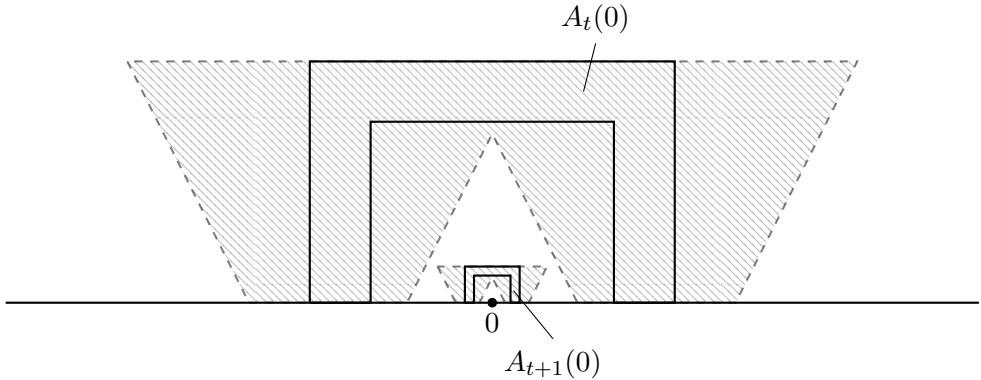
\begin{figure}[ht]
    \centering
\begin{tikzpicture}
\begin{scope}[scale=0.8]
\filldraw[black] (0,0) circle (2pt) node[anchor=north]{$0$};
\draw[thick] (-8,0)--(8,0);
\draw [thick](-3,0)--(-3,4)--(3,4)--(3,0)--(2,0)--(2,3)--(-2,3)--(-2,0)--(-3,0);
\draw[thick,dashed,pattern=north west lines,opacity=0.5] (-6,4)--(-4,0)--(-1.4,0)--(0,2.8)--(1.4,0)--(4,0)--(6,4)--(-6,4);
\begin{scope}[scale=0.15]
\draw [thick](-3,0)--(-3,4)--(3,4)--(3,0)--(2,0)--(2,3)--(-2,3)--(-2,0)--(-3,0);
\draw[thick,dashed,pattern=north west lines,opacity=0.5] (-6,4)--(-4,0)--(-1.4,0)--(0,2.8)--(1.4,0)--(4,0)--(6,4)--(-6,4);
\end{scope}
\draw (1.5,3.5)--(1.7,4.3);
\node [above] at (1.7,4.3) {$A_t(0)$};
\draw (0.35,0.2)--(1,-0.6);
\node [below] at (1,-0.6) {$A_{t+1}(0)$};
\end{scope}
\end{tikzpicture}
\caption{Events of the form \eqref{e:ApproxEvents} are determined by the white noise restricted to the shaded region containing $A_t(0)$. Such regions are disjoint for $t$ and $t+1$, implying that the corresponding events are independent. (Note: this illustration is not to scale.)}\label{Fig:WhiteNoiseDecomp3}
\end{figure}

Turning to $\mathrm{Size}(t,s)$, if we now replace $\tau^{(1)}(B_t)$ and $\tau^{(2)}(B_s)$ by their approximations according to \eqref{e:WhiteNoiseApprox} then we obtain an event of the form
\begin{equation}\label{e:ApproxEvent2}
    c^{-1}\leq\frac{\rho^{-t}\tau^{(1)}_t(B_t\setminus B_{t+7/8})}{\rho^{-s}\tau^{(2)}_s(B_s\setminus B_{s+7/8})}\cdot\exp(\star)\leq c
\end{equation}
where
\begin{displaymath}
    \star=\gamma H^{(1)}_{\rho^t}(x)-\gamma H^{(2)}_{\rho^s}(y)-\big(1+{\gamma^2}/{2}\big)(t-s)\log(1/\rho).
\end{displaymath}
(Note that we have treated $\tau^{(1)}([0,1])$ and $\tau^{(2)}([0,1])$ as constants and approximated the variance of $H^{(j)}_{\xi}(0)$ by $\log(1/\xi)$.) By a simple telescoping:
\begin{displaymath}
\gamma H_{\rho^t}^{(1)}(x)-t(1+\gamma^2/2)\log(1/\rho)=\gamma H_0^{(1)}(x)+\sum_{k=1}^t\big(\gamma H_{\rho^k}^{(1)}(x)-\gamma H_{\rho^{k-1}}^{(1)}(x)-(1+\gamma^2/2)\log(1/\rho)\big),
\end{displaymath}
the terms making up $\star$ can be expressed as sums of independent Gaussian variables (see Figure~\ref{Fig:WhiteNoiseDecomp4}). Proving that \eqref{e:ApproxEvent2} occurs for many values of $t$ and $s$ thus reduces to the study of two independent, biased random walks. We analyse this situation in Section~\ref{s:LargeDeviations}.
\begin{figure}[ht]
    \centering
\begin{tikzpicture}
\begin{axis}[trig format plots=rad,samples=500,domain=-0.5:0.5,ymax=1,ymin=-0.15,axis equal,hide axis]
\addplot[dashed,domain=-0.48:0.48,name path=A] {(2/pi)*tan(abs(pi*\x))};
\addplot[domain=0.3:0.8,name path=B]{5};
\addplot[pattern=north west lines,opacity=0.5] fill between[of=A and B];
\filldraw[black] (0,0) circle (1pt) node[anchor=north]{$x$};
\draw[thick] (-5,0)--(5,0);

\foreach \a in {0.8,0.4,0.2,0.1,0.05}{
\addplot[dashed,domain={-atan(\a*pi/2)/pi}:{atan(\a*pi/2)/pi},name path=B]{\a};
}
\draw [thick,decorate,
    decoration = {brace}] ({-atan(0.4*pi)/180},0) -- node[left] {$\rho$} ({-atan(0.4*pi)/180},0.8);
\draw [thick,decorate,
    decoration = {brace}] ({-atan(0.2*pi)/180},0) -- node[left] {$\rho^2$} ({-atan(0.2*pi)/180},0.4);
\draw [thick,decorate,
    decoration = {brace}] ({atan(0.1*pi)/180},0.2) -- node[right] {$\rho^3$} ({atan(0.1*pi)/180},0);
\draw (0.15,0.6)--(0.4,0.5);
\node [right] at (0.4,0.5) {$\mathcal{H}+x$};
\end{axis}
\end{tikzpicture}
\caption{The Gaussian variable $H_{\rho^{t+1}}^{(1)}(x)-H_{\rho^{t}}^{(1)}(x)$ is determined by the white noise restricted to $(\mathcal{H}+x)\cap\{\rho^{t+1}\leq y\leq \rho^t\}$ and so these variables are independent for different integer values of $t$.}\label{Fig:WhiteNoiseDecomp4}
\end{figure}
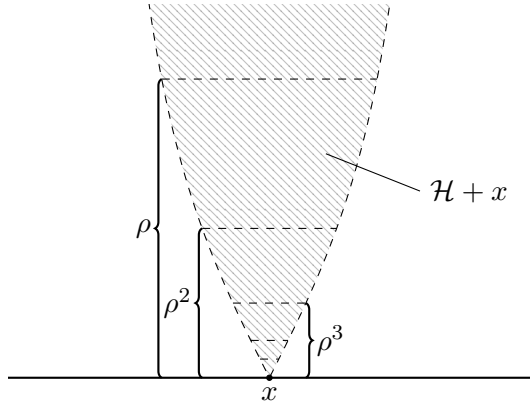

\subsection{Sufficient conditions for the approximation}\label{ss:Sufficient}
We now begin to justify the approximation described above rigorously. In this subsection we introduce a number of events which ensure that \eqref{e:WhiteNoiseApprox} is valid, allow us to deal with non-integer values of $t$ and $s$ and justify treating $\tau^{(1)}([0,1])$ and $\tau^{(2)}([0,1])$ as constants (for our purposes). In the next subsection we will show that these events occur with high probability for many values of $t$ and $s$.
\medskip

\noindent\underline{Upper-scale correlations:} First we consider how to approximate the term $E^{(j)}(u,\epsilon,\rho^t)$ in \eqref{e:WhiteNoiseDecomp1}. For $k\in\N\cup\{0\}$ and a Borel set $I$, we define
\begin{align*}
    \mathcal{S}_k^{(1)}(I):=\sup_{u\in I}\frac{E^{(1)}(u,\rho^{k+1},\rho^k)}{E^{(1)}(x,\rho^{k+1},\rho^k)},\quad \mathcal{S}_{-1}^{(1)}(I):=\sup_{u\in I}\frac{E^{(1)}(u,1)}{E^{(1)}(x,1)},
\end{align*}
and we define $\mathcal{I}_k^{(1)}(I)$ to be the corresponding infima. Moreover we define analogous quantities $\mathcal{S}_k^{(2)}(\cdot)$ and $\mathcal{I}_k^{(2)}(\cdot)$ by replacing $E^{(1)}$ with $E^{(2)}$ and $x$ with $y$.

From \eqref{e:WhiteNoiseDecomp1} we see that for any Borel set $I$ and $n\in\N$
\begin{equation}\label{e:Upper-scale}
    \prod_{k=-1}^{n-1}\mathcal{I}_k^{(1)}(I)\leq\frac{\tau^{(1)}(I)}{\tau_n^{(1)}(I)E^{(1)}(x,\rho^n)}\leq\prod_{k=-1}^{n-1}\mathcal{S}_k^{(1)}(I).
\end{equation}
Therefore if each of these suprema/infima is close to one, we can justify the first part of our approximation. Accordingly we define the events
\begin{align*}
    \mathrm{Upp}^{(1)}_{n,k}&:=\left\{\log \mathcal{S}_k^{(1)}(x+2B_n)\leq 2^{k-n}\log(2)\right\}\cap\left\{\log \mathcal{I}_k^{(1)}(x+2B_n)\geq -2^{k-n}\log(2)\right\},\\ \mathrm{Upp}_n^{(1)}&:=\bigcap_{k=-1}^{n-1}\mathrm{Upp}^{(1)}_{n,k}
\end{align*}
and note that on $\mathrm{Upp}_n^{(1)}$ we have
\begin{displaymath}
1\leq\prod_{k=-1}^{n-1}\mathcal{S}_k^{(1)}(x+2B_n)\leq 2,\quad\text{and}\quad \frac{1}{2}\leq\prod_{k=-1}^{n-1}\mathcal{I}_k^{(1)}(x+2B_n)\leq 1.
\end{displaymath}
We define $\mathrm{Upp}_{n,k}^{(2)}$ and $\mathrm{Upp}_n^{(2)}$ analogously on replacing $x$ by $y$ and $\mathcal{S}_k^{(1)},\mathcal{I}_k^{(1)}$ by $\mathcal{S}_k^{(2)},\mathcal{I}_k^{(2)}$. The notation `$\mathrm{Upp}$' (short for upper-scale correlations) is to indicate that these events account for removing the dependency on the two upper regions of white noise in Figure~\ref{Fig:WhiteNoiseDecomp2}. We can obtain probability bounds on these events using standard estimates for the supremum of a Gaussian process (i.e., the Borell-TIS inequality), and will do so in the next subsection.
\medskip

\noindent\underline{Lower-scale correlations:} We next wish to justify replacing $\tau_t(I)$ in \eqref{e:WhiteNoiseDecomp1} by $\tau_t(I\setminus B_{t+7/8})$ in \eqref{e:WhiteNoiseApprox}. To do so we simply consider events on which the measure of $B_{t+7/8}$ is small relative to that of the sets we are interested in: for $n,k\in\N$ with $k\geq n$ we define
\begin{align*}
\mathrm{Low}_{n,k}^{(1)}:=\left\{\frac{\tau_n^{(1)}(x+B_{k+7/8}\setminus B_{k+1+7/8})}{\tau_n^{(1)}(x+B_{n+5/8}\setminus B_{n+7/8})}\leq 2^{n-k-33}\right\},\qquad \mathrm{Low}_n^{(1)}:=\bigcap_{k=n}^\infty \mathrm{Low}^{(1)}_{n,k}.
\end{align*}
Similarly we define $\mathrm{Low}_{n,k}^{(2)}$ and $\mathrm{Low}_{n}^{(2)}$ by replacing $\tau_n^{(1)}$ and $x$ with $\tau_n^{(2)}$ and $y$ respectively. By countable additivity of measures, on $\mathrm{Low}_n^{(1)}$ we have
\begin{equation}\label{e:Low}
\frac{\tau_n^{(1)}(x+B_{n+7/8})}{\tau_n^{(1)}(x+B_{n+5/8}\setminus B_{n+7/8})}\leq 2^{-32}.
\end{equation}
This event accounts for the dependence of the hyperbolic white noise in the lower unshaded region of Figure~\ref{Fig:WhiteNoiseDecomp2}. In the next subsection we will control the probability of these events using moment bounds on $\tau_n^{(j)}(\cdot)$. In anticipation of this, we observe that the events $\mathrm{Low}_{n_1,k_1}^{(j)}$ and $\mathrm{Low}_{n_2,k_2}^{(j)}$ depend on disjoint regions of the white noise whenever $k_1+2\leq n_2$ (provided that $\rho\leq 2^{-8}$, in which case these regions look qualitatively similar to those in Figure~\ref{Fig:WhiteNoiseDecomp3}) and so are independent.

\medskip

\noindent\underline{Fractional-scale correlations:}
The events defined so far control the behaviour of our measures on integer scales (i.e., for $t,s\in\N$). In order to extend this control to non-integer scales we consider the following event:
\begin{displaymath}
\mathrm{Frac}_n^{(1)}:=\left\{\sup_{t\in[n,n+5/8]}\sup_{u\in (x+2B_n)}\left\lvert\log E^{(1)}(u,\rho^t,\rho^n)\right\rvert<\log(2)\right\}.
\end{displaymath}
On this event, for any $t\in[n,n+5/8]$ and $I\subseteq(x+2B_n)$ we see from \eqref{e:WhiteNoiseDecomp1} that
\begin{equation}\label{e:Frac}
\frac{1}{2}\tau_n^{(1)}(I)\leq\tau_t^{(1)}(I)\leq 2\tau_n^{(1)}(I).
\end{equation}
We define $\mathrm{Frac}^{(2)}_n$ analogously. The probability of these events can once again be controlled through consideration of the supremum of a continuous Gaussian process.

\medskip

\noindent\underline{Scaling-limit correlations:} In the heuristics of the previous subsection we approximated the variance of $H^{(j)}_\epsilon$ by $\log(1/\epsilon)$. This is useful as it allows us to consider random walks with constant bias in the next section. To justify the approximation, we will define two fields $V^{(1)}$ and $V^{(2)}$ which are close to $H^{(1)}$ and $H^{(2)}$ and satisfy an exact scaling relation (as opposed to the approximate scaling relation satisfied by the $H^{(j)}$).

We define
\begin{displaymath}
\mathcal{V}=\left\{(x,y)\in\mathbb{H}\;|\; 2\lvert x\rvert\leq y\leq 1/2\right\}
\end{displaymath}
and for $0<\xi\leq 1/2$
\begin{displaymath}
V^{(j)}_\xi(u)=W^{(j)}((\mathcal{V}+u)\cap\{y\geq\xi\})
\end{displaymath}
where $W^{(j)}$ is the white noise associated with the measure $\tau^{(j)}$. Direct computation shows that $\Var[V_\xi(u)]=\log(1/2)-\log(\xi)$. Analogously with the measures $\tau^{(j)}_t$ we define $\nu_t^{(j)}$ by
\begin{equation}\label{e:MeasureDef2}
\lim_{\xi\to 0^+}\exp\Big(\gamma \big(V^{(j)}_\xi(u)-V^{(j)}_{\rho^t}(u)\big)-\frac{\gamma^2}{2}\Var\big[V^{(j)}_\xi(u)-V^{(j)}_{\rho^t}(u)\big]\Big)\;dx=\nu_t^{(j)}(dx).
\end{equation}
More precisely, it was shown in \cite[Section~3.3]{ajks} that there exists a version of hyperbolic white noise for which both $V^{(j)}_\xi$ and $H^{(j)}_\xi$ can be defined continuously and the measures $\nu_t^{(j)}$ and $\tau^{(j)}_t$ defined by \eqref{e:WhiteNoiseDecomp1} and \eqref{e:MeasureDef2} exist. We denote $\nu^{(j)}:=\nu^{(j)}_0$.

From Figure~\ref{Fig:WhiteNoiseDecomp5} we see that $\mathcal{V}$ approximates $\mathcal{H}$ close to the real axis and so we expect that $\nu_t^{(j)}$ should be a good approximation for $\tau_t^{(j)}$ whenever $t$ is large.
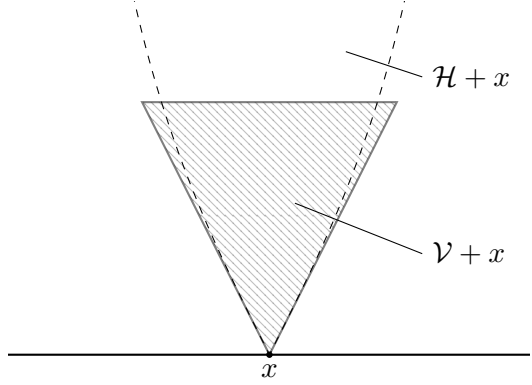
\begin{figure}[ht]
    \centering
\begin{tikzpicture}
\begin{axis}[trig format plots=rad,samples=500,domain=-0.5:0.5,ymax=0.7,ymin=-0.15,axis equal,hide axis]
\addplot[dashed,domain=-0.48:0.48] {(2/pi)*tan(abs(pi*\x))};
\filldraw[black] (0,0) circle (1pt) node[anchor=north]{$x$};
\draw[thick] (-5,0)--(5,0);

\draw[thick,pattern=north west lines,opacity=0.5] (0,0)--(0.25,0.5)--(-0.25,0.5)--cycle;
\draw (0.15,0.6)--(0.3,0.55);
\node [right] at (0.3,0.55) {$\mathcal{H}+x$};

\draw (0.05,0.3)--(0.3,0.2);
\node [right] at (0.3,0.2) {$\mathcal{V}+x$};
\end{axis}
\end{tikzpicture}

\caption{The set $\mathcal{V}$ is chosen to approximate $\mathcal{H}$ near the real axis, ensuring that $\tau_t^{(j)}$ and $\nu_t^{(j)}$ are close for large $t$.}\label{Fig:WhiteNoiseDecomp5}
\end{figure}

We next define for $t,s\geq 1$
\begin{align*}
X^{(H)}_{t,s}&:=\gamma H^{(1)}_{\rho^t}(x)-\gamma H^{(2)}_{\rho^s}(y)-\frac{\gamma^2}{2}\Big(\Var\big[H^{(1)}_{\rho^t}(x)\big]-\Var\big[H^{(2)}_{\rho^s}(y)\big]\Big)-(t-s)\log(1/\rho)\\
X^{(V)}_{t,s}&:=\gamma V^{(1)}_{\rho^t}(x)-\gamma V^{(2)}_{\rho^s}(y)-\left(\frac{\gamma^2}{2}+1\right)(t-s)\log(1/\rho).
\end{align*}
The former is the expression which we have to show is not too large in order for $\mathrm{Size}(t,s)$ to occur. Our final class of events justifying the approximations described in the previous subsection are
\begin{multline*}
\mathrm{Scal}(N):=\left\{\frac{\tau^{(1)}(x+B_N)}{\tau^{(1)}([0,1])}<2^{-1}\right\}\cap\left\{\frac{\tau^{(2)}(y+B_N)}{\tau^{(2)}([0,1])}<2^{-1}\right\}\\
\cap\left\{\sup_{t,s\in[N,5N]}\left\lvert X^{(H)}_{t,s}-X^{(H)}_{N,N}-(X^{(V)}_{t,s}-X^{(V)}_{N,N})\right\rvert<\log(2)\right\}
\end{multline*}
defined for $N\in\N$.
\medskip

\noindent\underline{Reduced events:} Armed with the four previous classes of events we can begin to apply our approximate decomposition to define simplified (or reduced) versions of $\mathrm{Shape}^{(j)}(t)$ and $\mathrm{Size}(t,s)$. For $t\geq 0$ with fractional part at most $5/8$, we define
\begin{multline*}
\mathrm{ShapeRed}^{(1)}(t):=\\
\left\{1\leq\tau_t^{(1)}(x+B_t\setminus B_{\lfloor t\rfloor+7/8})\rho^{-t}\leq 2^2 \right\}\cap\bigcap_{\ell=1}^8\left\{\frac{\tau_t^{(1)}(x+J_{t,\ell}\setminus B_{\lfloor t\rfloor+7/8})}{\tau_t^{(1)}(x+B_t\setminus B_{\lfloor t\rfloor+7/8})}\geq 2^{-4}\right\}\\
\cap\left\{\frac{\tau_t^{(1)}(x+2B_{t+1/4}\setminus B_{\lfloor t\rfloor+7/8})}{\tau_t^{(1)}(x+B_t\setminus B_{\lfloor t\rfloor+7/8})}\leq 2^{-30}\right\}\cap
\left\{\frac{\tau_t^{(1)}(x+2B_t\setminus B_{\lfloor t\rfloor+7/8})}{\tau_t^{(1)}(x+B_t\setminus B_{\lfloor t\rfloor+7/8})}\leq 2^2\right\}\\
\cap\left\{\sum_{m=0}^\infty\sum_{\ell\in S(m,\rho^{t+1/4},\rho^{t})}\frac{\tau_t^{(1)}(x+I_{m,\ell,t})^2}{\tau_t^{(1)}(x+B_t\setminus B_{\lfloor t\rfloor+7/8})^2}\leq 2^5\right\}.
\end{multline*}
Observe that each of the events in this intersection, apart from the first, are obtained from $\mathrm{Shape}^{(1)}(t)$ by replacing $\tau^{(1)}(x+\cdot)$ with $\tau^{(1)}_t(x+\cdot\setminus B_{\lfloor t\rfloor+7/8})$ and changing the constants slightly. The role of the first event will be explained soon. As the reader may have come to expect, we define $\mathrm{ShapeRed}^{(2)}(t)$ analogously on replacing $\tau^{(1)}$ with $\tau^{(2)}$ and $x$ with $y$.

Then for $N\in\N$ and $t,s\geq N$ we define
\begin{displaymath}
\mathrm{SizeRed}(t,s)=\left\{\left\lvert X^{(V)}_{t,s}-X_{N,N}^{(V)}+X_{N,N}^{(H)}+\log\left(\frac{\tau^{(2)}([0,1]\setminus (y+B_{N}))}{\tau^{(1)}([0,1]\setminus (x+B_{N}))}\right)\right\rvert\leq \log(2)\right\}.
\end{displaymath}

After controlling for the different types of `correlation events' described above, these reduced events imply their oirginal counterparts:

\begin{lemma}\label{l:Decompose}
Let $N\in\N$ and $t,s\in[N,5N]$. Suppose that $t-\lfloor t\rfloor,s-\lfloor s\rfloor\in[0,5/8]$. The intersection of the following events:
\begin{align*}
\mathrm{Upp}^{(1)}_{\lfloor t\rfloor},\mathrm{Low}_{\lfloor t\rfloor}^{(1)},\mathrm{Frac}^{(1)}_{\lfloor t\rfloor},\mathrm{ShapeRed}^{(1)}(t),\\
\mathrm{Upp}^{(2)}_{\lfloor s\rfloor},\mathrm{Low}_{\lfloor s\rfloor}^{(2)},\mathrm{Frac}^{(2)}_{\lfloor s\rfloor},\mathrm{ShapeRed}^{(2)}(s),\\
\mathrm{Centre}(N),\mathrm{Scal}(N),\mathrm{SizeRed}(t,s)
\end{align*}
is contained in $\mathrm{Ann}^\prime(x,t,y,s,N)$ (as defined in Lemma~\ref{l:Good_annuli}).
\end{lemma}
\begin{proof}
From the definition of $\mathrm{Ann}^\prime(x,t,y,s,N)$ we need to show that our list of events imply $\mathrm{Shape}^{(1)}(t)$, $\mathrm{Shape}^{(2)}(s)$ and $\mathrm{Size}(t,s)$. To ease notation we denote $n=\lfloor t\rfloor$ and $k=\lfloor s\rfloor$.

We argue first that $\mathrm{Size}(t,s)$ occurs. Applying \eqref{e:Upper-scale} for $j=1,2$ and \eqref{e:Low} and using $\mathrm{Frac}^{(1)}_n$ and $\mathrm{Frac}^{(2)}_k$ twice we see that
\begin{align*}
&\frac{\tau^{(1)}(x+B_t)}{\tau^{(1)}([0,1])}\frac{\tau^{(2)}([0,1])}{\tau^{(2)}(y+B_s)}\\
&\qquad\qquad\leq2^5\frac{\tau_t^{(1)}(x+B_t\setminus B_{n+7/8})\rho^{-t}}{\tau_s^{(2)}(y+B_s\setminus B_{k+7/8})\rho^{-s}}\exp\big(X_{t,s}^{(H)})\frac{\prod_{i=-1}^{n-1}\mathcal{S}_i^{(1)}(x+B_t)}{\prod_{i=-1}^{k-1}\mathcal{I}_i^{(2)}(y+B_s)}\frac{\tau^{(2)}([0,1])}{\tau^{(1)}([0,1])}\\
&\qquad\qquad\leq 2^9\exp\big(X_{t,s}^{(H)}\big)\frac{\tau^{(2)}([0,1])}{\tau^{(1)}([0,1])}
\end{align*}
where we have used the first parts of $\mathrm{ShapeRed}^{(1)}(t)$ and $\mathrm{ShapeRed}^{(2)}(s)$ to control the first term on the right hand side and $\mathrm{Upp}^{(1)}_{n}\cap\mathrm{Upp}^{(2)}_k$ to control the third term. Then using monotonicity of the measure $\tau^{(1)}$ and the second part of $\mathrm{Scal}(N)$
\begin{align*}
    \frac{\tau^{(2)}([0,1])}{\tau^{(1)}([0,1])}&\leq\frac{\tau^{(2)}([0,1])}{\tau^{(2)}([0,1]\setminus (y+B_N))}\frac{\tau^{(2)}([0,1]\setminus (y+B_N))}{\tau^{(1)}([0,1]\setminus (x+B_N))}\leq 2 \frac{\tau^{(2)}([0,1]\setminus (y+B_N))}{\tau^{(1)}([0,1]\setminus (x+B_N))}.
\end{align*}
Finally we observe that by $\mathrm{SizeRed}(t,s)$ and the third part of $\mathrm{Scal}(N)$
\begin{displaymath}
\exp\big(X_{t,s}^{(H)})\frac{\tau^{(2)}([0,1]\setminus (y+B_N))}{\tau^{(1)}([0,1]\setminus (x+B_N))}\leq 2^2.
\end{displaymath}
(This follows from summing the two terms inside absolute value signs for the stated events and taking the exponential.) Combining the three previous equations shows that
\begin{displaymath}
\frac{\tau^{(1)}(x+B_t)}{\tau^{(1)}([0,1])}\frac{\tau^{(2)}([0,1])}{\tau^{(2)}(y+B_s)}\leq 2^{12}.
\end{displaymath}
An entirely analogous argument applied to the reciprocal of this expression proves the lower bound of $2^{-12}$, verifying that $\mathrm{Size}(t,s)$ holds.

It remains to verify that $\mathrm{Shape}^{(1)}(t)$ occurs. (Verification of $\mathrm{Shape}^{(2)}(s)$ requires only minor notational changes and so is omitted.) From \eqref{e:Frac} and \eqref{e:Low}
\begin{align*}
\frac{\tau_t^{(1)}(x+B_{n+7/8})}{\tau_t^{(1)}(x+B_t\setminus  B_{n+7/8})}\leq \frac{\tau_t^{(1)}(x+B_{n+7/8})}{\tau_t^{(1)}(x+B_{n+5/8}\setminus B_{n+7/8})}\leq 2^2\frac{\tau_n^{(1)}(x+B_{n+7/8})}{\tau_n^{(1)}(x+B_{n+5/8}\setminus B_{n+7/8})}\leq 2^{-30}.
\end{align*}
From \eqref{e:Upper-scale}, \eqref{e:Frac} and the definitions of $\mathrm{Upp}_n^{(1)}$ and $\mathrm{Frac}_n^{(1)}$, for any $I,J\subset(x+2B_n)$ we have 
\begin{displaymath}
2^{-4}\frac{\tau_t^{(1)}(I)}{\tau_t^{(1)}(J)}\leq\frac{\tau^{(1)}(I)}{\tau^{(1)}(J)}\leq 2^4\frac{\tau^{(1)}_t(I)}{\tau^{(1)}_t(J)}.
\end{displaymath}
We use these two facts repeatedly to show that each part of $\mathrm{Shape}^{(1)}(t)$ is implied by the corresponding part of $\mathrm{ShapeRed}^{(1)}(t)$. To simplify notation, we assume below that $x=0$. The proof for general $x\in[0,1]$ follows from shifting every interval by $x$. 

By part 3 of $\mathrm{ShapeRed}^{(1)}(t)$,
\begin{align*}
\frac{\tau^{(1)}(2B_{t+1/4})}{\tau^{(1)}(B_t)}\leq 2^4\frac{\tau^{(1)}_t(2B_{t+1/4})}{\tau^{(1)}_t(B_t)}&\leq 2^4 \frac{\tau^{(1)}_t(2B_{t+1/4})}{\tau^{(1)}_t(B_t\setminus B_{n+7/8})}\\
&\leq 2^4\frac{\tau^{(1)}_t(2B_{t+1/4}\setminus B_{n+7/8})}{\tau^{(1)}_t(B_t\setminus B_{n+7/8})}+2^4\frac{\tau^{(1)}_t(B_{n+7/8})}{\tau^{(1)}_t(B_t\setminus  B_{n+7/8})}\leq 2^{-25}.
\end{align*}
This verifies the second part of $\mathrm{Shape}^{(1)}(t)$.

Similarly by part 2 of $\mathrm{ShapeRed}^{(1)}(t)$, for all $\ell=1,\dots, 8$
\begin{align*}
\frac{\tau^{(1)}(J_{t,\ell})}{\tau^{(1)}(B_t)}\geq 2^{-4}\frac{\tau^{(1)}_t(J_{t,\ell})}{\tau^{(1)}_t(B_t)}\geq 2^{-4}\frac{\tau^{(1)}_t(J_{t,\ell}\setminus B_{n+7/8})}{\tau^{(1)}_t(B_t)}\geq 2^{-5} \frac{\tau^{(1)}_t(J_{t,\ell}\setminus B_{n+7/8})}{\tau^{(1)}_t(B_t\setminus B_{n+7/8})}\geq 2^{-9}
\end{align*}
verifying the first part of $\mathrm{Shape}^{(1)}(t)$.

By part 4 of $\mathrm{ShapeRed}^{(1)}(t)$
\begin{align*}
\frac{\tau^{(1)}(2B_t)}{\tau^{(1)}(B_t)}\leq 2^4\frac{\tau^{(1)}_t(2B_t)}{\tau^{(1)}_t(B_t)}&\leq 2^4\frac{\tau^{(1)}_t(2B_t)}{\tau^{(1)}_t(B_t\setminus B_{n+7/8})}\\
&\leq 2^4\frac{\tau^{(1)}_t(2B_t\setminus B_{n+7/8})}{\tau^{(1)}_t(B_t\setminus B_{n+7/8})}+2^4\frac{\tau^{(1)}_t( B_{n+7/8})}{\tau^{(1)}_t(B_t\setminus B_{n+7/8})}\leq 2^7
\end{align*}
which verifies the third part of $\mathrm{Shape}^{(1)}(t)$.

Finally by part 5 of $\mathrm{ShapeRed}^{(1)}(t)$
\begin{displaymath}
\sum_m\sum_\ell\frac{\tau^{(1)}(I_{m,\ell,t})^2}{\tau^{(1)}(B_t)^2}\leq 2^8\sum_m\sum_\ell\frac{\tau^{(1)}_t(I_{m,\ell,t})^2}{\tau^{(1)}_t(B_t)^2}\leq 2^8\sum_m\sum_\ell\frac{\tau^{(1)}_t(I_{m,\ell,t})^2}{\tau^{(1)}_t(B_t\setminus B_{n+7/8})^2}\leq 2^{13}
\end{displaymath}
which proves the fourth and final part of $\mathrm{Shape}^{(1)}(t)$.
\end{proof}

\subsection{Verifying the sufficient conditions}
We now show that the events justifying our approximation occur sufficiently often. More precisely, we show that with high probability the events $\mathrm{Centre}(N)$ and $\mathrm{Scal}(N)$ occur, as well as the events $\mathrm{Upp}^{(j)}_n$, $\mathrm{Low}^{(j)}_n$ and $\mathrm{Frac}^{(j)}_n$ for many values of $n$. This will justify restricting our attention to the reduced events $\mathrm{ShapeRed}^{(j)}(t)$ and $\mathrm{SizeRed}(t,s)$ which we analyse in the next section.

The main result of this subsection is:
\begin{proposition}\label{p:LDApp}
For all $\rho>0$ sufficiently small, $\delta\in(0,1)$ and $\gamma<\gamma_0(\delta,\rho)$, there exists $C>0$ such that
\begin{displaymath}
    \P\big(\mathrm{Scal}\left(N\right)\cap\mathrm{Centre}(N)\big)>1-C\rho^{(2+3\epsilon)5N}
\end{displaymath}
and
\begin{equation}\label{e:GoodApprox}
\P\left(\mathrm{Approx}^{(j)}(N)\right)>1-C\rho^{(2+3\epsilon)5N}
\end{equation}
for $j=1,2$ and $N>N_0(\rho,\delta,\gamma)$, where
\begin{displaymath}
\mathrm{Approx}^{(j)}(N):=\left\{\sum_{n=1}^{5N}\ind_{\mathrm{Upp}^{(j)}_n\cap\mathrm{Low}_n^{(j)}\cap\mathrm{Frac}_n^{(j)}}>(1-\delta)5N\right\}.
\end{displaymath}
\end{proposition}
In proving this result, we will make repeated use of the two following estimates: a bound on the supremum of continuous Gaussian processes and a bound on ratios for the measures $\tau^{(j)}$.

\begin{lemma}\label{l:BTIS}
Let $Y_t$ be a continuous, centred Gaussian process indexed by $T:=[0,\ell_1]\times\dots\times[0,\ell_d]$ where $\ell_1\leq\dots\leq\ell_d$. Let $\sigma_T^2=\sup_{t\in T}\E[Y_t^2]$ and suppose that $Y_{t_0}=0$ almost surely for some $t_0\in T$. If $\E[\lvert Y_t-Y_s\rvert^2]\leq L\lvert t-s\rvert^\alpha$ for some $\alpha\in(0,1]$ and all $t,s\in T$, then for all $u\geq 0$
\begin{displaymath}
\P\Big(\sup_{t\in T}\lvert Y_t\rvert>u\Big)<C_{d,\alpha}\Big(1+L^{d/\alpha}\ell_d^d\Big(\frac{u}{\sigma_T^2}\Big)^{2d/\alpha}\Big)e^{-u^2/2\sigma_T^2}
\end{displaymath}
where $C_{d.\alpha}>0$ depends only on $d$ and $\alpha$.
\end{lemma}
\begin{proof}
We consider the normalised process $\Tilde{Y}_t:=(L\ell_d^\alpha)^{-1/2}Y_{(\ell_1 t_1,\dots,\ell_dt_d)}$ indexed by $[0,1]^d$. Defining the intrinsic metric $d(t,s)=\E[\lvert \Tilde{Y}_t-\Tilde{Y}_s\rvert^2]^{1/2}$ we have $d(t,s)\leq\lvert t-s\rvert^{\alpha/2}$. Therefore the minimal number of intrinsic metric balls of radius $\delta\in(0,1]$ required to cover $[0,1]^d$ is at most $C_{d}\delta^{-2d/\alpha}$. Since $Y_{t_0}=0$ we have $\tilde{\sigma}_T^2:=\sup_{t\in[0,1]^d}\E[\tilde{Y}_t^2]\leq d^{\alpha/2}$. These conditions allows us to apply the form of the Borell-TIS inequality stated in \cite[Theorem~4.1.2]{at07}. To be precise, using the notation of \cite{at07} we set $A:=\max\{d^{\alpha/4},C_d^{1/\alpha}\}$ and $\epsilon_0:=\min\{\tilde{\sigma}_T,1\}$. The bound follows for $\tilde{u}:=L^{-1/2}\ell_d^{-\alpha/2}u>\tilde{\sigma}_T^2(1+\sqrt{2d/\alpha})/\epsilon_0$ by the cited result whilst for the remaining values of $\tilde{u}$ we bound the probability by one and use the fact that 
\begin{displaymath}
    \exp\Big(-\frac{\tilde{u}^2}{2\tilde{\sigma}_T^2}\Big)\geq \exp\Big(-\frac{\tilde{\sigma}_T^2(1+\sqrt{2d/\alpha})^2}{2\epsilon_0^2}\Big)\geq\exp\Big(-\frac{1}{2}d^{\alpha/2}(1+\sqrt{
    2d/\alpha})^2\Big).
\end{displaymath}
\end{proof}

\begin{lemma}\label{l:momentbound}
Given $q_+>q_->0$, there exists $\gamma_0,C>0$ such that the following holds: for all $\gamma\leq\gamma_0$, $j=1,2$ and all $J,I\subset[-1/4,1/4]$ with $\lvert J\rvert\leq \lvert I\rvert$ and mutual distance $\leq 100\lvert I\rvert$
\begin{displaymath}
\P\left(\frac{\tau^{(j)}(J)}{\tau^{(j)}(I)}>\lambda\right)\leq C\lambda^{-q_+}\left(\frac{\lvert J\rvert}{\lvert I\rvert}\right)^{q_-}.
\end{displaymath}
\end{lemma}
\begin{proof}
For the proof we denote $\nu=\nu^{(j)}$ (recall the definition in \eqref{e:MeasureDef2}) and $\tau=\tau^{(j)}$ since the index $j$ plays no role. The proof of \cite[Lemma~4.4]{ajks} shows that for any $\gamma<\sqrt{2}$, $1<q_+<p<2/\gamma^2$ and $I,J$ as above
\begin{equation}\label{e:MeasureRatio}
\left\|\frac{\nu(J)}{\nu(I)}\right\|_{q_+}\leq C\left(\frac{\lvert J\rvert}{\lvert I\rvert}\right)^{\zeta_p(\gamma)/p}
\end{equation}
where $\|X\|_q:=\E[\lvert X\rvert^q]^{1/q}$, $C$ depends only on $q_+$ and $p$ and we recall that
\begin{displaymath}
\zeta_p(\gamma)=p\left(1-\frac{p-1}{2}\gamma^2\right)
\end{displaymath}
denotes the multifractal spectrum of the measures. Lemma 3.6 of \cite{ajks} states that there exists a positive random variable $G$ with polynomial moments of all orders such that $G^{-1}\nu(\cdot)\leq\tau(\cdot)\leq G\nu(\cdot)$. Moreover $G$ does not depend on $\gamma$. We see then by H\"older's inequality that for $\gamma$ sufficiently small and $p>2q_+$
\begin{displaymath}
\E\left[\left(\frac{\tau(J)}{\tau(I)}\right)^{q_+}\right]\leq \E[G^{4q_+}]^{1/2}\E\left[\left(\frac{\nu(J)}{\nu(I)}\right)^{2q_+}\right]^{1/2}\leq C^\prime\left(\frac{\lvert J\rvert}{\lvert I\rvert}\right)^{q_+\zeta_p(\gamma)/p}.
\end{displaymath}
By Markov's inequality
\begin{displaymath}
\P\left(\frac{\tau(J)}{\tau(I)}>\lambda\right)\leq C^\prime\lambda^{-q_+}\left(\frac{\lvert J\rvert}{\lvert I\rvert}\right)^{q_+(1-\frac{p-1}{2}\gamma^2)}.
\end{displaymath}
After restricting $\gamma_0$ sufficiently small, this exponent matches (or exceeds) the desired one.
\end{proof}

We can now give a quantitative statement that the increments of $\nu_t$ and $\tau_t$ are close when $t$ is large:

\begin{lemma}\label{l:MeasureComp}
    For $0<\delta\leq r\leq 1/4$ and $t\in[0,1]$ let
    \begin{displaymath}
        U_\delta^r(t):=H_\delta^{(1)}(t)-H_r^{(1)}(t)-\big(V_\delta^{(1)}(t)-V_r^{(1)}(t)\big).
    \end{displaymath}
    Then there exist absolute constants $c_0,C_1>0$ such that 
    \begin{displaymath}
        \sup_{\delta\in(0,r]}\Var[U_\delta^r(0)]\in[c_0r^2,r^2]
    \end{displaymath}
    and for $u\geq 0$
    \begin{displaymath}
        \P\left(\sup_{\delta\in(0,r]}\sup_{t\in[0,1]}\lvert U_\delta^r(t)\rvert>u\right)\leq C_1\big(1+u^6/r^{12}\big)e^{-u^2/2r^2}.
    \end{displaymath}
\end{lemma}
\begin{proof}
    Recalling that $\lambda$ denotes the hyperbolic measure on the upper half plane, by definition of the fields $H^{(1)}$ and $V^{(1)}$
    \begin{displaymath}
        \Var[U_\delta^r(0)]=\lambda\Big((\mathcal{V}\setminus\mathcal{H})\cap\{(a,b)\in\H\;|\;\delta\leq b\leq r\}\Big)=\int_\delta^r\frac{2}{b^2}\Big(\frac{b}{2}-\frac{1}{\pi}\arctan(\pi b/2)\Big)\;db.
    \end{displaymath}
    Using the inequality
    \begin{displaymath}
        \frac{\pi^2b^3}{24}-\frac{\pi^4b^5}{160}\leq\frac{b}{2}-\frac{\arctan(\pi b/2)}{\pi}\leq \frac{\pi^2b^3}{24}
    \end{displaymath}
     which follows from the series expansion of $\arctan(\cdot)$, the latter integral is at most
     \begin{displaymath}
         \int_\delta^r b\;db=(r^2-\delta^2)/2\leq r^2/2
     \end{displaymath}
     and at least
     \begin{displaymath}
         \int_\delta^r \frac{\pi^2}{12}b-\frac{\pi^4}{80}b^3\;db\geq\frac{\pi^2}{24}(r^2-\delta^2)-\frac{\pi^4}{320}(r^4-\delta^4)\geq \big(\frac{\pi^2}{24}-\frac{\pi^2}{32}\big)(r^2-\delta^2)
     \end{displaymath}
     proving the first statement of the lemma.

     Aiming towards Lemma~\ref{l:BTIS}, for $0<\eta\leq\delta\leq r$ and $0\leq s\leq t\leq 1$
     \begin{equation}\label{e:MeasureComp1}
         \E\left[(U^r_\delta(t)-U^r_\eta(s))^2\right]=\E\left[(U^r_\delta(t)-U^r_\delta(s))^2\right]+\E\left[(U^r_\delta(s)-U^r_\eta(s))^2\right]
     \end{equation}
     using the fact that these terms depend on distinct regions of the white noise. By stationarity and our previous argument
     \begin{equation}\label{e:MeasureComp2}
         \E\left[(U^r_\delta(s)-U^r_\eta(s))^2\right]=\E[U^\delta_\eta(0)^2]\leq \delta-\eta.
     \end{equation}
    Similarly by definition of the white noise $W$
     \begin{equation}\label{e:MeasureComp3}
         \E\left[(U^r_\delta(t)-U^r_\delta(s))^2\right]=\lambda\left((\mathcal{V}\setminus\mathcal{H})\Delta(\mathcal{V}\setminus\mathcal{H}+t-s)\cap\{(a,b)\;|\;\delta\leq b\leq r\}\right)
     \end{equation}
     where $\Delta$ denotes the symmetric difference. Using the inequality for $\arctan(\cdot)$ above, the one-dimensional measure of
     \begin{displaymath}
         (\mathcal{V}\setminus\mathcal{H})\Delta(\mathcal{V}\setminus\mathcal{H}+t-s)\cap\{(a,b)\;|\;b=b_0\}
     \end{displaymath}
     is bounded above by $\min\{2(t-s),2b_0^3\}$ (see Figure~\ref{Fig:WhiteNoiseDecomp5}). Note that since $r\leq 1/4$ the periodicity of the white noise plays no role. Therefore \eqref{e:MeasureComp3} is bounded above by
     \begin{displaymath}
         \int_\delta^r\frac{2}{b^2}\min\{t-s,b^3\}\;db\leq
         \int_0^{\sqrt[3]{t-s}}2b\;db+\int_{r\wedge \sqrt[3]{t-s}}^r2(t-s)/b^2\;db\leq 3(t-s)^{2/3}.
     \end{displaymath}
     Combining this with \eqref{e:MeasureComp1}, \eqref{e:MeasureComp2} and \eqref{e:MeasureComp3} we may apply Lemma~\ref{l:BTIS} for $\alpha=2/3$ which yields the statement of the lemma.
\end{proof}

We now prove Proposition~\ref{p:LDApp} in a series of steps:

\begin{proof}[Proof of Proposition~\ref{p:LDApp} for $\mathrm{Scal}(N)$ and $\mathrm{Centre}(N)$]
Recall that $\epsilon>0$ was fixed in Section~\ref{s:Holder}. We choose $q_-=(2+3\epsilon)/\epsilon+1$ and $q_+=q_-+1$. By the definition of $\mathrm{Centre}(N)$ and Lemma~\ref{l:momentbound}, for $\gamma<\gamma_0$ sufficiently small
\begin{align*}
    \P(\mathrm{Centre}(N)^c)=\P\left(\frac{\tau^{(2)}([0,\rho^{(1+\epsilon)5N}])}{\tau^{(2)}(B_{5N+1})}>2^{-25}\right)&\leq C 2^{25q_+}\left(\frac{\rho^{(1+\epsilon)5N}}{2\rho^{5N+1}}\right)^{q_-}\\
    &\leq C^\prime2^{24q_-}\rho^{(5\epsilon N-1)q_-}.
\end{align*}
By our definition of $q_-$, the right hand side is at most $C_\epsilon\rho^{(2+3\epsilon)5N}$ for $N$ sufficiently large.

Turning to $\mathrm{Scal}(N)$; by Lemma~\ref{l:momentbound}, for any $1<q_-<q_+$ we can find $\gamma_0>0$ such that for $\gamma\in[0,\gamma_0]$
\begin{align*}
\P\left(\frac{\tau^{(j)}(B_{N})}{\tau^{(j)}([0,1])}>\frac{1}{2}\right)\leq\P\left(\frac{\tau^{(j)}(B_{N})}{\tau^{(j)}([-1/4,1/4])}>\frac{1}{2}\right)\leq C 2^{q_+}\left(4\rho^{N}\right)^{q_-}.
\end{align*}
For $q_->5(2+3\epsilon)$ and $q_+=q_-+1$ the right hand side is at most $C_\epsilon\rho^{(2+3\epsilon)5N}$.

For the final part of $\mathrm{Scal}(N)$ we write
\begin{equation*}
    X_{t,s}^{(H)}-X_{N,N}^{(H)}-(X_{t,s}^{(V)}-X_{N,N}^{(V)})=I_t+II_s+III
\end{equation*}
where
\begin{align*}
    I_t&:=\gamma\Big(H_{\rho^t}^{(1)}(x)-H_{\rho^N}^{(1)}(x)-(V_{\rho^t}^{(1)}(x)-V_{\rho^N}^{(1)}(x))\Big)\\
    II_s&:=\gamma\Big(H_{\rho^s}^{(2)}(y)-H_{\rho^N}^{(2)}(y)-(V_{\rho^s}^{(2)}(y)-V_{\rho^m}^{(2)}(y))\Big)\\
    III&:=\frac{\gamma^2}{2}\Big(\Var\big[H_{\rho^t}^{(1)}(0)-H_{\rho^s}^{(1)}(0)\big]-\Var\big[V_{\rho^t}^{(1)}(0)-V_{\rho^s}^{(1)}(0)\big]\Big).
\end{align*}
We bound the suprema of each of these terms separately. By the first statement of Lemma~\ref{l:MeasureComp}, $\lvert III\rvert\leq\gamma^2\rho^{2N}$. Taking $\rho$ sufficiently small this is less than $\log(2)/3$ for all $N$. By the second statement of Lemma~\ref{l:MeasureComp} with $u=\log(2)/3$ and $r=\rho^N$
\begin{displaymath}
    \P\left(\sup_{t\geq N}\lvert I_t\rvert>\log(2)/3\right)\leq C^\prime(1+\rho^{-12N})\exp(-C\rho^{-2N})
\end{displaymath}
for absolute constants $C,C^\prime>0$. In particular the right-hand side is at most $C^{\prime\prime}\rho^{(2+3\epsilon)5N}$. The term $II_s$ has an identical distribution, so the same probability bound holds. Combining this with the bound on $III$ we conclude that 
\begin{displaymath}
    \P\left(\sup_{t,s\in[N,5N]}\left\lvert X_{t,s}^{(H)}-X_{N,N}^{(H)}-(X_{t,s}^{(H)}-X_{N,N}^{(H)})\right\rvert>\log(2)\right)<2C^{\prime\prime}\rho^{(2+3\epsilon)5N}
\end{displaymath}
for all $N$ sufficiently large and hence
\begin{displaymath}
    \P\big(\mathrm{Scal}(N)^c\big)<C^{\prime\prime\prime}\rho^{(2+3\epsilon)5N}.
\end{displaymath}
\end{proof}

\begin{lemma}\label{l:LDFrac}
For all $\rho>0$ sufficiently small, $\delta\in(0,1)$ and $\gamma<\gamma_0(\delta,\rho)$, there exists $C>0$ such that
\begin{displaymath}
\P\left(\sum_{n=1}^{5N}\ind_{\mathrm{Frac}_n^{(j)}}>(1-\delta)5N\right)>1-C\rho^{(2+3\epsilon)5N}
\end{displaymath}
for $j=1,2$ and $N>N_0(\rho,\delta,\gamma)$.
\end{lemma}
\begin{proof}
The key to the proof is that $\mathrm{Frac}_1^{(j)},\mathrm{Frac}_2^{(j)},\dots$ are independent (as they depend on the white noise $W$ on disjoint regions) and so we can apply standard large deviation estimates for Bernoulli variables. Specifically if $Z_1,Z_2,\dots,$ are i.i.d.\ Bernoulli variables with parameter $p>1-\delta$, then a simple argument using the exponential Markov inequality shows that
\begin{displaymath}
    \P\left(\sum_{i=1}^{5N} Z_i<(1-\delta)5N\right)\leq e^{-5NI(\delta,p)}
\end{displaymath}
where
\begin{displaymath}
    I(\delta,p)=\delta\log\Big(\frac{\delta}{1-p}\Big)+(1-\delta)\log\Big(\frac{1-\delta}{p}\Big).
\end{displaymath}
Clearly for any value of $\delta\in(0,1)$, $I(\delta,p)\to\infty$ as $p\to 1$. Moreover if we instead suppose that $Z_i\sim\mathrm{Ber}(p_i)$ where $p_i\geq p$ for all $i$, then the above estimate still holds (to see this one can define $Z_i=\ind_{U_i\leq p_i}$ and $\Tilde{Z}_i=\ind_{U_i\leq p}$ where the $U_i$ are i.i.d.\ uniform variables on $[0,1]$). Therefore, if for any $p\in(0,1)$ we can choose $\gamma_0$ (depending on $\rho$) such that $\P(\mathrm{Frac}_n^{(j)})\geq p$ for all $n$, then the statement of the lemma follows.

Recall that $\mathrm{Frac}_n^{(1)}$ is the indicator function for the event that
\begin{equation}
    \sup_{t\in[n,n+5/8]}\sup_{u\in (x+2B_n)}\Big\lvert\gamma(H^{(1)}_{\rho^t}(u)-H^{(1)}_{\rho^n}(u))-\frac{\gamma^2}{2}\Var\big[H_{\rho^t}(0)-H_{\rho^n}(0)\big]\Big\rvert<\log(2).
\end{equation}
We first observe that
\begin{align*}
    \Var\big[H_{\rho^t}(0)-H_{\rho^n}(0)\big]=\int_{\mathcal{H}\cap\{\rho^t\leq y\leq\rho^n\}}y^{-2}\;dxdy&\leq \int_{\rho^t}^{\rho^n}y^{-2}\int_{-y/2}^{y/2}\;dxdy\\
    &=(t-n)\log(1/\rho)\leq\frac{5}{8}\log(1/\rho).
\end{align*}
Therefore taking $\gamma$ sufficiently small relative to $\rho$, we have $\frac{\gamma^2}{2}\Var\big[H_{\rho^t}(0)-H_{\rho^n}(0)\big]<\log(2)/2$ for all $t$ and $n$.

Next we define $Y_{t,u}=H_{\rho^t}(x+\rho^nu)-H_{\rho^n}(x+\rho^nu)$ which is a continuous centred Gaussian process on $(t,u)\in[n,n+5/8]\times[-2,2]$. Then by independence of the white noise on disjoint domains, for $t>s$
\begin{align*}
    \E\big[\lvert Y_{t,u}-Y_{s,w}\rvert^2]=\E\big[\lvert Y_{t,u}-Y_{s,u}\rvert^2\big]+\E\big[\lvert Y_{s,u}-Y_{s,w}\rvert^2\big].
\end{align*}
Our previous calculation shows that the first term on the right hand side is bounded above by $\lvert t-s\rvert\log(1/\rho)$. The second term is given by
\begin{align*}
    \lambda\big((\rho^nu+\mathcal{H})\Delta(\rho^nw+\mathcal{H})\cap\{\rho^s\leq y\leq\rho^n\}\big)\leq 2\rho^n\lvert u-w\rvert\int_{\rho^s}^{\rho^n} y^{-2}\;dy\leq 2\rho^{-5/8}\lvert u-w\rvert
\end{align*}
where $\Delta$ denotes the symmetric difference, so that
\begin{displaymath}
    \E\big[\lvert Y_{t,u}-Y_{s,w}\rvert^2]\leq 2\rho^{-5/8}(\lvert u-w\rvert+\lvert t-s\rvert).
\end{displaymath}
Finally we observe that
\begin{displaymath}
    \sup_{t,u}\E[Y_{t,u}^2]=\E[Y_{n+5/8,0}^2]\in[c_0\log(1/\rho),c_1\log(1/\rho)]
\end{displaymath}
for some $0<c_0<c_1$ independent of $n$. Applying Lemma~\ref{l:BTIS} to $Y_{t,u}$ with $\alpha=1$ we see that for any $a>0$
\begin{displaymath}
    \P\left(\sup_{t\in[n,n+5/8],u\in(x+2B_n)}\gamma\lvert Y_{t,u}\rvert>a\right)<C(1+\rho^{-5/4}(a/\gamma)^4)e^{-a^2/S\gamma^22c_1\log(1/\rho)}
\end{displaymath}
for an absolute constant $C>0$. Taking $a=\log(2)/2$ and $\gamma$ sufficiently small, depending on $\rho$, ensures that the probability of the supremum exceeding $\log(2)/2$ is as small as desired. Therefore for each $n$, $\mathrm{Frac}_n^{(j)}$ will occur with sufficiently high probability completing the proof of the lemma.
\end{proof}

It remains to prove that the events $\mathrm{Upp}_n^{(j)}$ and $\mathrm{Low}_n^{(j)}$ occur on a sequence of density $1-\delta$. This will be more challenging due to the dependence between these events. We follow the arguments of \cite[Section~4.3]{ajks}, although in some cases our proofs simplify since we consider only small values of $\gamma$.

\begin{lemma}\label{l:LDLow}
For all $\rho>0$ sufficiently small, $\delta\in(0,1)$ and $\gamma<\gamma_0(\delta,\rho)$, there exists $C>0$ such that
\begin{displaymath}
\P\left(\sum_{n=1}^{5N}\ind_{\mathrm{Low}_n^{(j)}}>(1-\delta)5N\right)>1-C\rho^{(2+3\epsilon)5N}
\end{displaymath}
for $j=1,2$ and $N\in\N$.
\end{lemma}

\begin{proof}
For the proof we write $\mathrm{Low}_n$ rather than $\mathrm{Low}_n^{(j)}$ since the index $j=1,2$ will play no role. First we claim that for any $q\geq 1$ there exists $C>0$ such that for all $\rho,\gamma>0$ sufficiently small
\begin{displaymath}
    \P(\mathrm{Low}_{n,k}^c)<C\rho^{q(k-n+1/4)}.
\end{displaymath}
Indeed if we choose $\rho<2^{-(q+2)}$ then we can apply Lemma~\ref{l:momentbound} with $q_+=q_-+1=q+2$ to see that for $\gamma$ sufficiently small
\begin{align*}
\P(\mathrm{Low}_{n,k}^c)=\P\left(\frac{\tau_n(B_{k+7/8}\setminus B_{k+1+7/8})}{\tau_n(B_{n+5/8}\setminus B_{n+7/8})}> 2^{n-k-33}\right)&\leq C_0 2^{q_+(33+k-n)}\rho^{q_-(k-n+1/4)}\\
&\leq C\rho^{q(k-n+1/4)}
\end{align*}
where $C$ depends only on $q$.

Turning to the estimate we wish to prove; by the union bound
\begin{equation}\label{e:LowUnionBound}
\P\left(\sum_{n=1}^{5N}\ind_{\mathrm{Low}_n^c}>\delta 5N\right)\leq\sum_{B\subseteq\{1,\dots,5N\},\lvert B\rvert>\delta 5N}\P\left(\bigcap_{n\in B}\mathrm{Low}_n^c\right).
\end{equation}
In order to bound the expression on the right, we consider a sequence $n_1<n_2<\dots<n_M$ and recalling that $\mathrm{Low}_n=\cap_{\ell\geq 0}\mathrm{Low}_{n,n+\ell}$ we have
\begin{align*}
    \P(\mathrm{Low}_{n_1}^c\cap\dots\cap \mathrm{Low}_{n_M}^c)\leq \sum_{\ell_1=0}^\infty\P\big(\mathrm{Low}^c_{n_1,n_1+\ell_1}\cap\mathrm{Low}_{n_{i_2}}^c\cap\dots\cap\mathrm{Low}_{n_M}^c\big)
\end{align*}
where $i_2$ is chosen as the first index $j$ such that $n_1+\ell_1+1<n_j$. Iterating this argument we obtain that
\begin{align*}
    \P(\mathrm{Low}_{n_1}^c\cap\dots\cap \mathrm{Low}_{n_M}^c)\leq \sum_{r=1}^M\sum_{\ell_1,\dots,\ell_r}\P\left(\bigcap_{k=1}^r\mathrm{Low}^c_{n_{i_k},n_{i_k}+\ell_k}\right)
\end{align*}
where $i_k$ is defined recursively as the first $j$ such that $n_j>n_{i_{k-1}}+\ell_{k-1}+1$ and we sum over all values of $\ell_1,\dots,\ell_r$ such that
\begin{equation}\label{e:low_prob}
    \{n_1,\dots,n_M\}\subset\bigcup_{k=1}^r[n_{i_k},n_{i_k}+\ell_k+1].
\end{equation}
Now using the observation from Section~\ref{ss:Sufficient} that $\mathrm{Low}_{n_1,k_1}$ and $\mathrm{Low}_{n_2,k_2}$ are independent whenever $k_1+2\leq n_2$ along with the estimate above we have
\begin{displaymath}
    \P(\mathrm{Low}_{n_1}^c\cap\dots\cap \mathrm{Low}_{n_M}^c)\leq\sum_{r=1}^M\sum_{\ell_1,\dots,\ell_r}C^r\rho^{q\sum_{k=1}^r(\ell_k+1/4)}.
\end{displaymath}
Next we observe that since each interval on the right of \eqref{e:low_prob} contains at most $\ell_k+2$ integers, we must have $M\leq\sum_{k=1}^r\ell_k+2r$. Substituting this into the equation above
\begin{align*}
    \P(\mathrm{Low}_{n_1}^c\cap\dots\cap \mathrm{Low}_{n_M}^c)&\leq\sum_{r=1}^MC^r\rho^{\frac{q}{10}(M-2r)}\sum_{\ell_1,\dots,\ell_r}\rho^{\frac{9q}{10}\sum_{k=1}^r(\ell_k+1/4)}\\
    &\leq\sum_{r=1}^MC^r\rho^{\frac{q}{10}(M-2r)}\left(\sum_{\ell=0}^\infty\rho^{\frac{9q}{10}(\ell+1/4)}\right)^r\leq\sum_{r=1}^MC^r\rho^{\frac{q}{10}(M-2r)}(C^\prime)^r\rho^{\frac{9qr}{40}}
\end{align*}
where $C^\prime>0$ is an absolute constant (provided $\rho\leq 1/2$, say). The latter expression can be rearranged as
\begin{align*}
    \rho^{\frac{q}{10}M}\sum_{r=1}^M(CC^{\prime})^r\rho^{\frac{qr}{40}}\leq M\rho^{\frac{q}{10}M}\leq C^{\prime\prime}\rho^{\frac{q}{11}M}.
\end{align*}
where we have assumed that $\rho^{q/40}<1/(CC^\prime)$ (recall that $C$ depends on $q$ but not $\rho$ or $\gamma$ and $C^\prime$ is an absolute constant). Substituting this bound into \eqref{e:LowUnionBound} we see that
\begin{align*}
    \P\left(\sum_{n=1}^{5N}\ind_{\mathrm{Low}_n^c}>\delta 5N\right)\leq 2^{5N}C^{\prime\prime}\rho^{\frac{5q}{11}\delta N}.
\end{align*}
Choosing $q$ sufficiently large (depending on $\delta$) ensures that the right hand side is at most $C^{\prime\prime\prime}\rho^{(2+3\epsilon)5N}$, completing the proof of the lemma.
\end{proof}

\begin{lemma}\label{l:upscale}
There exist absolute constants $C_0,c_0>0$ such that for all $n\in\N$ and all $k\in\{-1,0,\dots,n-1\}$
\begin{displaymath}
\P\left(\mathrm{Upp}_{n,k}^{(j)}\right)\geq 1-C_0\exp\left(\frac{-c_0}{\gamma^2(4\rho)^{n-k-1}}\right).
\end{displaymath}
\end{lemma}
\begin{proof}
As in previous proofs we will omit the superscript $j$ as it plays no role here. Recalling the definition of $\mathrm{Upp}_{n,k}$ at the start of Section~\ref{ss:Sufficient} and noting that
\begin{displaymath}
    \frac{E(u,\rho^{k+1},\rho^k)}{E(x,\rho^{k+1},\rho^k)}=\exp\left(\gamma\Big(H_{\rho^{k+1}}(u)-H_{\rho^k}(u)-H_{\rho^{k+1}}(x)+H_{\rho^k}(x)\Big)\right)
\end{displaymath}
we see that a sufficient condition for $\mathrm{Upp}_{n,k}$ is that
\begin{equation}\label{e:UpperSufficient}
    \sup_{u\in x+2B_n}\gamma\Big(\left\lvert H_{\rho^{k+1}}(u)-H_{\rho^{k+1}}(x)\right\rvert+\left\lvert H_{\rho^k}(u)-H_{\rho^k}(x)\right\rvert\Big)\leq 2^{k-n}\log(2).
\end{equation}
(Note this is true even for $k=-1$, although in this case we can ignore the $H_{\rho^k}$ terms in the above expression.)

Letting $Y_u=H_{\rho^k}(u)-H_{\rho^k}(x)$ for $u\in x+2B_n$ we have
\begin{displaymath}
    \E[(Y_u-Y_w)^2]\leq\lambda((u+\mathcal{H}_{\rho^k})\Delta(w+\mathcal{H}_{\rho^k}))\leq 2\lvert u-w\rvert\int_{\rho^k}^\infty y^{-2}\;dy=2\rho^{-k}\lvert u-w\rvert.
\end{displaymath}
In particular $\sup_{u\in x+2B_n}\E[Y_u^2]\leq 4\rho^{n-k}$. A simple geometric argument shows this bound is tight up to constants: that is $\sup_{u\in x+2B_n}\E[Y_u^2]\geq \E[Y_{x+\rho^n/2}^2]\geq c_0\rho^{n-k}$ for some $c_0>0$ and all $k\leq n$. Since $Y_x=0$, applying Lemma~\ref{l:BTIS} with $\alpha=1$, $L=2\rho^{-k}$ and $\ell_d=4\rho^{n}$ we have
\begin{displaymath}
\P\left(\sup_{u\in x+ 2B_n}\gamma\left\lvert H_{\rho^k}(u)-H_{\rho^k}(x)\right\rvert> a\right)\leq C\Big(1+\rho^{n-k}\frac{a^2}{\gamma^2\rho^{n-k}}\Big)e^{-a^2/8\gamma^2\rho^{n-k}}
\end{displaymath}
for any $a\geq 0$ where $C>0$ is an absolute constant. Applying this to both expressions on the left hand side of \eqref{e:UpperSufficient} (with $a=2^{k-n-1}\log(2)$) we see that
\begin{displaymath}
    \P(\mathrm{Upp}_{n,k}^c)\leq c\left(1+\gamma^{-2}4^{k-n}\right)\exp\left(-\frac{c^\prime}{\gamma^2(4\rho)^{n-k-1}}\right)
\end{displaymath}
where $c,c^\prime>0$ are absolute constants. Since $x\mapsto (1+1/x)e^{-c^\prime/(2x)}$ is bounded above near $x=0$ by an absolute constant, the expression above satisfies the bound in the statement of the lemma.
\end{proof}

\begin{lemma}\label{l:LDUpp}
For all $\rho>0$ sufficiently small, $\delta\in(0,1)$ and $\gamma<\gamma_0(\delta,\rho)$, there exists $C>0$ such that
\begin{displaymath}
\P\left(\sum_{n=1}^{5N}\ind_{\mathrm{Upp}_n^{(j)}}>(1-\delta)5N\right)>1-C\rho^{(2+3\epsilon)5N}
\end{displaymath}
for $j=1,2$ and $N\in\N$.
\end{lemma}

\begin{proof}
We once more suppress the superscript $j$ in our arguments, which roughly follow the proof of Lemma~\ref{l:LDLow}. By the union bound
\begin{equation}\label{e:UpperUnion}
\P\left(\sum_{n=1}^{5N}\ind_{\mathrm{Upp}_n^c}>\delta 5N\right)\leq\sum_{B\subset\{1,\dots,5N\},\lvert B\rvert>\delta 5N}\P\left(\bigcap_{n\in B}\mathrm{Upp}_n^c\right).
\end{equation}
For a sequence $n_1>\dots>n_M$, arguing as in the proof of Lemma~\ref{l:LDLow} we have
\begin{align*}
    \P(\mathrm{Upp}_{n_1}^c\cap\dots\cap\mathrm{Upp}_{n_M}^c)\leq\sum_{r=1}^M\sum_{\ell_1,\dots,\ell_r}\P\left(\bigcap_{k=1}^r\mathrm{Upp}^c_{n_{i_k},n_{i_k}-\ell_k}\right)
\end{align*}
where now $n_{i_{k+1}}$ is the largest $n_j$ less than $n_{i_k}-\ell_k$ and we sum over $\ell_1,\dots,\ell_r$ such that
\begin{displaymath}
    \{n_1,\dots,n_M\}\subset\bigcup_{k=1}^r[n_{i_k},n_{i_k}-\ell_k].
\end{displaymath}
The latter condition implies that $M\leq\sum_{k=1}^r\ell_k+r$.

Next we observe that the events $\mathrm{Upp}_{n_{i_k},n_{i_k}-\ell_k}$ are independent for distinct $k$ since they depend on disjoint horizontal strips of the white noise process. Therefore using Lemma~\ref{l:upscale}
\begin{align*}
    \P(\mathrm{Upp}_{n_1}^c\cap\dots\cap\mathrm{Upp}_{n_M}^c)\leq\sum_{r=1}^M\sum_{\ell_1,\dots,\ell_r}C_0^r\exp\left(\sum_{k=0}^r\frac{-c_0}{\gamma^2(4\rho)^{\ell_k-1}}\right).
\end{align*}
Now we choose $\rho>0$ small enough that $(4\rho)^{-\ell}\geq 1+\ell$ for all $\ell\geq 0$ and then $\gamma<4\rho$. Using this and $\sum_k\ell_k\geq M-r$, the above expression is bounded by
\begin{align*}
    \sum_{r=1}^M\sum_{\ell_1,\dots,\ell_r}C_0^r\exp\left(\frac{-c_0}{\gamma}\Big(r+\sum_{k=0}^r\ell_k\Big)\right)&\leq \exp\left(\frac{-c_0}{2\gamma}M\right)\sum_{r=1}^MC_0^r\left(\sum_{\ell\geq 0}\exp\left(\frac{-c_0}{2\gamma}(1+\ell)\right)\right)^r\\
    &\leq \exp\left(\frac{-c_0}{2\gamma}M\right)\sum_{r=1}^MC_0^rC_1^r\exp\left(\frac{-c_0}{2\gamma}r\right)
\end{align*}
for some absolute constant $C_1>0$. Choosing $\gamma$ small enough that $C_0C_1e^{-c_0/(2\gamma)}\leq 1$ we conclude that
\begin{displaymath}
    \P(\mathrm{Upp}_{n_1}^c\cap\dots \mathrm{Upp}_{n_M}^c)\leq  M\exp\left(\frac{-c_0}{2\gamma}M\right) \leq C^{\prime}\exp\left(\frac{-c_0^\prime}{2\gamma}M\right).
\end{displaymath}
Substituting this into \eqref{e:UpperUnion} we obtain
\begin{displaymath}
    \P\left(\sum_{n=1}^{5N}\ind_{\mathrm{Upp}_n^c}>\delta 5N\right)\leq 2^{5N}C^{\prime}\exp\left(\frac{-c_0^\prime}{2\gamma}\delta 5N\right)
\end{displaymath}
and taking $\gamma$ sufficiently small ensures that this is at most $C\rho^{(2+3\epsilon)5N}$ as required.
\end{proof}

\begin{proof}[Proof of Proposition~\ref{p:LDApp}]
    We have already proven the required probability bound for $\mathrm{Scal}(\cdot)$ and $\mathrm{Centre}(\cdot)$. Combining Lemmas~\ref{l:LDFrac}, \ref{l:LDLow} and \ref{l:LDUpp} with the union bound proves \eqref{e:GoodApprox} with $3\delta$ replacing $\delta$. Since $\delta\in(0,1)$ was arbitrary, this suffices to prove the proposition.
\end{proof}

\section{Large deviation bounds for reduced events}\label{s:LargeDeviations}
In this section we tie together the remaining arguments needed to show that the maps $F_n$ are uniformly H\"older continuous and so complete the proof of our conformal welding result.

\subsection{Completing the proof of H\"older continuity}\label{ss:CompletingHolder}
The final ingredient we require to prove the uniform H\"older continuity of Theorem~\ref{t:Holder} is a large deviation bound for the reduced size and shape events described in the previous section:

\begin{proposition}\label{p:LDRed}
    For some $\delta>0$, all $\rho>0$ sufficiently small, all $\gamma<\gamma_0(\rho)$ and all $N$ sufficiently large the following holds: there exist two random sequences $(t_n)_{n\in\N}$ and $(s_n)_{n\in\N}$ such that
    \begin{enumerate}
        \item $\lfloor t_{n}\rfloor-\lfloor t_{n-1}\rfloor,\lfloor s_n\rfloor-\lfloor s_{n-1}\rfloor\geq 1$ for all $n$
        \item $t_n-\lfloor t_n\rfloor,s_n-\lfloor s_n\rfloor\in[0,5/8]$ for all $n$
        \item for some $C>0$ independent of $N$
        \begin{displaymath}
        \P\left(\sum_{n\::\:N\leq t_n,s_n\leq 5N}\ind_{\mathrm{ShapeRed}^{(1)}(t_n)\cap\mathrm{ShapeRed}^{(2)}(s_n)\cap\mathrm{SizeRed}(t_n,s_n)}<11\delta N\right)\leq C\rho^{(2+3\epsilon)5N}.
    \end{displaymath}
    \end{enumerate}
\end{proposition}

Armed with this result, and our previous arguments, the proof of Theorem~\ref{t:Holder} becomes straightforward:

\begin{proof}[Proof of Theorem~\ref{t:Holder}]
By Proposition~\ref{p:Stationary_condition}, uniform H\"older continuity of $(F_n)_{n\in\N}$ follows if we can find two sequences $t_n$ and $s_n$ such that $t_n-t_{n-1},s_n-s_{n-1}\geq 1/4$ and
\begin{equation}\label{e:FinalLD}
    \P\left(\sum_{t_n,s_n\leq 5N}\ind_{\mathrm{Ann}^\prime(x,t_n,y,s_n,N)}<\delta N\right)\leq C\rho^{(2+3\epsilon)5N}.
\end{equation}
Lemma~\ref{l:Decompose} shows that for $t_n,s_n\geq N$ we can replace $\mathrm{Ann}^\prime(x,t_n,y,s_n,N)$ by the intersection of the following events
\begin{equation}\label{e:ConditionList}
\begin{aligned}
\mathrm{Upp}^{(1)}_{\lfloor t_n\rfloor},\mathrm{Low}_{\lfloor t_n\rfloor}^{(1)},\mathrm{Frac}^{(1)}_{\lfloor t_n\rfloor},
\mathrm{Upp}^{(2)}_{\lfloor s_n\rfloor},\mathrm{Low}_{\lfloor s_n\rfloor}^{(2)},\mathrm{Frac}^{(2)}_{\lfloor s_n\rfloor},\\
\mathrm{Centre}(N),\mathrm{Scal}(N),
\mathrm{ShapeRed}^{(1)}(t_n),\mathrm{ShapeRed}^{(2)}(s_n),\mathrm{SizeRed}(t_n,s_n).
\end{aligned}
\end{equation}
Let us choose $\delta>0$ as given in the statement of Proposition~\ref{p:LDRed}. Taking $\gamma>0$ sufficiently small, Proposition~\ref{p:LDApp} states that with probability at least $1-C\rho^{(2+3\epsilon)5N}$ the event $\mathrm{Upp}^{(j)}_m\cap\mathrm{Low}^{(j)}_m\cap \mathrm{Frac}^{(j)}_m$ occurs for at least $(1-\delta)5N$ points from $m\in\{1,2,\dots,5N\}$ and $j\in\{1,2\}$. Moreover the same proposition gives the same lower bound on the probability of $\mathrm{Centre}(N)\cap\mathrm{Scal}(N)$.  Hence by Proposition~\ref{p:LDRed} and the union bound, all of the events in \eqref{e:ConditionList} occur for a sequence of length at least $\delta N$ with probability at least $1-C\rho^{(2+3\epsilon)5N}$. This verifies \eqref{e:FinalLD} and so completes the proof.
\end{proof}

Let us turn to proving Proposition~\ref{p:LDRed}. We note that the measures $\tau^{(j)}$, $\nu^{(j)}$, $\tau^{(j)}_t$, $\nu^{(j)}_t$ are all stationary in the sense that $\tau^{(j)}(u+\cdot)$ has the same distribution as $\tau^{(j)}(\cdot)$ for any given $u\in\R$ (and analogous statements hold for the other measures). This fact follows from the corresponding property of the periodic white noise processes used to construct these measures. From the definitions of the events $\mathrm{ShapeRed}^{(j)}(t)$ and $\mathrm{SizeRed}(t,s)$ we see that every interval to which $\tau^{(1)}$, $\tau^{(1)}_t$, $\nu^{(1)}$ or $\nu^{(1)}_t$ is applied has been translated by $x$ and similarly the intervals for measures with superscript $(2)$ have been translated by $y$. In proving Proposition~\ref{p:LDRed} it is therefore sufficient to consider the case $x=y=0$, which will simplify our notation in what follows.

\subsection{Oscillating random walk algorithm}\label{ss:algorithm}
We now describe the algorithm for selecting $(t_n,s_n)$, which forms the basis of Proposition~\ref{p:LDRed}. We begin with some heuristics. Recall that the event $\mathrm{SizeRed}(t,s)$ occurs if
\begin{equation}\label{e:SizeRed}
    \Bigg\lvert \underbrace{X^{(V)}_{t,s}-X_{N,N}^{(V)}}_{(1)}+\underbrace{X_{N,N}^{(H)}+\log\left(\frac{\tau^{(2)}([0,1]\setminus B_{N})}{\tau^{(1)}([0,1]\setminus B_{N})}\right)}_{(2)}\Bigg\rvert\leq\log(2)
\end{equation}
where
\begin{displaymath}
    X^{(V)}_{t,s}=\gamma V^{(1)}_{\rho^t}(0)-\gamma V^{(2)}_{\rho^s}(0)-(t-s)\left(1+\frac{\gamma^2}{2}\right)\log(1/\rho)
\end{displaymath}
and $X_{t,s}^{(H)}$ is defined analogously for $H$. Viewed as a function of $t$, the process $V^{(j)}_{\rho^t}(0)$ is a Brownian motion with speed $\log(1/\rho)$ and so $X_{t,s}^{(V)}$ can be viewed as a difference of two independent Brownian motions with the same drift (where we allow the time parameters $t$ and $s$ of each Brownian motion to vary independently). In particular
\begin{align}\label{e:Increment}
    X^{(V)}_{t+u,s+v}-X^{(V)}_{t,s}&\sim\mathcal{N}(d(v-u),\sigma^2(u+v))
\end{align}
where
\begin{equation}\label{e:Drift}
    d:=\left(1+\frac{\gamma^2}{2}\right)\log(1/\rho)\quad\text{and}\quad\sigma^2:=\gamma^2\log(1/\rho).
\end{equation}
Therefore increasing one parameter or the other will tend to increase or decrease the value of $(1)$ in \eqref{e:SizeRed} on average.

The idea of the algorithm is to choose these parameter values $t$ and $s$ iteratively, increasing whichever one will bring the left hand side of \eqref{e:SizeRed} closer to the target interval $[-\log(2),\log(2)]$. Moreover, we can view expression $(2)$ of \eqref{e:SizeRed} as an `initial distribution' since for $t,s\geq N$, expression $(1)$ is independent of $(2)$ (as they depend on disjoint regions of the hyperbolic white noise). Therefore the increments will be independent of the left hand side of \eqref{e:SizeRed}.

It will turn out to be convenient to mostly consider integer values of the parameters (this is for consistency with the events defined in earlier sections on integer scales). Then once we find a value of the parameters such that \eqref{e:SizeRed} is not too large (on the order of $d$) we choose fractional values for the next increment to obtain a high probability of landing in the target interval. These fractional parameter values will be $(t_n,s_n)$.

More precisely our algorithm is the following:
\begin{algorithm}\label{Algo}
    Given $N\in\N$ we define the integer sequences $(i_m,j_m)_{m\geq N}$ inductively as follows:
    \begin{itemize}
        \item Set $i_N=j_N=N$
        \item Given $(i_m,j_m)$, define
        \begin{displaymath}
         Y_m=X^{(V)}_{i_m,j_m}-X^{(V)}_{N,N}+X^{(H)}_{N,N}+\log\frac{\tau^{(2)}([0,1]\setminus B_{N})}{\tau^{(1)}([0,1]\setminus B_{N})}
         \end{displaymath}
         \item If $Y_m<-d$ then set $(i_{m+1},j_{m+1})=(i_m,j_m+1)$
        \item If $Y_m>d$ then set $(i_{m+1},j_{m+1})=(i_m+1,j_m)$
        \item Otherwise set $(i_{m+1},j_{m+1})=(i_m+2,j_m+2)$.
    \end{itemize}
See Figure~\ref{fig:Algorithm}. We next define the stopping times $T_0,T_1,\dots$ inductively as $T_0=N-1$ and
\begin{displaymath}
T_{n}=\inf\left\{m>T_{n-1}\;:\;\lvert Y_m\rvert\leq d\right\}
\end{displaymath}
for $n\in\N$. Then for each $n$, let $u$ and $v$ be the minimal values in $[1,1+5/8]\cup\{2\}$ such that
\begin{displaymath}
    Y_{T_n}+d(v-u)=0
\end{displaymath}
and set
\begin{displaymath}
    (t_n,s_n)=(i_{T_n}+u,j_{T_n}+v).
\end{displaymath}
\end{algorithm}

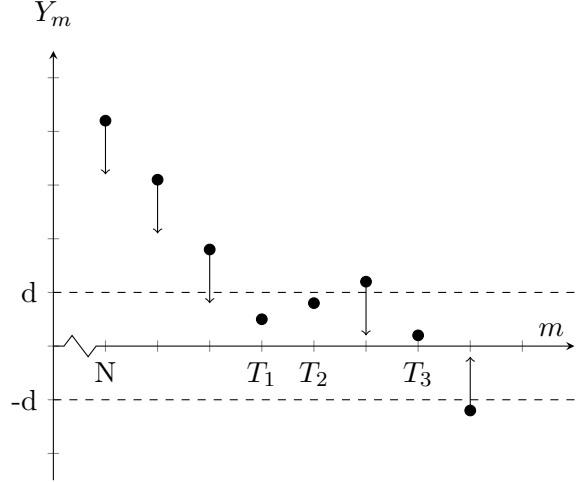
\begin{figure}[ht]
    \centering
    \begin{tikzpicture}
\begin{axis}[
xmin=0, xmax=10,
ymin=-2.5, ymax=5.5,axis x line =center, axis y line = left, xlabel={$m$}, ylabel={$Y_m$}, ylabel style={rotate=-90,at={(current axis.north west)},above=2mm},
ytick={-2,-1,...,5}, yticklabels={,-d,,d,,,},
xtick={1,2,3,4,5,6,7,8,9}, xticklabels={N,,,$T_1$,$T_2$,,$T_3$},
axis x discontinuity=crunch
]
\draw[fill] (1,4.2) circle (2pt);
\draw[->] (1,4.2)--(1,3.2);
\draw[fill] (2,3.1) circle (2pt);
\draw[->] (2,3.1)--(2,2.1);
\draw[fill] (3,1.8) circle (2pt);
\draw[->] (3,1.8)--(3,0.8);
\draw[fill] (4,0.5) circle (2pt);
\draw[fill] (5,0.8) circle (2pt);
\draw[fill] (6,1.2) circle (2pt);
\draw[->] (6,1.2)--(6,0.2);
\draw[fill] (7,0.2) circle (2pt);
\draw[fill] (8,-1.2) circle (2pt);
\draw[->] (8,-1.2)--(8,-0.2);
\draw[dashed] (0,1)--(10,1);
\draw[dashed] (0,-1)--(10,-1);
\end{axis}
\end{tikzpicture}
    \caption{The process $Y_m$ described in Algorithm~\ref{Algo} can be thought of as an oscillating random walk with Gaussian steps: when $Y_m$ is outside the interval $[-d,d]$ it takes Gaussian steps with expectation $\pm d$ towards the origin (indicated by arrows in the figure). When $\lvert Y_m\rvert\leq d$ the next step has an expectation of zero.}
    \label{fig:Algorithm}
\end{figure}

Having now defined our algorithm, we begin to prove Proposition~\ref{p:LDRed}. The first step is to show that (with high probability) the number of stopping times $T_n$ which occur before time $5N$ is of order $N$. This ensures that many values of $t_n$ and $s_n$ are bounded above by $5N$. We prove this in the next subsection. The second step is to show that the events of interest (i.e., $\mathrm{ShapeRed}$ and $\mathrm{SizeRed}$) occur for many such $(t_n,s_n)$. This is done in the final subsection.

\subsection{Large deviation bounds for occupation time}\label{ss:LargeDeviation}
The stopping times $(T_n)_{n\in\N}$ occur whenever the process $(Y_m)_{m\geq N}$ hits $[-d,d]$. Therefore proving that many $T_n$ occur before $5N$ is equivalent to proving bounds on the occupation time of $[-d,d]$ for $(Y_m)_{m\geq N}$. From \eqref{e:Increment} and Algorithm~\ref{Algo}, we see that the distribution of $(Y_m)_{m\geq N}$ can be described as follows: for all $m$, $Y_{m+1}-Y_m$ is independent of $Y_N,\dots,Y_m$ and
\begin{displaymath}
    Y_{m+1}-Y_m\sim\begin{cases}
\mathcal{N}(-d,\sigma^2) &\text{if }Y_m>d\\
\mathcal{N}(d,\sigma^2) &\text{if }Y_m< -d\\
\mathcal{N}(0,4\sigma^2) &\text{if }\lvert Y_m\rvert\leq d.
\end{cases}
\end{displaymath}
The initial distribution of $(Y_m)_{m\geq N}$ is
\begin{displaymath}
    Y_N=X^{(H)}_{N,N}+\log\frac{\tau^{(2)}([0,1]\setminus B_{{N}})}{\tau^{(1)}([0,1]\setminus B_{N})}.
\end{displaymath}
This process is clearly attracted to the interval $[-d,d]$ and so we should expect it to spend a constant proportion of its time in this interval in the long run. One obstacle to proving a large deviation estimate as $N\to\infty$ for this occupation time is the fact that the initial distribution $Y_N$ increases in magnitude with $N$. Fortunately this can be controlled fairly easily using tail bounds for $\tau^{(j)}$ and the Gaussian distribution:

\begin{lemma}[Initial distribution]\label{l:Initial}
For all $\rho>0$ sufficiently small and $\gamma<\gamma_0(\rho)$
\begin{displaymath}
\P\left(\left\lvert Y_N\right\rvert>\frac{dN}{8}\right)\leq C\rho^{(2+3\epsilon)5N}
\end{displaymath}
for some $C>0$ and all $N$ sufficiently large.
\end{lemma}
\begin{proof}
Since $X_{N,N}^{(H)}$ is a centred Gaussian with variance bounded above by $C \gamma^2 N\log(1/\rho)$ for an absolute constant $C$, using the standard Gaussian tail inequality we have
\begin{displaymath}
\P\Big(\big\lvert X_{N,N}^{(H)}\big\rvert>\frac{dN}{16}\Big)\leq \exp\big(-d^2N/(2^9C\gamma^2\log(1/\rho))\big).
\end{displaymath}
Choosing $\gamma$ sufficiently small, depending on $\rho$, we can ensure that the latter is bounded above by $C\rho^{(2+3\epsilon)5N}$.

By the union bound
\begin{multline*}
\P\left(\left\lvert \log\frac{\tau^{(2)}([0,1]\setminus B_{N})}{\tau^{(1)}([0,1]\setminus B_{N})}\right\rvert>(d/16)N\right)\\
\begin{aligned}
\leq& 2\P\Big(\tau^{(j)}([0,1]\setminus B_N)>e^{(d/32)N}\Big)+2\P\Big(\tau^{(j)}([0,1]\setminus B_N)^{-1}>e^{(d/32)N}\Big)\\
\leq& 2\P\Big(\tau^{(j)}([0,1])>e^{(d/32)N}\Big)+2\P\Big(\tau^{(j)}([0,1])^{-1}>2^{-1}e^{(d/32)N}\Big)\\
&+2\P\Big(\tau^{(j)}(B_N)/\tau^{(j)}([0,1])>2^{-1}\Big).
\end{aligned}
\end{multline*}
By Proposition~\ref{p:LDApp} (specifically the statement for the event $\mathrm{Scal}(N)$) the final term above is at most $C\rho^{(2+3\epsilon)5N}$. Recall from Lemma~\ref{l:MeasureBasic} that $\E[\tau^{(j)}([0,1])^p]<\infty$ for all $p\in(-\infty,2/\gamma^2)$. We choose $p>(2+3\epsilon)160$ which, by \eqref{e:Drift}, ensures that $e^{-pd/32}<\rho^{(2+3\epsilon)5}$. Taking $\gamma$ small enough that $p<2/\gamma^2$, by the Markov inequality
\begin{align*}
\P(\tau^{(j)}([0,1])>e^{(d/32)N})\leq \E[\tau^{(j)}([0,1])^p]e^{-(pd/32)N}\leq C_p\rho^{(2+3\epsilon)5N}
\end{align*}
as required. Similarly
\begin{align*}
\P(\tau^{(j)}([0,1])^{-1}>2^{-1}e^{(d/32)N})\leq \E[\tau^{(j)}([0,1])^{-p}]2^pe^{-(pd/32)N}\leq C_p^\prime\rho^{(2+3\epsilon)5N}
\end{align*}
completing the proof of the lemma.
\end{proof}

Our next objective is to control how long it takes for $Y_m$ to travel from the (typically large) initial value $Y_N$ to the interval $[-d,d]$. Intuitively this is just the time that a negatively biased random walk started at a large positive value takes to first cross the origin. However since our steps are normally distributed it is possible for the walk to `overshoot' the interval $[-d,d]$ and then become a random walk biased in the opposite direction. We control the overshoot distribution with the following lemma:

\begin{lemma}[Overshoot distribution]\label{l:Jump}
Let $S=\inf\{m\geq N: Y_m\leq d\}$ then for any $a\leq-d$
\begin{displaymath}
\P(Y_{S}<a\;|\;Y_N>d)\leq Ce^{-\frac{a^2}{2\sigma^2}}
\end{displaymath}
where $C>0$ is an absolute constant. Moreover this is true even if we allow $Y_N$ to have an arbitrary distribution (which is independent of $(Y_m-Y_N)_{m>N}$ and not supported on $(-\infty,d]$).
\end{lemma}
\begin{proof}
Throughout this proof we condition on the event $\{Y_N>d\}$ and so we omit this from our notation. By partitioning over the different possible values of $S$
\begin{align*}
    \P(Y_S< a)&=\sum_{n=1}^\infty\P(Y_n< a\;|\;S=n)\P(S=n)\leq\sup_{n\in\N}\P(Y_n< a\;|\;S=n).
\end{align*}
By definition of $S$, for a given $n$
\begin{align*}
    \P(Y_n<a\;|\;S=n)&=\P(Y_n<a\;|\;Y_n\leq d< Y_1,\dots,Y_{n-1})\\
    &=\frac{\P(Y_n<a,Y_{n-1}>d\;|\;Y_1,\dots,Y_{n-2}>d)}{\P(Y_n\leq d,Y_{n-1}>d\;|\;Y_1,\dots,Y_{n-2}>d)}.
\end{align*}
If we write $\Tilde{p}$ for the law of $Y_{n-1}$ conditional on the event $Y_1,\dots,Y_{n-2}>d$, then by independence of increments we can rewrite the above expression as
\begin{align*}
    \frac{\int_d^\infty\P(Y_{n}<a\;|\;Y_{n-1}=u)\;d\Tilde{p}(u)}{\int_d^\infty\P(Y_{n}<d\;|\;Y_{n-1}=u)\;d\Tilde{p}(u)}\leq \sup_{u>d}\frac{\P(Y_{n}<a\;|\;Y_{n-1}=u)}{\P(Y_{n}<d\;|\;Y_{n-1}=u)}
\end{align*}
where this inequality follows from simply factoring out the supremum. The distribution of $Y_n$ conditional on $Y_{n-1}=u>d$ is Gaussian with mean $u-d$ and variance $\sigma^2$ (independent of $n$). Therefore we conclude that
\begin{displaymath}
\P(Y_S<a)\leq\sup_{u\geq d}\P\left(Z<\frac{a+d-u}{\sigma}\;\middle|\:Z<\frac{2d-u}{\sigma}\right)=\sup_{u\geq 0}\P\left(Z<\frac{a-u}{\sigma}\;\middle|\:Z<\frac{d-u}{\sigma}\right)
\end{displaymath}
where $Z\sim\mathcal{N}(0,1)$. We can bound this expression using the standard Gaussian tail inequalities
\begin{displaymath}
    \frac{1}{\sqrt{2\pi}}\Big(\frac{1}{x}-\frac{1}{x^3}\Big)e^{-x^2/2}\leq\P(\mathcal{N}(0,1)>x)\leq\min\Big\{1,\frac{1}{\sqrt{2\pi}x}\Big\} e^{-x^2/2}
\end{displaymath}
valid for $x>0$. Applying this with $u\in[0,d+2\sigma]$ we have
\begin{align*}
\P\left(Z<\frac{a-u}{\sigma}\;\middle|\:Z<\frac{d-u}{\sigma}\right)&\leq \P\left(Z<\frac{a-u}{\sigma}\right)/\P\left(Z<-2\right)\leq Ce^{-\frac{(a-u)^2}{2\sigma^2}}.
\end{align*}
Since $a-u\leq a<0$, the latter is at most $Ce^{-a^2/(2\sigma^2)}$. Similarly for $u>d+2\sigma$
\begin{align*}
\P\left(Z<\frac{a-u}{\sigma}\;\middle|\:Z<\frac{d-u}{\sigma}\right)&\leq\frac{\frac{\sigma}{u-a}e^{-\frac{(u-a)^2}{2\sigma^2}}}{\frac{\sigma}{u-d}\left(1-\frac{\sigma^2}{(u-d)^2}\right)e^{-\frac{(u-d)^2}{2\sigma^2}}}= \frac{u-d}{u-a}\frac{1}{1-\frac{\sigma^2}{(u-d)^2}}e^{-\frac{a^2-2au+2du-d^2}{2\sigma^2}}.
\end{align*}
The two fractional terms are bounded above by $1$ and $4/3$ respectively for all $u>d+2\sigma$ whilst the exponential term is bounded above by $e^{-a^2/(2\sigma^2)}$, completing the proof of the lemma.
\end{proof}

We can now give the desired tail estimate on the first stopping time $T_1$, conditional on the initial value $Y_N$. (In fact by the Markov property, we will be able to apply the following result to each of the stopping times $T_n$. So this is really the key estimate of the subsection.)

\begin{lemma}[Burn-in period]\label{l:Burn-in}
Given $\kappa\in(0,1)$, for $\rho>0$ sufficiently small, all $\gamma\in(0,\sqrt{2})$ and all $j\in\N$
\begin{displaymath}
\P\left(T_1-N>j\;|\;Y_N=u\right)<\kappa^{2j-8\lvert u\rvert/d}.
\end{displaymath}
\end{lemma}

\begin{proof}
We will assume that $u>d$; the case $u<-d$ follows by symmetry of the oscillating random walk about zero and the case $\lvert u\rvert\leq d$ is trivial since this implies $T_1=N$. We inductively define the stopping times $(S_k)_{k\geq 0}$ by $S_0:=N$,
\begin{displaymath}
S_{2k+1}=\inf\{m>S_{2k}\;|\;{Y}_m\leq d\}\wedge T_1\quad\text{and}\quad S_{2k}=\inf\{m>S_{2k-1}\;|\;{Y}_m\geq-d\}\wedge T_1
\end{displaymath}
where $n_1\wedge n_2:=\min\{n_1,n_2\}$ and as usual the infimum of an empty set is defined to be $\infty$. So informally speaking, $S_k$ is the $k$-th time that the oscillating random walk crosses over the interval $[-d,d]$ without hitting it (see Figure~\ref{fig:CrossingTimes}). Let $\eta$ be the number of such crossings, i.e., $\eta$ is the smallest integer such that
\begin{equation}\label{e:EtaDef}
    T_1-N=\sum_{k=0}^\eta S_{k+1}-S_k.
\end{equation}

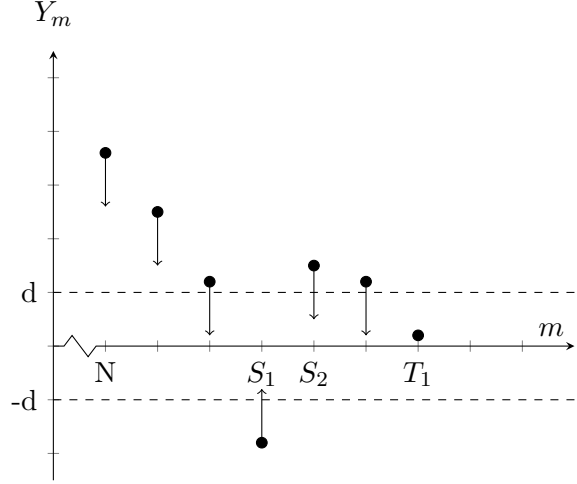
\begin{figure}[ht]
    \centering
    \begin{tikzpicture}
\begin{axis}[
xmin=0, xmax=10,
ymin=-2.5, ymax=5.5,axis x line =center, axis y line = left, xlabel={$m$}, ylabel={$Y_m$}, ylabel style={rotate=-90,at={(current axis.north west)},above=2mm},
ytick={-2,-1,...,5}, yticklabels={,-d,,d,,,},
xtick={1,2,3,4,5,6,7,8,9}, xticklabels={N,,,$S_1$,$S_2$,,$T_1$},
axis x discontinuity=crunch
]
\draw[fill] (1,3.6) circle (2pt);
\draw[->] (1,3.6)--(1,2.6);
\draw[fill] (2,2.5) circle (2pt);
\draw[->] (2,2.5)--(2,1.5);
\draw[fill] (3,1.2) circle (2pt);
\draw[->] (3,1.2)--(3,.2);
\draw[fill] (4,-1.8) circle (2pt);
\draw[->] (4,-1.8)--(4,-.8);
\draw[fill] (5,1.5) circle (2pt);
\draw[->] (5,1.5)--(5,0.5);
\draw[fill] (6,1.2) circle (2pt);
\draw[->] (6,1.2)--(6,0.2);
\draw[fill] (7,0.2) circle (2pt);
\draw[dashed] (0,1)--(10,1);
\draw[dashed] (0,-1)--(10,-1);
\end{axis}
\end{tikzpicture}
    \caption{An illustration of the stopping times $(S_k)_{k\geq 0}$. Arrows indicate the expected increments of the process $Y_m$. Roughly speaking, $S_{k+1}$ is the first time after $S_k$ that $Y_m$ `crosses over' the interval $[-d,d]$ before hitting it. For the illustrated process $\eta=2$ because $T_1=S_3>S_2$.}
    \label{fig:CrossingTimes}
\end{figure}

Our first goal is to prove a tail bound for $\eta$. For $k\in\N$,
\begin{displaymath}
\P(\eta\geq k\;|\;\eta\geq k-1)=\P\left(\lvert Y_{S_k}\rvert>d\;\middle|\;\lvert Y_{S_{k-1}}\rvert>d\right).
\end{displaymath}
The event $\left\{\lvert Y_{S_{k-1}}\rvert>d\right\}$ is measurable with respect to the stopped filtration $\mathcal{F}_{S_{k-1}}$ (where $\mathcal{F}_n=\sigma(Y_1,\dots,Y_n)$ is the natural filtration) and conditioning on this event induces some distribution on $Y_{S_{k-1}}$. We call the distribution $\Lambda_{k-1}$ and note that $\Lambda_{k-1}$ is supported on $\R\setminus(-d,d)$. By the strong Markov property then
\begin{displaymath}
\P\left(\lvert Y_{S_k}\rvert>d\;\middle|\;\lvert Y_{ S_{k-1}}\rvert>d\right)=\P\left(\lvert Y_{S_1}\rvert>d\;|\;Y_N\sim\Lambda_{k-1}\right)
\end{displaymath}
where, in a slight abuse of notation, the latter expression means we consider a random walk with the same distribution as $Y_m$ except that the first variable $Y_N$ has the distribution $\Lambda_{k-1}$. By symmetry, the latter expression is equal to
\begin{displaymath}
\P(Y_{S_1}<-d\;|\;Y_N\sim\Lambda^\prime_{k-1})
\end{displaymath}
for some $\Lambda^\prime_{k-1}$ supported on $[d,\infty)$. By Lemma~\ref{l:Jump} this probability is at most $Ce^{-d^2/(2\sigma^2)}$ for an absolute constant $C$. Therefore taking $d/\sigma$ sufficiently large, we may ensure that for all $k\in\N$, $\P(\eta\geq k\;|\;\eta\geq k-1)$ is smaller than any given constant. Recalling \eqref{e:Drift} we see that this can be done by taking $\rho$ sufficiently small. Hence, by iterated conditioning, we can choose $\rho$ small enough that for all $j\in\N$
\begin{equation}\label{e:BurnIn0}
\P(\eta>j/8)\leq \frac{1}{3}\kappa^{2j}.
\end{equation}

With this bound proven, it is now enough for us to control the probability that
\begin{displaymath}
\sum_{k=0}^{j/8}S_{k+1}-S_{k}>j.
\end{displaymath}
We consider the first term of this sum separately: if $S_1-S_0>j/2$ then the first $\lfloor j/2\rfloor$ increments of the random walk must have negative drift and the sum of these increments and the initial value must be greater than $d$. Therefore
\begin{equation}\label{e:BurnIn1}
    \begin{aligned}
     \P(S_1-S_0>j/2\;|\;Y_N=u)&\leq\P\Big(u+\mathcal{N}\big(-d\lfloor j/2\rfloor,\sigma^2 \lfloor j/2\rfloor\big)>d\Big)\\
     &=\P\left(\mathcal{N}(0,1)>\frac{d(\lfloor j/2\rfloor+1)-u}{\sigma\sqrt{\lfloor j/2\rfloor}}\right).   
    \end{aligned}
\end{equation}
We may assume that $j\geq 4u/d$ since otherwise the statement of the lemma becomes trivial. Hence the above expression is bounded by
\begin{equation}\label{e:BurnIn2}
    \P\left(\mathcal{N}(0,1)>\frac{\sqrt{2j}d}{4\sigma}\right)\leq e^{-\frac{jd^2}{16\sigma^2}}
\end{equation}
using the standard Gaussian tail inequality. Taking $d/\sigma$ sufficiently large, we can ensure that this is at most $\kappa^{2j}/3$.

Now moving to the subsequent increments, for any $i\in\N$ by the union bound
\begin{equation}\label{e:Burn-in3}
\begin{aligned}
\P(S_{k+1}-S_k>i\;|\;S_1,\dots,S_k)\leq&\P\Big(\{S_{k+1}-S_k>i\}\cap\big\{\big\lvert Y_{S_k}\big\rvert\leq id/2\big\}\;\big|\;S_1,\dots,S_k\Big)\\
&+\P(\lvert Y_{S_k}\rvert>id/2\;|\;S_1,\dots,S_k)
\end{aligned}
\end{equation}
and we consider these two terms in turn. For any increasing sequence $n_1<\dots<n_k$
\begin{multline*}
\P(\{S_{k+1}-S_k>i\}\cap\{\lvert Y_{n_k}\rvert\leq id/2\}\;|\;S_1=n_1,\dots,S_k=n_k)\\
\begin{aligned}
&\leq \P(\{\lvert Y_{n_k}\rvert-id+Z_i>d\}\cap\{\lvert Y_{n_k}\rvert\leq id/2\}\;|\;S_1=n_1,\dots,S_k=n_k)\\
&\leq\P(Z_i>id/2\;|\;S_1=n_1,\dots,S_k=n_k)
\end{aligned}
\end{multline*}
where $Z_i\sim\mathcal{N}(0,i\sigma^2)$ is independent of $S_1,\dots,S_k$ and $Y_{n_k}$, which allows us to drop the conditioning and bound the expression by $e^{-id^2/(8\sigma^2)}$. Since this bound is uniform over the choice of $n_1,\dots,n_k$ it holds for the first term on the right hand side of \eqref{e:Burn-in3}.

Turning to the second term in \eqref{e:Burn-in3}, let us assume that $k$ is odd (a near identical proof holds if $k$ is even). We again fix $n_1<\dots<n_k$ and define
\begin{displaymath}
A=\{S_1=n_1,\dots,S_{k-1}=n_{k-1}\}\cap\{Y_{n_{k-1}},\dots,Y_{n_k-1}>d\}
\end{displaymath}
so that $A\cap\{Y_{n_k}\leq d\}=\{S_1=n_1,\dots,S_k=n_k\}$ (this is where we use the fact that $k$ is odd). Then
\begin{align*}
\P\big(\lvert Y_{S_k}\rvert>id/2\;\big|\;S_1=n_1,\dots,S_k=n_k\big)&=\P\big(Y_{S_k}<-id/2\;\big|\;A\cap\{Y_{n_k}\leq d\}\big)\\
&=\frac{\P(Y_{n_k}<-id/2\;|A)}{\P(Y_{n_k}\leq d\;|A)}
\end{align*}
By conditioning over the different possible values of $Y_{n_k-1}$ and using the Markov property, this is at most
\begin{displaymath}
    \sup_{u\geq d}\P(Y_{n_k}-Y_{n_k-1}<-id/2-u\;|\;Y_{n_k}-Y_{n_k-1}<d-u)=\sup_{u\geq 0}\P(Z<-id/2-u\;|\;Z<d-u)
\end{displaymath}
where $Z\sim\mathcal{N}(0,\sigma^2)$. In the proof of Lemma~\ref{l:Jump} it was shown that this expression is bounded above, for $i\geq 2$, by $Ce^{-i^2d^2/(8\sigma^2)}$. We conclude that the right hand side of \eqref{e:Burn-in3} is bounded above by $C^\prime e^{-id^2/(8\sigma^2)}$ for all $i\geq 2$ and an absolute constant $C^\prime>0$.

Let $\theta>1$, then using a standard identity for $\E[\theta^X]$ when $X$ is some random variable taking values in $\{0\}\cup\N$, we have
\begin{equation*}
\begin{aligned}
\E\left[\theta^{S_{k+1}-S_k}\;\middle|\;S_1,\dots,S_k\right]&=1+(\theta-1)\sum_{i=0}^\infty\theta^i\P\left(S_{k+1}-S_k>i\;|\;S_1,\dots,S_k\right)\\
&\leq 1+(\theta-1)\Big(1+\theta+\sum_{i=2}^\infty\theta^iC^\prime e^{-id^2/(8\sigma^2)}\Big)
\end{aligned}
\end{equation*}
Thus for each $\theta>0$ we can choose $d/\sigma$ sufficiently large to ensure that this expression is less than, say, $2\theta^2$.

We now apply this to the sum of crossing times. By the Markov inequality and standard properties of conditional expectation, for any $\theta>0$
\begin{align*}
\P\left(\sum_{k=1}^{j/8}S_{k+1}-S_{k}>j/2\right)&\leq \theta^{-j/2}\E\left[\theta^{\sum_{k=1}^{j/8}S_{k+1}-S_{k}}\right]\\
&\leq \theta^{-j/2}\E\left[\E[\theta^{S_{j/8+1}-S_{j/8}}\;|\;S_1,\dots,S_{j/8}]\theta^{\sum_{k=1}^{j/8-1}S_{k+1}-S_{k}}\right]\\
&\leq \theta^{-j/2}2\theta^2\E\left[\theta^{\sum_{k=1}^{j/8-1}S_{k+1}-S_{k}}\right]\leq\dots\leq 2^{j/8}\theta^{-j/4}. 
\end{align*}
Choosing $\theta$ sufficiently large, which in turn requires $\rho$ sufficiently small, the above expression is at most $\kappa^{2j}/3$. Combining this with \eqref{e:EtaDef}, \eqref{e:BurnIn0} and \eqref{e:BurnIn1}-\eqref{e:BurnIn2} completes the proof of the lemma.
\end{proof}

Combining this tail bound on  $T_1$ with the Markov property allows us to prove a large deviation bound for the occupation time of the oscillating random walk.

\begin{lemma}[Occupation time]\label{l:Occupationtime}
Given $\rho>0$, there exists $\delta>0$ such that for all $\gamma<\gamma_0(\rho)$
\begin{displaymath}
\P\left(T_{14\delta N}>3N\right)<C\rho^{(2+3\epsilon)5N}
\end{displaymath}
for some $C>0$ and all $N\in\N$.
\end{lemma}
\begin{proof}
Defining $\delta^\prime=14\delta$, let us write
\begin{displaymath}
T_{\delta^\prime N}=T_0+\sum_{k=1}^{\delta^\prime N}T_k-T_{k-1}.
\end{displaymath}
Our aim is to apply Lemma~\ref{l:Burn-in} to each term in this sum (using the Markov property). First note that by combining Lemma~\ref{l:Initial} and Lemma~\ref{l:Burn-in} with $\kappa=\rho^{(2+3\epsilon)5}$, the probability that $T_1-T_0>N$ is less than $C\rho^{(2+3\epsilon)5N}$.

Next by the Markov property, for any $k,j\in\N$ we have
\begin{align*}
    \P(T_k-T_{k-1}>j\;|\;T_1,\dots,T_{k-1})&\leq\sup_{u\in[-d,d]}\P(T_k-T_{k-1}>j\;|\;Y_{T_{k-1}}=u)\\
    &=\sup_{u\in[-d,d]}\int_{\R\setminus[-d,d]}\P(T_k-T_{k-1}>j\;|\;Y_{T_{k-1}+1}=v)p_u(v)\;dv
\end{align*}
where $p_u(v)$ denotes the density of a $\mathcal{N}(u,4\sigma^2)$ random variable and we have removed $[-d,d]$ from the domain of integration because $T_k-T_{k-1}=1$ if $Y_{T_{k-1}+1}$ falls in this interval. Since $(Y_m)$ is a time-homogeneous Markov chain, we can apply Lemma~\ref{l:Burn-in} to bound the latter expression by
\begin{displaymath}
    \sup_{u\in[-d,d]}\int_{\R\setminus[-d,d]}\kappa^{2(j-1)-8\lvert v\rvert/d}p_u(v)\;dv.
\end{displaymath}
As $\sigma$ decreases to zero, the distribution $p_u(\cdot)$ will concentrate around $u\in[-d,d]$. In particular for $\sigma>0$ sufficiently small, we can ensure that the above integral is bounded by, say, $\kappa^{2(j-1)-8}$ where $\kappa>0$ will be specified below.

Now for $\theta>1$, using the same identity introduced in the proof of Lemma~\ref{l:Burn-in}
\begin{align*}
    \E\left[\theta^{T_k-T_{k-1}}\;\middle|\;T_1,\dots,T_{k-1}\right]&=1+(\theta-1)\sum_{j=0}^\infty\theta^j\P\left(T_k-T_{k-1}>j\;|\;T_1,\dots,T_{k-1}\right)\\
    &\leq 1+(\theta-1)\left(1+\sum_{j=1}^\infty\theta^j\kappa^{2j-10}\right)= \theta+(\theta-1)\kappa^{-10}\frac{\theta\kappa^2}{1-\theta\kappa^2}
\end{align*}
where the final inequality holds for $\theta\kappa^2<1$. If we take $\kappa^2=\theta^{-1}/2$ then this expression is bounded by $2^6\theta^6$. By iterating this inequality (analogously to the argument in the proof of Lemma~\ref{l:Burn-in}) we find that
\begin{align*}
\P(T_{\delta^\prime N}-T_1>N)\leq \theta^{-N}\E\left[\theta^{\sum_{n=2}^{\delta^\prime N}T_n-T_{n-1}}\right]\leq \theta^{-N}(2\theta)^{6(\delta^\prime N-1)}.
\end{align*}
Choosing $\delta^\prime<1/6$ and $\theta$ sufficiently large (the latter of which is possible by choosing $\sigma$ sufficiently small) we find that this expression is at most $\rho^{(2+3\epsilon)5N}$. Combined with the analogous bound for $T_1-T_0=T_1-(N-1)$ this proves the lemma.
\end{proof}

\subsection{Reduced size and shape events}\label{ss:Reduced}

\begin{lemma}\label{l:LDSizeRed}
    For all $\kappa>0$, all $\rho>0$ sufficiently small, all $\gamma<\gamma_0(\rho,\kappa)$ and all $M\in\N$ the sequence $(\ind_{\mathrm{SizeRed}(t_n,s_n)})_{n=1}^M$ stochastically dominates a sequence of $M$ i.i.d.\ Bernoulli random variables with parameter $1-\kappa$.
\end{lemma}
\begin{proof}
We write $\chi_n=\ind_{\mathrm{SizeRed}(t_n,s_n)}$ and define
\begin{displaymath}
    A=\{\chi_1=a_1,\dots,\chi_{n-1}=a_{n-1}\}
\end{displaymath}
where $a_1,\dots,a_{n-1}\in\{0,1\}$. From Algorithm~\ref{Algo}, we see that conditional on $Y_{T_{n}}$, $\chi_{n}$ is independent of $\chi_1,\dots,\chi_{n-1}$. Therefore
\begin{equation}\label{e:SizeRedCondition}
    \P(\chi_{n}=1\;|\;A)=\int_{-d}^d\P(\chi_{n}=1\;|\;Y_{T_{n}}=u)dp(u)
\end{equation}
where $p$ denotes the law of $Y_{T_{n}}$ conditional on $A$. We wish to show that this integrand is bounded below by $1-\kappa$ (for sufficiently small $\gamma$). Recall from the definitions in Section~\ref{ss:algorithm} that
\begin{displaymath}
\mathrm{SizeRed}(t_n,s_n)=\left\{\left\lvert Y_{T_n}+X_{t_n,s_n}^{(V)}-X^{(V)}_{i_{T_n},j_{T_n}}\right\rvert\leq\log(2)\right\}
\end{displaymath}
and that, conditional on $Y_{T_n}=u$, $X_{t_n,s_n}^{(V)}-X^{(V)}_{i_{T_n},j_{T_n}}$ is normally distributed with mean $-u$ and variance bounded above by $4\sigma^2$. Therefore
\begin{displaymath}
\P(\chi_n=1\;|\;Y_{T_n}=u)\geq\P(2\sigma \lvert Z\rvert\leq\log(2))
\end{displaymath}
where $Z\sim\mathcal{N}(0,1)$. Clearly we can ensure that this probability is greater than $1-\kappa$ by choosing $\sigma$ sufficiently small. This is equivalent to choosing $\gamma>0$ small depending on $\rho,\kappa>0$. Hence by \eqref{e:SizeRedCondition} we see that the probability of $\{\chi_n=1\}$ is at least $1-\kappa$ independent of conditioning on $\chi_1,\dots,\chi_{n-1}$. This is true for any $n$, so by iteration we see that $\chi_1,\dots,\chi_M$ stochastically dominates an i.i.d.\ sequence of Bernoulli variables with parameter $1-\kappa$.
\end{proof}

\begin{lemma}\label{l:LDShapeRed}
    The statement of Lemma~\ref{l:LDSizeRed} holds if we replace $\mathrm{SizeRed}(t_n,s_n)$ by either $\mathrm{ShapeRed}^{(1)}(t_n)$ or $\mathrm{ShapeRed}^{(2)}(s_n)$ and take $t_n,s_n$ greater than some $N_0$ depending on $\rho,\kappa$.
\end{lemma}

\begin{proof}
We give only the argument for $\mathrm{ShapeRed}^{(1)}(t_n)$ since the argument for $\mathrm{ShapeRed}^{(2)}(s_n)$ is entirely symmetric. Let $\chi_t=\ind_{\mathrm{ShapeRed}^{(1)}(t)}$ and for $m\geq N$ define
    \begin{displaymath}
        Z_m=\begin{cases}
            \chi_{t_n}&\text{if }m-1=T_n\text{ for some }n\geq 0\\
            0 &\text{otherwise.}
        \end{cases}
    \end{displaymath}
    We first show that $(Y_m,Z_m)_{m\geq N}$ is a Markov chain. In order to construct an explicit filtration, we define a different version of the process as follows: let $W_1,W_2,\dots$ be an independent sequence of hyperbolic white noise processes (as defined in Section~\ref{s:Intro}) and let $\mathcal{F}_m=\sigma(W_1,\dots,W_{2m})$. For each $m\geq N$ we let $(Y_m-Y_{m-1},Z_m)$ be defined as in Algorithm~\ref{Algo} except that $V^{(1)}$, $\nu^{(1)}$, $\tau^{(1)}$ etc are sampled using $W_{2m-1}$ and $V^{(2)}$, $\nu^{(2)}$ etc are sampled using $W_{2m}$. This does not change the distribution of $(Y_m-Y_{m-1},Z_m)_{m\geq N}$ since each element of this sequence is determined by the restriction of the white noise to disjoint domains; the original white noise $W$ is independent on these domains and so it does not matter if we instead sample from $W_{2m-1}$ and $W_{2m}$. To be more precise, $Y_m-Y_{m-1}$ is determined by the white noise for $\nu^{(1)}$ restricted to $\mathcal{V}\cap\{(x,y)\;|\;\rho^{i_m}\leq y\leq\rho^{i_{m-1}}\}$ and the white noise for $\nu^{(2)}$ restricted to $\mathcal{V}\cap\{(x,y)\;|\;\rho^{j_m}\leq y\leq\rho^{j_{m-1}}\}$. Similarly $Z_m$ is determined by the white noise for $\nu^{(1)}$ restricted to
    \begin{displaymath}
        \bigcup_{\rho^{t_n+7/8}\leq\lvert a\rvert\leq 2\rho^{t_n}}(\mathcal{V}+a)\cap\big\{(x,y)\;|\;y\leq\rho^{i_{m-1}}\big\}.
    \end{displaymath}
    These regions are disjoint for distinct values of $m$ (Figures~\ref{Fig:WhiteNoiseDecomp3} and~\ref{Fig:WhiteNoiseDecomp4} show the corresponding regions with $\mathcal{H}$ instead of $\mathcal{V}$).

    For each $n\in\N$ we note that $T_n$ is a stopping time with respect to $(\mathcal{F}_k)$ and that the events $\mathrm{ShapeRed}^{(1)}(t_1),\dots,\mathrm{ShapeRed}^{(1)}(t_{n-1})$ are measurable with respect to the stopped $\sigma$-algebra $\mathcal{F}_{T_n}$. Therefore to prove the result, it is enough to show that for $n$ sufficiently large
    \begin{displaymath}
        \E\big[\chi_{t_n}\;|\;\mathcal{F}_{T_n}\big]\geq 1-\kappa
    \end{displaymath}
    almost surely. By definition of $\mathrm{ShapeRed}$ we see that
    \begin{displaymath}
        \chi_{t_n}=G(t_n,W_{2T_n-1})
    \end{displaymath}
    for some measurable function $G$. Since $W_{2T_n-1}$ is independent of $\mathcal{F}_{T_n}$ and $t_n$ is measurable with respect to this $\sigma$-algebra, we see that
    \begin{displaymath}
        \E\big[\chi_{t_n}\;|\;\mathcal{F}_{T_n}\big]=\tilde{G}(t_n)
    \end{displaymath}
    where $\tilde{G}(u):=\E[G(u,W_1)]=\P(\mathrm{ShapeRed}^{(1)}(u))$ for any constant $u\geq 1$. Therefore the lemma follows if we can show that the latter probability is at least $1-\kappa$ for all $u\geq N_0(\rho,\kappa)$. This is precisely the statement of the next lemma.
\end{proof}

\begin{lemma}
    Given $\kappa>0$, there exists $\gamma_0>0$ such that for all $\gamma\leq\gamma_0$ and all $t$ sufficiently large
    \begin{displaymath}
        \P(\mathrm{ShapeRed}^{(1)}(t))\geq 1-\kappa.
    \end{displaymath}
\end{lemma}
\begin{proof}
    Throughout this proof, we denote $\nu_t=\nu_t^{(1)}$ and $\tau_t=\tau_t^{(1)}$ since the superscripts play no role in our argument. For any Borel set $I\subset[0,1]$ and any $t\geq 1$, it is clear from the definition that
    \begin{displaymath}
        \frac{\tau_t(I)}{\nu_t(I)}=\lim_{\delta\downarrow 0}\frac{\int_I\exp\big(\gamma (H_\delta(a)-H_{\rho^t}(a))-(\gamma^2/2)\Var[H_\delta(0)-H_{\rho^t}(0)]\big)\;da}{\int_I\exp\big(\gamma (V_\delta(a)-V_{\rho^t}(a))-(\gamma^2/2)\Var[V_\delta(0)-V_{\rho^t}(0)]\big)\;da}
    \end{displaymath}
    is bounded above and below by $\exp(\pm G_{t,\gamma})$ where
    \begin{displaymath}
        G_{t,\gamma}=\sup_{\delta\in(0,\rho^t]}\sup_{a\in[0,1]}\left\lvert\gamma U_\delta^{\rho^t}(a)-\frac{\gamma^2}{2}\Var\big[U_\delta^{\rho^t}(a)\big]\right\rvert
    \end{displaymath}
    and we recall the definition of $U$ from Lemma~\ref{l:MeasureComp}. Applying this lemma, we see that as $t\to\infty$, $G_{t,\gamma}$ converges to zero in probability uniformly over $\gamma$. Therefore in each of the events defining $\mathrm{ShapeRed}(t)$ we can replace $\tau_t$ by $\nu_t$ provided we shrink the intervals by some constant. More precisely it is enough to prove that for some $\delta>0$ and all $t\geq N_0(\rho,\kappa)$, the probability of 
    \begin{equation}\label{e:ShapeRed2}
    \begin{aligned}
    &\left\{1+\delta\leq\nu_t(B_t\setminus B_{\lfloor t\rfloor+7/8})\rho^{-t}\leq 2^2-\delta \right\}\cap\bigcap_{\ell=1}^8\left\{\frac{\nu_t(J_{t,\ell}\setminus B_{\lfloor t\rfloor+7/8})}{\nu_t(B_t\setminus B_{\lfloor t\rfloor+7/8})}\geq 2^{-4}+\delta\right\}\\
    &\qquad\qquad\cap\left\{\frac{\nu_t(2B_{t+1/4}\setminus B_{\lfloor t\rfloor+7/8})}{\nu_t(B_t\setminus B_{\lfloor t\rfloor+7/8})}\leq 2^{-30}-\delta\right\}\cap
    \left\{\frac{\nu_t(2B_t\setminus B_{\lfloor t\rfloor+7/8})}{\nu_t(B_t\setminus B_{\lfloor t\rfloor+7/8})}\leq 2^2-\delta\right\}\\
    &\quad\qquad\qquad\qquad\qquad\qquad\cap\left\{\sum_{m=0}^\infty\sum_{\ell\in S(m,\rho^{t+1/4},\rho^{t})}\frac{\nu_t(I_{m,\ell,t})^2}{\nu_t(B_t\setminus B_{\lfloor t\rfloor+7/8})^2}\leq 2^5-\delta\right\}
    \end{aligned}
    \end{equation}
    is at least $1-\kappa/2$.

    Next we derive a scaling relation for the measures $\nu_t$ which shows that the probability of the previous event is roughly constant over $t$. Denoting $V^r_\delta(a):=V_\delta(a)-V_r(a)$ for $0<\delta\leq r\leq 1/2$ and $a,b\in[-1/4,1/4]$ we have
    \begin{displaymath}
        \mathrm{Cov}\big[V^r_\delta(a),V^r_\delta(b)\big]=
        \begin{cases}
        \log(r/\delta)-\lvert b-a\rvert(1/\delta-1/r) &\text{if }\lvert b-a\rvert\leq\delta\\
        \log(r/\lvert b-a\rvert)-1+\lvert b-a\rvert/r &\text{if }\delta<\lvert b-a\rvert\leq r\\
        0 &\text{otherwise.}
        \end{cases}
    \end{displaymath}
    (Note the restriction on $a,b$ and $r$ ensures that the periodicity of the white noise can be ignored.) This covariance structure is unchanged when all parameters are scaled by a common factor (less than one), more precisely: for $\beta\in(0,1]$
    \begin{displaymath}
        \mathrm{Cov}\big[V^r_\delta(a),V^r_\delta(b)\big]=\mathrm{Cov}\big[V^{\beta r}_{\beta\delta}(\beta a),V^{\beta r}_{\beta\delta}(\beta b)\big].
    \end{displaymath}
    Therefore we conclude that the two families of random variables
    \begin{align*}
        \int_{\beta I}\beta^{-1}\exp\Big(\gamma V^{\beta r}_{\beta\delta}(x)-\frac{\gamma^2}{2}\Var[V^{\beta r}_{\beta\delta}(0)]\Big)\;dx\quad\text{and}\quad\int_{I}\exp\Big(\gamma V^{r}_{\delta}(x)-\frac{\gamma^2}{2}\Var[V^{r}_{\delta}(0)]\Big)\;dx
    \end{align*}
    indexed by $\delta\in(0,r]$ and Borel sets $I\subset[-1/4,1/4]$ are equal in distribution. From the definition of $\nu_t$ in \eqref{e:MeasureDef2} we see that for any $t\geq 1$
    \begin{equation}\label{e:ScaleInvar}
        \rho^{-t}\nu_t(\rho^t\cdot)\overset{d}{=}\rho^{-1}\nu_1(\rho\cdot)=:\nu_\star(\cdot).
    \end{equation}
    (Note we view $\nu_\star$ as a measure on $[-2,2]$ since $\rho\cdot[-2,2]\subset[-1/4,1/4]$ so the necessary restriction on $a,b$ above holds.)

    In \cite[Appendix~A]{ajks} it is shown that there exists a version of $\nu_\star=\nu_\star^{(\gamma)}$ such that for any Borel set $I$, $\gamma\mapsto\nu_\star^{(\gamma)}(I)$ is analytic. In particular, as $\gamma\downarrow0$
    \begin{equation}\label{e:MeasureConv}
        \nu_\star^{(\gamma)}(I)\to\nu_\star^{(0)}(I)=\lvert I\rvert.
    \end{equation}
    Turning to the first part of the event \eqref{e:ShapeRed2}, by monotonicity of $\nu_t$ we have the sandwich bound
    \begin{displaymath}
        \rho^{-t}\nu_t(B_t\setminus B_{t+7/8})\leq \rho^{-t}\nu_t(B_t\setminus B_{\lfloor t\rfloor+7/8})\leq \rho^{-t}\nu_t(B_t).
    \end{displaymath}
    By \eqref{e:ScaleInvar}, the left and right expressions here are equal in distribution to
    \begin{displaymath}
        \nu_\star([-1,1]\setminus[-\rho^{7/8},\rho^{7/8}])\quad\text{and}\quad\nu_\star([-1,1])
    \end{displaymath}
    respectively. By \eqref{e:MeasureConv}, as $\gamma\downarrow 0$ these converge in probability to $2(1-\rho^{7/8})$ and $2$ respectively. Hence by taking $\gamma$ sufficiently small, we can ensure that
    \begin{displaymath}
        \P\left(\rho^{-t}\nu_t(B_t\setminus B_{\lfloor t\rfloor+7/8})\in[3/2,5/2]\right)
    \end{displaymath}
    is as close to one as desired.

    A very similar sandwiching approach yields the same conclusion for the next three parts of \eqref{e:ShapeRed2}. The essence of the argument in each case is one of the following deterministic inequalities:
    \begin{align*}
        \frac{\lvert J_{t,\ell}\setminus B_{\lfloor t\rfloor+7/8}\rvert}{\lvert B_t\setminus B_{\lfloor t\rfloor+7/8}\rvert}&\geq\frac{\frac{1}{4}\rho^t-2\rho^{t+1/4}}{2\rho^t}=\frac{1}{8}-\rho^{1/4}\\
        \frac{\lvert 2B_{t+1/4}\setminus B_{\lfloor t\rfloor+7/8}\rvert}{\lvert B_t\setminus B_{\lfloor t\rfloor+7/8}\rvert}&\leq\frac{4\rho^{t+1/4}}{2\rho^t-2\rho^{t+1/4}}\leq \frac{2\rho^{1/4}}{1-\rho^{1/4}}\\
        \frac{\lvert 2B_t\setminus B_{\lfloor t\rfloor+7/8}\rvert}{\lvert B_t\setminus B_{\lfloor t\rfloor+7/8}\rvert}&\leq\frac{4\rho^t}{2\rho^t-2\rho^{t+1/4}}\leq\frac{2}{1-\rho^{1/4}}
    \end{align*}
    all of which satisfy the necessary bounds in \eqref{e:ShapeRed2} for $\rho$ sufficiently small.

    The final part of \eqref{e:ShapeRed2} requires slightly more work since it involves an infinite sum. Recalling the definitions
    \begin{displaymath}
        I_{m,\ell,t}=[(\ell-2)\rho^t2^{-m},(\ell+2)\rho^t2^{-m}]\quad\text{and}\quad S(m,r,R)=\{\ell\in\Z\;|\;r\leq \lvert \ell\rvert R2^{-m}\leq R\},
    \end{displaymath}
    by \eqref{e:ScaleInvar}
    \begin{align*}
        \sum_{m=0}^\infty\sum_{\ell\in S(m,\rho^{t+1/4},\rho^{t})}\frac{\nu_t(I_{m,\ell,t})^2}{\nu_t(B_t\setminus B_{\lfloor t\rfloor+7/8})^2}&\leq \sum_{m=0}^\infty\sum_{\lvert\ell\rvert\leq 2^m}\frac{\nu_t(I_{m,\ell,t})^2}{\nu_t(B_t\setminus B_{t+1/4})^2}\\
        &\overset{d}{=}\sum_{m=0}^\infty\sum_{\lvert\ell\rvert\leq 2^m}\frac{\nu_\star(I_{m,\ell,0})^2}{\nu_\star(B_0\setminus B_{1/4})^2}.
    \end{align*}
    By \eqref{e:MeasureRatio}, for $\gamma\leq\gamma_0$ sufficiently small, the expectation of the latter quantity is at most
    \begin{displaymath}
        \sum_{m=0}^\infty 2^{m+1}\cdot C_{\gamma_0}\big(2^{-m+2}\big)^{3/2}<\infty.
    \end{displaymath}
    Therefore by the Markov inequality we can choose $M\in\N$ (independent of $t,\gamma,\rho$) large enough to ensure that
    \begin{displaymath}
        \P\left(\sum_{m=M}^\infty\sum_{\lvert\ell\rvert\leq 2^m}\frac{\nu_\star(I_{m,\ell,0})^2}{\nu_\star(B_0\setminus B_{1/4})^2}>\delta\right)
    \end{displaymath}
    is arbitrarily small. For $\gamma$ sufficiently small, by \eqref{e:MeasureConv} applied to finitely many intervals we know that with probability sufficiently close to one
    \begin{align*}
        \sum_{m=0}^M\sum_{\lvert\ell\rvert\leq 2^m}\frac{\nu_\star(I_{m,\ell,0})^2}{\nu_\star(B_0\setminus B_{1/4})^2}\leq \sum_{m=0}^M\sum_{\lvert\ell\rvert\leq 2^m}\frac{\lvert I_{m,\ell,0}\rvert^2}{\lvert B_0\setminus B_{1/4}\rvert^2}+\delta&\leq\sum_{m=0}^\infty 2^{m+1}\frac{2^{-2(m-2)}}{(2(1-\rho^{1/4}))^2}+\delta\\
        &\leq \frac{16}{(1-\rho^{1/4})^2}+\delta.
    \end{align*}
    Combined with the previous displayed equations, this proves that the final event of \eqref{e:ShapeRed2} occurs with probability arbitrarily close to one for $\gamma$ sufficiently small, completing the proof of the lemma.
\end{proof}

\begin{proof}[Proof of Proposition~\ref{p:LDRed}]
We denote
\begin{displaymath}
    \chi_1(n)=\ind_{\mathrm{ShapeRed}^{(1)}(t_n)},\quad\chi_2(n)= \ind_{\mathrm{ShapeRed}^{(2)}(s_n)},\quad\chi_3(n)=\ind_{\mathrm{SizeRed}(t_n,s_n)}.
\end{displaymath}
By Lemma~\ref{l:Occupationtime} there exists $\delta>0$ such that for all $\gamma$ sufficiently small
\begin{displaymath}
\P\left(T_{14\delta N}>3N\right)<C\rho^{(2+3\epsilon)5N}.
\end{displaymath}
From Algorithm~\ref{Algo} we see that $t_{14\delta N},s_{14\delta N}\leq N+2(T_{14\delta N}-N)$ and so with probability at least $1-C\rho^{(2+3\epsilon)5N}$
\begin{equation}\label{e:StoppingTime}
    \sum_{n\;:\;N\leq t_n,s_n\leq 5N}\chi_1(n)\chi_2(n)\chi_3(n)\geq\sum_{n=1}^{14\delta N}\chi_1(n)\chi_2(n)\chi_3(n).
\end{equation}
By Lemma~\ref{l:LDShapeRed}, for $N\geq N_0$ the sequence $\chi_1(n)$ dominates a Bernoulli sequence with parameter $1-\kappa$. Therefore by the standard large deviation estimate for Bernoulli variables
\begin{displaymath}
    \P\left(\sum_{n=1}^{14\delta N}\chi_1(n)< (1-\eta)14\delta N\right)< e^{-14\delta NI(\eta,\kappa)}
\end{displaymath}
where
\begin{displaymath}
    I(\eta,\kappa)=\eta\log\left(\frac{\eta}{1-\kappa}\right)+(1-\eta)\log\left(\frac{1-\eta}{\kappa}\right).
\end{displaymath}
Choosing first $\eta$ and then $\kappa$ sufficiently small we may ensure that
\begin{displaymath}
    \P\left(\sum_{n=1}^{14\delta N}\chi_1(n)< 13\delta N\right)<\rho^{(2+3\epsilon)5N}.
\end{displaymath}
Applying Lemma~\ref{l:LDShapeRed} to $\chi_2(n)$ and Lemma~\ref{l:LDSizeRed} to $\chi_3(n)$, by the union bound we have
\begin{displaymath}
    \P\left(\sum_{n=1}^{14\delta N}\chi_1(n)\chi_2(n)\chi_3(n)< 11\delta N\right)< 3\rho^{(2+3\epsilon)5N}.
\end{displaymath}
Combined with \eqref{e:StoppingTime} this proves the statement of the proposition.
\end{proof}

\section{Generalisations}\label{s:generalisations}
In this section we generalise our main result in two ways: (i) we allow the homeomorphisms to have different values of the parameter $\gamma$ and (ii) we prove a simultaneous welding result for rotated versions of the homeomorphisms $\phi_1$ and $\phi_2$.
\subsection{Measures with different parameter values}\label{s:DiffParam}
Let $\tau^{(1)}$ and $\tau^{(2)}$ be defined as in \eqref{e:MeasureDef} with (possibly) different parameters $\gamma_1$ and $\gamma_2$ respectively and let $\phi_1$ and $\phi_2$ be defined by \eqref{e:DefineHomeo}.
\begin{theorem}\label{t:WeldingGeneral}
There exists $\gamma_0\in(0,\sqrt{2})$ such that for each $\gamma_1,\gamma_2\in[0,\gamma_0]$ the following holds with probability one: there exist conformal maps
\begin{displaymath}
f_1:\D\to D,\quad\text{and}\quad f_2:\C\setminus\overline{\D}\to \C\backslash \overline{D}
\end{displaymath}
(where $D\subset\C$ is some simply connected domain) which may be extended to homeomorphisms of their closures such that $f_1\circ\phi_1^{-1}=f_2\circ\phi_2^{-1}$. Moreover the maps $f_1$ and $f_2$ are unique up to post-composition with a M\"obius transformation.
\end{theorem}
The proof of this result is nearly identical to that of Theorem~\ref{t:Welding} and so we will simply highlight the points of difference.
\begin{proof}[Proof of Theorem~\ref{t:WeldingGeneral}]
    All of the arguments given in Sections~\ref{s:Beltrami} and~\ref{s:Holder} remain valid for $\tau^{(1)}$ and $\tau^{(2)}$ defined using $\gamma_1$ and $\gamma_2$ respectively.
    
    The arguments in Section~\ref{s:Decompose} are equally valid in this setting, provided that we redefine our processes using $\gamma_1$ and $\gamma_2$. More precisely, we define $E^{(j)}(u,a)$, $\tau_t^{(j)}$ and $\nu_t^{(j)}$ as before but with $\gamma_j$ replacing $\gamma$, and we now define
    \begin{align*}
    X^{(H)}_{t,s}&:=\gamma_1 H^{(1)}_{\rho^t}(x)-\gamma_2 H^{(2)}_{\rho^s}(y)-\frac{\gamma_1^2}{2}\Var\big[H^{(1)}_{\rho^t}(x)\big]+\frac{\gamma_2^2}{2}\Var\big[H^{(2)}_{\rho^s}(y)\big]-(t-s)\log(1/\rho)\\
    X^{(V)}_{t,s}&:=\gamma_1 V^{(1)}_{\rho^t}(x)-\gamma_2 V^{(2)}_{\rho^s}(y)-\Big(\frac{\gamma_1^2}{2}+1\Big)t\log(1/\rho)+\Big(\frac{\gamma_2^2}{2}+1\Big)s\log(1/\rho).
    \end{align*}
    We define the events in Section~\ref{s:Decompose} as before (in terms of these redefined measures/processes) and then all proofs go through once more. (In particular, we note that the probabilistic estimates are valid since they only deal with processes defined using either $\gamma_1$ or $\gamma_2$ separately and both of these parameters are assumed less than $\gamma_0$.)
    
    The proof given in Section~\ref{ss:CompletingHolder} holds in our new setting. The arguments from Sections~\ref{ss:algorithm} and~\ref{ss:LargeDeviation} will extend after minor alterations, which we now describe. The increments of $X^{(V)}$ are now given by
    \begin{align}
    X^{(V)}_{t+u,s+v}-X^{(V)}_{t,s}&\sim\mathcal{N}(d_2v-d_1u,\sigma_1^2u+\sigma_2^2v))
    \end{align}
    where
    \begin{equation}
    d_i:=\left(1+\frac{\gamma_i^2}{2}\right)\log(1/\rho)\quad\text{and}\quad\sigma_i^2:=\gamma_i^2\log(1/\rho).
    \end{equation}
    Without loss of generality, we assume $d_1\leq d_2$ and define $d:=2d_1-d_2$. We then let $i_m$, $j_m$, $T_n$, $t_n$ and $s_n$ be defined as before (via Algorithm~\ref{Algo} and the subsequent paragraph) with the new definition of $d$. In particular we note that, assuming $\gamma_2\leq 1/2$ say, for any $Y_{T_n}\in[-d,d]$ we may choose $v,u\in[1,1+5/8]\cup\{2\}$ such that $Y_{T_n}+d_2v-d_1u=0$.

    Our oscillating random walk $Y_n$ now has different (expected) step-sizes $d_1$ or $d_2$ depending on whether it is above or below the target interval $[-d,d]$ and the initial distribution $Y_N$ has an expectation comparable to $N(d_2-d_1)$. However by choosing $\gamma_0$ sufficiently small we can ensure that the ratios $d_2/d_1$ and $d_2/d$ are arbitrarily close to one, which in turn means that Lemmas~\ref{l:Initial}-\ref{l:Occupationtime} will hold with the new definition of $d$ and $\sigma^2:=\gamma_0^2\log(1/\rho)$ after small modifications to their proofs. Specifically, for Lemma~\ref{l:Initial} we can use the union bound to control $\lvert\E[X_{N,N}^{(H)}]\rvert$ and $\lvert X_{N,N}^{(H)}-\E[X_{N,N}^{(H)}]\rvert$ separately; the former is bounded by, say, $dN/32$ whilst the probability of the latter exceeding $d N/32$ can be bound as before. The proof of Lemma~\ref{l:Jump} requires almost no change; we simply note that $\sigma_1^2\leq\sigma^2$ and the expectation of $Y_n$ conditional on $Y_{n-1}=u$ is $u-d_1$ rather than $u-d$. We also note that the same proof yields a `reflected' version of this lemma: if $S:=\inf\{m\geq N\;|\;Y_m\geq d\}$ then $\P(Y_S\geq a\;|\;Y_N<-d)\leq Ce^{-a^2/(2\sigma^2)}$ (previously this version was immediate from symmetry of the oscillating random walk about zero). The proofs of Lemmas~\ref{l:Burn-in} and~\ref{l:Occupationtime} go through exactly as before (with the necessary substitutions).

    All of the arguments in Section~\ref{ss:Reduced} remain valid as before (with our new definitions) completing the proof of H\"older continuity in this setting and hence proving Theorem~\ref{t:WeldingGeneral}.
\end{proof}

\subsection{Simultaneous welding}
Recall that our welding result (Theorem~\ref{t:Welding}) involves transforming $\overline{\D}$ and $\C\setminus\D$ (conformally) so that the `normalised quantum length' of any part of their common boundary with respect to $\tau^{(1)}$ and $\tau^{(2)}$ coincide. This property does not uniquely define a homeomorphism of $\partial\D$; we also need to normalise by choosing a pair of points on the respective boundaries whose images should coincide. In our formulation, these points depend on $\Theta_1$ and $\Theta_2$. We now show the stronger result that, with probability one there is a unique solution to the welding problem for all possible choices of such points simultaneously and that these solutions vary continuously with the choice of points.

For $u\in\T\simeq[0,1)$ we let $\phi_1^{(u)}$ and $\phi_2^{(u)}$ denote the homeomorphisms constructed earlier (in \eqref{e:DefineHomeo}) with $\Theta_1=u$ and $\Theta_2=0$. (There will be no loss of generality in setting $\Theta_2=0$.)

If $f_1:\D\to D$ and $f_2:\C\setminus\overline{\D}\to\C\setminus\overline{D}$ solve the conformal welding problem for a particular homeomorphism then we define the \emph{welding curve} for this solution to be $\partial D=f_1(\partial\D)=f_2(\partial\D)$. If the solution to the welding problem is unique (up to post-composition by a M\"obius transformation) then we may equip the welding curve with a topology induced by uniform convergence on compacts of $f_2$ after normalising so that $f_2(z)=z+O(1/z)$ as $z\to\infty$. (Note that this is stronger than the topology of uniform convergence on $\partial\D$ modulo reparameterisation.)

The main result of this subsection is the following:

\begin{theorem}\label{t:translations}
    There exists $\gamma_0\in(0,\sqrt{2})$ such that for $\gamma\in[0,\gamma_0]$ with probability one, the following holds:
    \begin{enumerate}
        \item for all $u\in\T$ the welding problem for $(\phi_1^{(u)})^{-1}\circ\phi_2^{(u)}$ admits a solution which is unique up to post-composition by a M\"obius transformation, and
        \item the welding curves induce by the above solutions are continuous with respect to $u\in\T$.
    \end{enumerate}
\end{theorem}

With the aid of Figure~\ref{fig:WeldingIllustration}, we can think of this result heuristically in the following manner: as we vary $u$, all points on one side of the welding curve will move along the curve in the same direction (but at different rates) and the curve itself will change to ensure that both sides continue to `match up'. With this visualisation in mind, one might describe this result colourfully as a `conformal earthquake welding'.

The proof of this result follows from a few modifications to our earlier arguments. To emphasise the dependence on $u$, we now write $\mu^{(u)}$ for the complex dilatation defined in \eqref{e:Distortion} and $F_n^{(u)}$ for the solution to the Beltrami equation \eqref{e:Beltrami} satisfying $F_n^{(u)}(z)=z+O(1/z)$ as $z\to\infty$. The key estimate is that the $F_n^{(u)}$ are H\"older continuous uniformly in both $n$ and $u$:

\begin{proposition}\label{p:translation_Holder}
    There exists $\gamma_0\in(0,\sqrt{2})$ such that for each $\gamma\in[0,\gamma_0]$, the collection of maps $\{F_n^{(u)}\;|\;n\in\N, u\in\T\}$ is uniformly H\"older continuous on $\partial\D$ with probability one.
\end{proposition}

\begin{proof}
    We recall the events $\mathrm{Ann}^\prime$, $\mathrm{AnnSeq}^\prime$ and $\mathrm{Match}$ defined prior to Lemma~\ref{l:Good_annuli} and in Definition~\ref{d:AnnSeq}. Observe that only the last of these depends on $u$ (the others are determined by $\tau^{(j)}$ for $j=1,2$) so we denote this by $\mathrm{Match}_u$.
    
    In the proof of Theorem~\ref{t:Holder} it was shown that there exists $C>0$ such that for any $x,y\in\T$ and $N\in\N$
    \begin{displaymath}
        \P(\mathrm{AnnSeq}^\prime(x,y,N)^c)\leq C\rho^{(2+3\epsilon)5N}.
    \end{displaymath}
    In the proof of Proposition~\ref{p:Stationary_condition}, this bound was shown to imply that with probability one there exists a random $N_0$ such that for all $N\geq N_0$ and all $x,y\in P_N$ the event $\mathrm{AnnSeq}^\prime(x,y,N)$ holds where $P_N$ is a $(1/2)\rho^{(1+\epsilon)5N}$-net of $[0,1]$. Now for any $u,v\in\T$ and $N\geq N_0$ we can find $x_1,x_2\in P_N$ such that
    \begin{displaymath}
        v\in\Psi_j^{(u)}([x_j,x_j+\rho^{(1+\epsilon)5N}))\qquad\text{for }j=1,2.
    \end{displaymath}
    This ensures that $\mathrm{Match}_u(x_1,x_2,N)$ occurs. Hence Lemma~\ref{l:Good_annuli} implies the existence of a suitable sequence of annuli surrounding $v$ which will allow us to apply Lemma~\ref{l:Annuli-Holder} following the same argument as in the conclusion of the proof of Proposition~\ref{p:Stationary_condition}.
\end{proof}

\begin{proof}[Proof of Theorem~\ref{t:translations}]
    By Lemma~\ref{l:Distortion}
    \begin{equation}\label{e:translation1}
        \sup_{u\in\T}\sup_{n\in\N}K(\cdot,F_n^{(u)})\in L^\infty_{\mathrm{loc}}(\C\setminus\partial\D).
    \end{equation}
    Next we observe that $K(\cdot,F_n^{(u)})$ is formed by pre-composing $K(\cdot,F_n^{(0)})$ with a rotation on $\D$ (and the identity on $\C\setminus\D$). Hence applying Lemma~\ref{l:Distortion2} with $\Theta_1=0$ implies that, with probability one,
    \begin{equation}\label{e:translation2}
        \{K(\cdot, F_n^{(u)})\;|\;n\in\N,u\in\T\}\qquad\text{is uniformly integrable on compact sets}.
    \end{equation}
    We now fix a realisation of our underlying white noise such that \eqref{e:translation1} and \eqref{e:translation2} hold and also, by Proposition~\ref{p:translation_Holder}, the $F_n^{(u)}$ are uniformly H\"older continuous.

    On this realisation, for any $u\in\T$ we may argue as in the proof of Theorem~\ref{t:Welding} to conclude that there exists $F^{(u)}\in W^{1,1}_\mathrm{loc}(\C)$ which satisfies the Beltrami equation for $\mu^{(u)}$ with the standard normalisation at infinity such that $F_n^{(u)}\to F^{(u)}$ locally uniformly. Moreover the welding problem for $(\phi_1^{(u)})^{-1}\circ\phi_2^{(u)}$ is solved by $f_1^{(u)}:=F^{(u)}\circ\Phi_1^{(u)}$ and $f_2^{(u)}:=F^{(u)}\circ\Phi_2^{(u)}$ and this solution is unique up to M\"obius transformation.

    It remains to prove continuity of the solutions in $u\in\T$. Given a sequence $u_n\to u$ as $n\to\infty$ we may choose a sequence $k_n\to\infty$ such that for all $R>1$
    \begin{equation}\label{e:local_unif}
        \lim_{n\to\infty}\sup_{z\in R\D}\lvert F_{k_n}^{(u_n)}(z)-F^{(u_n)}(z)\rvert=0.
    \end{equation}
    Defining $G_n=F_{k_n}^{(u_n)}$ it follows that all of the conditions of Proposition~\ref{p:dilatation_convergence} are satisfied with $\mu_\infty=\mu^{(u)}$ and so there exists a homeomorphism $G_\infty\in W^{1,1}_\mathrm{loc}(\C)$ which satisfies the Beltrami equation for $\mu_\infty=\mu^{(u)}$ such that $G_n\to G_\infty$ locally uniformly. The uniqueness of the solution to the welding problem for this choice of $u$ implies that $G_\infty=F^{(u)}$. Combined with \eqref{e:local_unif}, this shows that $F^{(u_n)}\to F^{(u)}$ uniformly on $R\D$. Since $f_2^{(u_n)}=F^{(u_n)}\circ\Phi_2^{(u_n)}$ and $\Phi_2^{(u)}(z)$ is jointly continuous in $u$ and $z$, we deduce the stated continuity of the welding curves.    
\end{proof}

%\bibliographystyle{plain}
%\bibliography{references.bib}
\printbibliography

\end{document}